\documentclass[10pt]{article}

\usepackage{authblk}
\usepackage{fullpage}
\usepackage[tight]{minitoc}
\usepackage{setspace}
 \usepackage[titles]{tocloft}

\usepackage{latexsym,amssymb,amsmath,amsthm,graphicx,bm,calligra}
\usepackage{thmtools}
\usepackage{thm-restate}
\usepackage{mathtools}

\usepackage[usenames, dvipsnames]{color}
\usepackage{tikz}
\usetikzlibrary{matrix}
\usetikzlibrary{pgfplots.groupplots}

\definecolor{greenA}{rgb}{0,0.5451,0.2706}

\usepackage{pgfplots}

\usepackage{authblk}
\usepackage{latexsym,amssymb,amsmath,amsthm,graphicx,float,bm}
\usepackage{enumerate}

\newtheorem{definition}{Definition}

\newtheorem{theorem}{Theorem}
\newtheorem{lemma}{Lemma}

\newtheorem{corollary}[theorem]{Corollary}
\newtheorem{proposition}{Proposition}

\newtheorem{example}{Example}
\def\bs{\ensuremath\boldsymbol}

\def\bs{\ensuremath\boldsymbol}

\def\z{\ensuremath{{\bf z}}}
\def\x{\ensuremath{{\bf x}}}
\def\Z{\ensuremath{{\bf Z}}}

\def\tr{\ensuremath\text{Tr}}
\def\mata{\ensuremath\text{Mat}}
\def\sign{\ensuremath\text{sign}}

\def\bs{\ensuremath\boldsymbol}

\def\bs{\ensuremath\boldsymbol}

\def\z{\ensuremath{{\bf z}}}

\numberwithin{equation}{section}

\begin{document}
\title{Stable rank one matrix completion is solved by two rounds of semidefinite programming relaxation.}
\author[1]{Augustin Cosse}
\affil[1]{
D\'epartement de Math\'ematiques et Applications,  \protect\\ Ecole Normale Sup\'erieure, Ulm, Paris. \protect\\ PSL Research University.}
\author[2]{Laurent Demanet}
\affil[2]{Department of Mathematics, Massachusetts Institute of Technology, MA}
\renewcommand\Authands{ and }

\maketitle

%


\begin{abstract}
This paper studies the problem of deterministic rank-one matrix completion. It is known that the simplest semidefinite programming relaxation, involving minimization of the nuclear norm, does not in general return the solution for this problem. In this paper, we show that in every instance where the problem has a unique solution, one can provably recover the original matrix through two rounds of semidefinite programming relaxation with minimization of the trace norm. We further show that the solution of the proposed semidefinite program is Lipschitz-stable with respect to perturbations of the observed entries, unlike more basic algorithms such as nonlinear propagation or ridge regression. Our proof is based on recursively building a certificate of optimality corresponding to a dual sum-of-squares (SOS) polynomial. This SOS polynomial is built from the polynomial ideal generated by the completion constraints and the monomials provided by the minimization of the trace.  The proposed relaxation fits in the framework of the Lasserre hierarchy, albeit with the key addition of the trace objective function. We further show how to represent and manipulate the moment tensor in favorable complexity by means of a hierarchical low-rank decomposition.

%
%

\end{abstract}
\hspace{1cm}\begin{minipage}{14.5cm}\date{\textbf{Acknowledgement.} Both authors were supported by a grant from the MISTI MIT-Belgium seed fund. AC was supported by the FNRS, FSMP, BAEF and Francqui Foundations. LD is supported by AFOSR grant FA9550-17-1-0316, ONR grant N00014-16-1-2122, and NSF grant DMS-1255203. AC is grateful to MIT Math, Harvard IACS, the University of Chicago as well as NYU Courant Institute and Center for Data Science for hosting him during this work.}\end{minipage}


\section{\label{SectionIntroduction}Introduction}

Low rank matrix completion has been studied extensively throughout the last few years, among other reasons because of its practical interest in machine learning and data science. Completion provides a useful tool to compress and manipulate large databases such as in genomics and finance, and to infer information from a few measurements such as in collaborative filtering or triangulation. Good introductions as well as recovery results for random designs and arbitrary ranks can be found in~\cite{candes2009exact, keshavan2010matrix}.

The objective of this paper is to provide an algorithm that solves the rank one case in a stable and comprehensive way. Let $\mathcal{M}(1;m\times n)$ denote the set of rank-$1$ matrices of size $m\times n$; we consider the problem of recovering an unknown rank one matrix $\bs X_0\in \mathcal{M}(1;m\times n),\; \bs X_0 = \bs x_0\bs y_0^T$ when we are given $\mathcal{O}(m+n)$ entries from this matrix, possibly corrupted by an additive noise $\varepsilon$. We do not make any assumption on the noise. In the noiseless case, this problem reads
\begin{align}\begin{split}
\text{find} \quad & \bs X\in \mathbb{R}^{m\times n}\\
\text{subject to} \quad & \text{rank}(\bs X) = 1\\
&  \bs X_{ij} = (\bs X_0)_{ij} \quad (i,j)\in \Omega,\end{split}%
\label{matrixcompletion}%
\end{align}%
where $\Omega$ denotes the set of measurements. As a slight abuse, we will also speak of constraints $\{\bs X_{ij} - (\bs X_0)_{ij} = 0\}_{(i,j)\in \Omega}$ as belonging to the set $\Omega$. In the noisy case, the data fit constraint is relaxed to $\| \bs X_{ij} - ((\bs X_0)_{ij} + \varepsilon_{ij}) \| \leq \sigma$ in a standard fashion.

Clearly, one cannot always solve problem~\eqref{matrixcompletion}. For example, if no information is known on a given column (resp. row), it becomes impossible to recover the entries corresponding to this column (resp. row). Another limitation occurs when the rank-1 matrix has a zero entry; then the corresponding row or column will be zero, and the completion problem will generically lack injectivity. As an illustration of the issue with zero entries, consider the problem where the first row and last column are known and are both trivial. The number of measurements is $(m+n-1)$. However in this case, any matrix $\bs X\in \mathbb{R}^{m\times n}$ of the form $\bs X = \bs v\bs  w^*$ with $v_1 = w_n = 0$ is a valid solution of the problem. For this reason, we consider the completion problem on $\mathcal{M}^*(1,m\times n)$, where $\mathcal{M}^*(1,m\times n)$ denotes the restriction of $\mathcal{M}(1;m\times n)$ to matrices for which none of the entries are zero. 

To formalize the notion of injectivity, we introduce the mapping $\mathcal{R}_{\Omega}:\mathbb{R}^{m\times n}\rightarrow \mathbb{R}^{|\Omega|}$ that corresponds to extracting the observed entries of the matrix. We let $\mathcal{R}_\Omega^1$ denote the restriction of $\mathcal{R}_\Omega$ to matrices of rank-$1$ that have no zero rows/columns. Invertibility of this restriction $\mathcal{R}_\Omega^1$ corresponds to asking whether one can uniquely recover the matrix $\bs X$ from the knowledge of $\mathcal{R}_\Omega(\bs X)$ and the fact that $\bs X$ has rank $1$. Let us denote by $\mathcal{V}_1$, $\mathcal{V}_2$ the sets of row and column indices of $X$. We consider the bipartite graph $\mathcal{G}(\mathcal{V}_1,\mathcal{V}_2,\mathcal{E})$ associated to problem~\eqref{matrixcompletion}, where the set of edges is defined by $(i,j)\in \mathcal{E}$ iff $(i,j)\in \Omega$. The vertices of the bipartite graph $\mathcal{G}$ corresponding to $\bs X$ are labeled by the corresponding row and column indices. The conditions for the recovery of the matrix $\bs X$ from the set $\Omega$ are related to the properties of this bipartite graph as expressed by the following lemma which can be found, for example, in~\cite{kiraly2012combinatorial}:

\begin{lemma}[\label{necessaryandsufficient}Rank-1 completion]
The mask $\mathcal{R}_\Omega$ is injective on $\mathcal{M}^*(1;m\times n)$ if and only if $\mathcal{G}$ is connected.
\end{lemma}

Lemma~\eqref{necessaryandsufficient} has an interesting consequence. Within the noiseless framework, rank one matrix completion can be solved exactly through a nonlinear propagation approach. To understand this, let us write $\bs X = \bs x\bs y^T \in \mathbb{R}^{m\times n}$ with $x_1=1$. Let us further use $\z\in \mathbb{R}^{m+n-1}$ to denote the concatenation of  $[x_2,\ldots, x_m]\in \mathbb{R}^{m-1}$ and $\bs y\in\mathbb{R}^n$, $\z = (\bs x,\,\bs y)$. When we deal with the rank one case, an implication of lemma~\ref{necessaryandsufficient} is that for all $x_n$, $y_m$, the bipartite graph corresponding to the mask $\Omega$ always contains at least one connected path starting with an edge corresponding to an element of the first row and for which the series of existing edges corresponds to running through $X$ according to chains of constraints such as
\begin{align}
y_{i_1} \rightarrow y_{i_1} x_{i_2} \rightarrow x_{i_2} y_{i_3}\rightarrow \ldots  y_{i_{L+1}} x_n \qquad \mbox{(to reach $x_n$)}
\label{chain1a} \\
y_{i'_1} \rightarrow y_{i'_1} x_{i'_2} \rightarrow x_{i'_2} y_{i'_3}\rightarrow \ldots  x_{i'_{L'+1}} y_m \qquad \mbox{(to reach $y_m$)}
\label{chain2a} 
\end{align}
More generally, the two chains~\eqref{chain1a} and~\eqref{chain2a} can read, using the vector $\z\in \mathbb{R}^{m+n-1}$,
\begin{align}
z_{i_1} \rightarrow z_{i_1} z_{i_2} \rightarrow z_{i_2} z_{i_3}\rightarrow \ldots  z_{i_{L+1}} z_n. 
\label{generalchain} 
\end{align}
In other words, each of the entries of $\bs x$ and $\bs y$ can always be related to an element of the first row whose value is known because of the normalization $x_1=1$. Each of the elements making up the bilinear constraints can then be obtained in the absence of noise by iteratively propagating the value of the elements of the first row through~\eqref{generalchain}. As we explain in the sequel, such a propagation scheme however lacks robustness to noise, especially, when the magnitude of the entries is on the order of the magnitude of the noise.

When the measurements are corrupted by noise, a popular approach is to turn to minimization of the nuclear norm as a proxy for the rank (see~\cite{fazel02ma, recht10gu} for early references). However, the nuclear norm does not always guarantee recovery of the rank one matrix $\bs X_0$ when the noise vanishes. An important gap regarding rank-one matrix completion has thus been the lack of an algorithm providing a proper (deterministic) stability estimate of the form
\begin{align}
\|\bs X - \bs X_0\|\leq \omega(\|\mathcal{R}_\Omega(\bs X) - \mathcal{R}_\Omega( \bs X_0)\|).
\end{align}
for some Lipschitz function $\omega(\tau)$ obeying $\omega(\tau) \rightarrow 0$ when $\tau\rightarrow 0$.

We can now state the main contributions of the paper.
\begin{itemize}
\item First, we show that rank-one matrix completion can be solved through two rounds of semidefinite programming relaxation. Our result is sharp in terms of measurements; recovery is always possible as soon as the nonlinear problem has a unique solution. This is in contrast to previous results that required a random measurement set~\cite{candes2010matrix}. This result also confirms that there exist instances of rank minimization problems that are solved in a comprehensive way (without constraints of incoherence and/or randomness) using a fixed, higher ($>1$) number of rounds of semidefinite programming relaxation.

\item Second, we show that when the measurements are corrupted by noise, the solution to the semidefinite relaxation remains proportional to the noise level. In particular, this solution is shown to be Lipschitz-stable with respect to the noise level. This is in contrast with nonlinear approaches such as~\cite{kiraly2012combinatorial,kiraly2015algebraic}. 

\item Finally, our proof system, based on constructing a dual polynomial, incidentally reveals two important facts: First, minimization of the trace norm helps certify recovery because it provides additional squares of monomials that are useful in constructing the dual polynomial. Second, recovery can be related to the possibility of propagating known information through the graph by means of polynomial equations.
\end{itemize}

The next sections discuss the limitations of propagation, minimization of the nuclear norm, and ridge regression. We illustrate these limitations on the simple problem of completing the rank-one matrix $\bs X_0$ with a small parameter $\delta$,
\begin{align}
\bs X_0 = \left(\begin{array}{cc}
1 & ? \\
\delta & 1 
\end{array}\right),\label{matrixEpsilon}
\end{align}
for which, given the rank one constraint, the only missing entry is obviously given by $1/\delta$.

\subsection{\label{propagationIsUnstable}Propagation is unstable}

We start by discussing the simple propagation scheme. In the noiseless framework, this scheme can be efficiently applied by writing $\bs X_0$ as
\begin{align}
\bs X_0 = \left(\begin{array}{ccc}
X_{11} & X_{12}\\
X_{21} & X_{22} 
\end{array}\right),
\end{align}
and simply deriving $X_{12}$ as $X_{12} = \frac{X_{22}X_{11}}{X_{21}}$. Now assume that the entries are corrupted by a noise $\bs\epsilon$ so that the measurements are now given by $\tilde{X}_{11} = (X_0)_{11}+ \varepsilon_{11}$, $\tilde{X}_{21}  = (X_0)_{21} + \varepsilon_{21}$ and $\tilde{X}_{22} = (X_0)_{22} + \varepsilon_{22}$ with $\varepsilon = (\varepsilon_{11}, \varepsilon_{21}, \varepsilon_{22})\in \mathbb{R}^3$. Taking a noise $\varepsilon$ with $\|\varepsilon \|$ on the order of $\delta$, such as for example $\varepsilon_{21} = -.9\delta$, will result in important errors when using propagation as shown below,
\begin{align}
\tilde{X}_{12}  = \frac{\tilde{X}_{22}\tilde{X}_{11}}{\tilde{X}_{21}} = \frac{1+\mathcal{O}(\delta)}{\delta+ \varepsilon_{21}}\sim 10\frac{1}{\delta}\qquad \frac{\tilde{X}_{12} - (X_{0})_{12}}{(X_0)_{12}} = 900\%
\end{align}
The estimate derived through the propagation algorithm are thus unreliable when the entries are corrupted by an unknown noise $\varepsilon$ of magnitude comparable to the smallest entries in the matrix. In addition, there is no effective, general method to select the propagation path optimally.

Another elementary method consists in taking the logarithm of the constraints, and solving the resulting system to obtain the logarithm of the unknowns. It is a very reasonable method for some positive matrices, although it is easy to see that it suffers from a similar kind of instability as the propagation scheme.

\subsection{Nuclear norm fails}

In this section, we briefly study how nuclear norm minimization would perform in the same framewok of problem~\eqref{matrixEpsilon} as before. Nuclear norm minimization was first formalized in~\cite{fazel02ma, recht10gu} and guarantees were given, for the matrix completion problem, in a probabilitic framework, in~\cite{candes10ma}. Nuclear norm minimization relies on solving the convex program
\begin{subequations}
\label{nnminimization1}
\begin{align}
\mbox{minimize} &\quad  \|\bs X\|_*\\
\mbox{subject to} &\quad \bs X_{ij} = (\bs X_0)_{ij},\quad (i,j)\in \Omega.
\end{align}
\end{subequations}
Where $\| \bs X\|_*$ is used as a proxy for the rank. In the case of~\eqref{matrixEpsilon}, for a sufficiently small $\delta$, if we let 
\begin{align}
\bm X = \left(\begin{array}{ccc}
1 & 2/\delta\\
\delta & 1 \end{array}\right)\label{deltaMat}
\end{align}
it can be easily verified that $\|\bs X\|_*<\|\bs X_0\|_*$, $\mbox{rank}(\bs X)= 2$, and the nuclear norm minimization~\eqref{nnminimization1} thus doesn't return the unknown matrix $\bs X_0$ despite the fact that a sufficient amount of measurements are provided. 

In the case of a symmetric, positive semidefinite matrix $\bs X$, program~\eqref{nnminimization1} becomes minimization of the trace under the constraint $\bs X \succeq 0$. When the diagonal is fully measured, it is known that this formulation succeeds at recovering the original matrix $\bs X_0$ when the completion problem is well-posed \cite{demanetjugnon,demanetjugnon2}.

\subsection{\label{RidgeRegression}Ridge regression has local minima}

For the sake of completeness, we briefly discuss ridge regression (a.k.a Tikhonov regularization) on the rank one factorization. This approach has gained in popularity over the last years and is equivalent to solving the quartic regularized problem. In fact it is natural to wonder whether the semidefinite programming formulation of this paper which relies on the minimization of the trace norm is not simply a form of ridge regression. This section precisely refutes this idea. The ridge regression problem reads
\begin{align}
\min & \|\mathcal{R}_{\Omega}(\bs x\bs y^T) - \mathcal{R}_\Omega(\bm X_0 + \varepsilon)\|_F^2 + \lambda(\|\bs x\|^2+ \|\bs y\|^2 )\label{quartic}
\end{align}
\begin{figure}[h!]
\hspace{-.6cm}
 \begin{minipage}[b]{0.3\textwidth}\centering
   \includegraphics[width=1.2\textwidth, trim=0cm .5cm 8.3cm 0cm,clip=true]{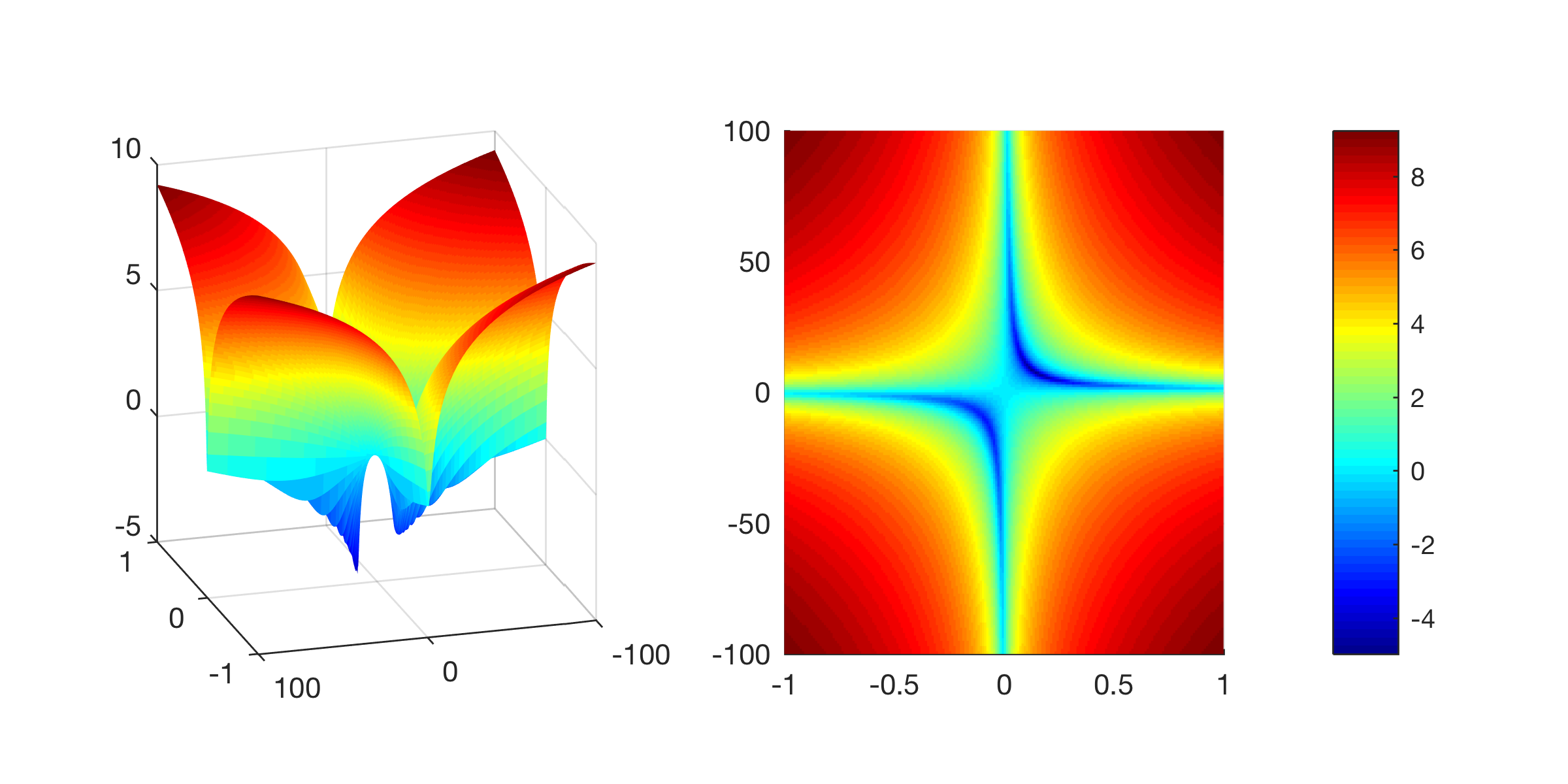}
  \end{minipage}\hspace{.8cm}
 \begin{minipage}[b]{0.3\textwidth}\centering
    \includegraphics[width=1.2\textwidth, trim=0cm .5cm 0cm 0cm,clip=true]{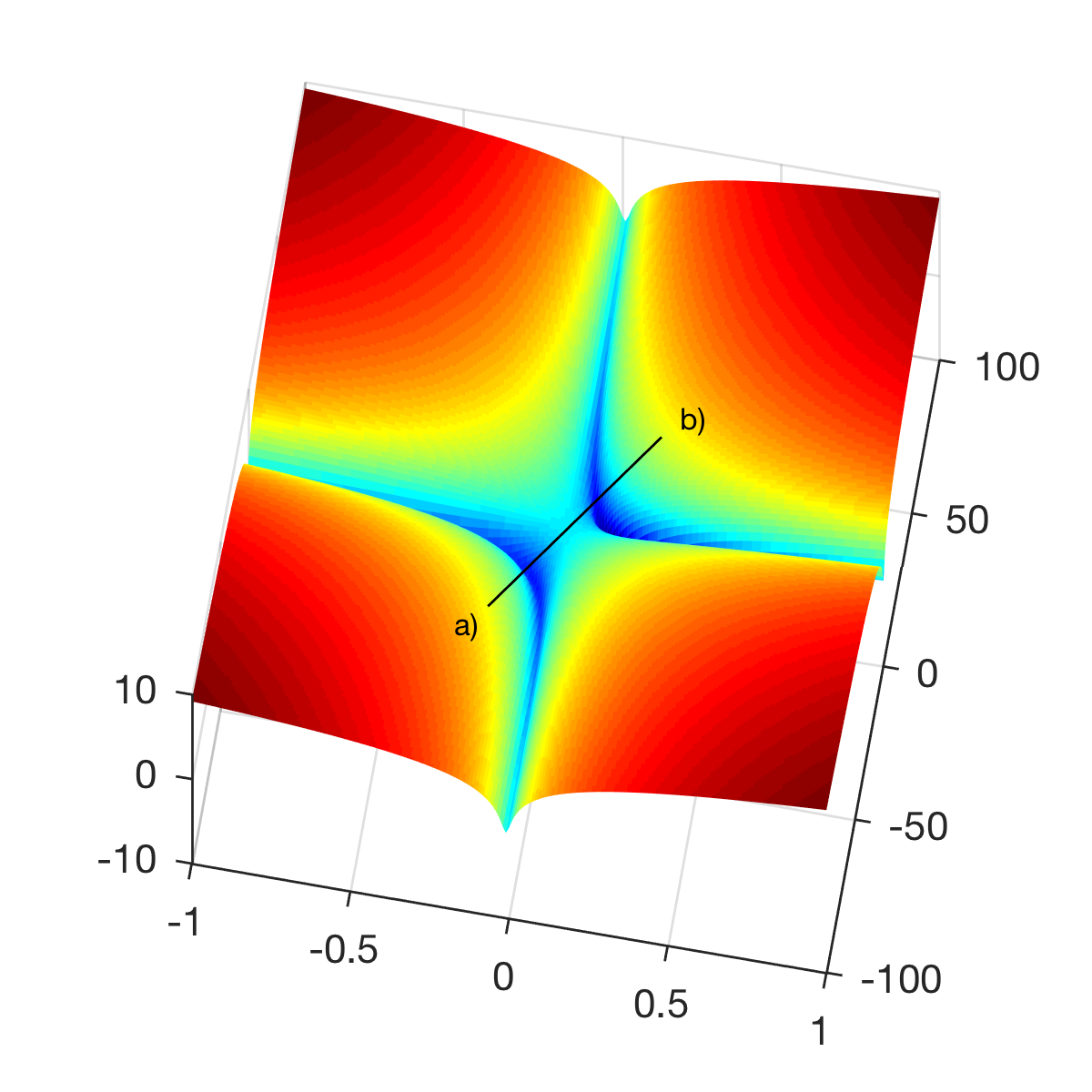}
  \end{minipage}\hspace{1.2cm}
  \begin{minipage}[b]{0.3\textwidth}\centering
%
%
\definecolor{mycolor1}{rgb}{0.00000,0.44700,0.74100}%
\begin{tikzpicture}

\begin{axis}[%
width=1.329in,
height=1.954in,
at={(1.011in,0.642in)},
scale only axis,
xmin=0,
xmax=1,
ymode=log,
ymin=0.01,
ymax=10,
yminorticks=true,
axis background/.style={fill=white}
]
\addplot [color=mycolor1, forget plot, line width=1.5pt]
  table[row sep=crcr]{%
0	9.130004\\
0.01	8.0884041209\\
0.02	7.1294516036\\
0.03	6.2493750881\\
0.04	5.4444809744\\
0.05	4.7111534225\\
0.06	4.0458543524\\
0.07	3.4451234441\\
0.08	2.9055781376\\
0.09	2.4239136329\\
0.1	1.99690289\\
0.11	1.6213966289\\
0.12	1.2943233296\\
0.13	1.0126892321\\
0.14	0.7735783364\\
0.15	0.5741524025\\
0.16	0.4116509504\\
0.17	0.2833912601\\
0.18	0.1867683716\\
0.19	0.1192550849\\
0.2	0.07840196\\
0.21	0.0618373169\\
0.22	0.0672672356\\
0.23	0.0924755561\\
0.24	0.1353238784\\
0.25	0.1937515625\\
0.26	0.2657757284\\
0.27	0.3494912561\\
0.28	0.4430707856\\
0.29	0.5447647169\\
0.3	0.65290121\\
0.31	0.7658861849\\
0.32	0.8822033216\\
0.33	1.0004140601\\
0.34	1.1191576004\\
0.35	1.2371509025\\
0.36	1.3531886864\\
0.37	1.4661434321\\
0.38	1.5749653796\\
0.39	1.6786825289\\
0.4	1.77640064\\
0.41	1.8673032329\\
0.42	1.9506515876\\
0.43	2.0257847441\\
0.44	2.0921195024\\
0.45	2.1491504225\\
0.46	2.1964498244\\
0.47	2.2336677881\\
0.48	2.2605321536\\
0.49	2.2768485209\\
0.5	2.28250025\\
0.51	2.2774484609\\
0.52	2.2617320336\\
0.53	2.2354676081\\
0.54	2.1988495844\\
0.55	2.1521501225\\
0.56	2.0957191424\\
0.57	2.0299843241\\
0.58	1.9554511076\\
0.59	1.8727026929\\
0.6	1.78240004\\
0.61	1.6852818689\\
0.62	1.5821646596\\
0.63	1.4739426521\\
0.64	1.3615878464\\
0.65	1.2461500025\\
0.66	1.1287566404\\
0.67	1.0106130401\\
0.68	0.893002241600001\\
0.69	0.777285044900001\\
0.7	0.664900010000001\\
0.71	0.557363456900001\\
0.72	0.4562694656\\
0.73	0.3632898761\\
0.74	0.2801742884\\
0.75	0.208750062500001\\
0.76	0.1509223184\\
0.77	0.1086739361\\
0.78	0.0840655556000001\\
0.79	0.0792355769\\
0.8	0.0964001599999999\\
0.81	0.1378532249\\
0.82	0.2059664516\\
0.83	0.303189280099999\\
0.84	0.432048910399999\\
0.85	0.5951503025\\
0.86	0.7951761764\\
0.87	1.0348870121\\
0.88	1.3171210496\\
0.89	1.6447942889\\
0.9	2.02090049\\
0.91	2.4485111729\\
0.92	2.9307756176\\
0.93	3.47092086409999\\
0.94	4.07225171239999\\
0.95	4.7381507225\\
0.96	5.47207821439999\\
0.97	6.27757226809999\\
0.98	7.15824872359999\\
0.99	8.11780118089999\\
1	9.16000099999999\\
};
\end{axis}
\end{tikzpicture}%
  \end{minipage}
\caption{\label{IllPosed1} Representation of the ridge regression energy landscape for a rank-one matrix completion problem with $m=n=1$. We consider the minimization problem $f(x,y) = \|xy - 1\|^2 +\|y-\delta\|^2$, plus regularization terms such as in~\eqref{quartic}. Without regularization, there is one infimum at $(-\infty, 0)$, and one minimum at $(\delta,1/\delta)$ (In the figure above, we take $\delta = .1$). With regularization the infimum located at $(-\infty, 0)$ becomes a minimum and moves closer to the origin such as shown above. The evolution of the loss function on the line joining the two minima is displayed by the figure on the right.}
\end{figure}
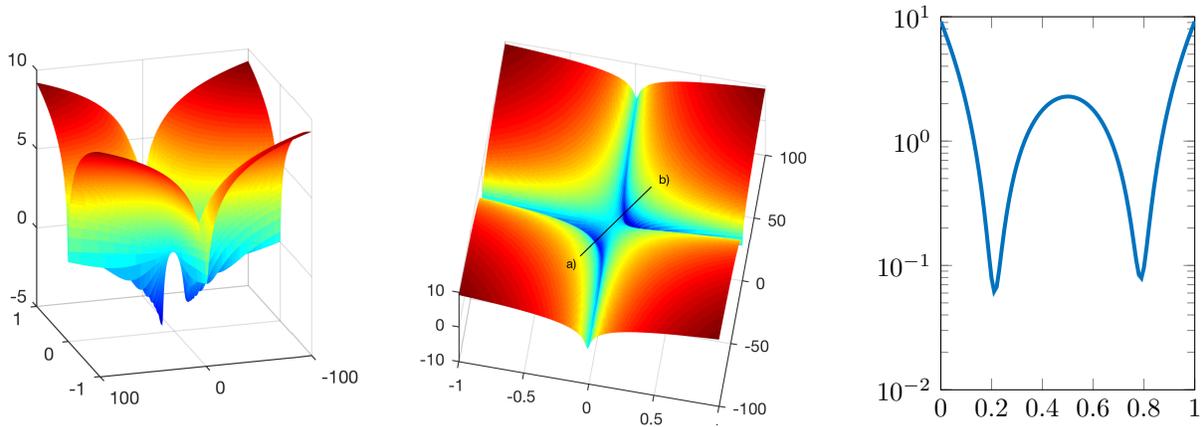
%
%
%
%
The most popular way to solve problem~\eqref{quartic} is through gradient descent. However, when several measurements are given and no convergence guarantee is known, it is not clear how to initialize the algorithm. We choose to follow standard practice and initialize it with the singular value decomposition of the matrix $\sum_{ij}(\mathcal{P}_{\Omega}(\bs X))_{ij}e_ie_j^T$ and taking the outerproduct of the corresponding top singular vectors weighted by their singular value. As shown by Fig.~\ref{IllPosed1}, the landscape underlying this formulation suffers from a lack of convexity. As a consequence, even in the absence of noise when the matrix size is sufficiently large, and the number of measurements is sufficiently close to the recoverability limit, ridge regression will fail to return the global minimizer.

\subsection[Algorithm: two rounds of SDP relaxation]{\label{introductionToAlgo}Algorithm: two rounds of semidefinite relaxation}

When minimizing the nuclear norm of rank one matrices, one only enforces constraints on monomials of degree at most two on the entries of the generating vectors ${\bs x}$ and ${\bs y}$. The nuclear norm was shown in~\cite{fazel02ma} to be equivalent to the following semidefinite program,
\begin{align}\begin{split}
\mbox{minimize}& \quad \mbox{Tr}(\bs W)\\
\mbox{subject to}& \quad \bs W = \left(\begin{array}{cc}
\bs X_{11} & \bs X_{12}\\
 \bs X_{21} & \bs X_{22}
\end{array}\right)\\
& \quad (\bs X_{12})_{ij} = (X_0)_{ij},\quad (i,j)\in \Omega,\\
&\quad \bs W\succeq 0.
\end{split}\label{nnminimization}
\end{align}
When $\bs X_0 = \bs x_0\bs y_0^T$, the matrix $\bs W$ is a proxy for the rank one matrix 
\begin{align}
\bs W_0 = \left(\begin{array}{cc}
\bs x_0 \bs x_0^T &  \bs x_0 \bs y_0^T\\
\bs y_0 \bs x_0^T & \bs y_0 \bs y_0^T
\end{array}\right)
\end{align} 
The positive semidefinite constraint on $\bs X$ is thus used in combination with the trace norm, as a convex relaxation of the rank one constraint. It is interesting to note that formulation~\eqref{nnminimization} only optimizes over monomials of bidegree $(1,1)$. The key idea of the "second round of lifting" is to extend this type of formulation to monomials of higher degree in the original unknowns. Introducing $\z_0 = (\bs x_0,\bs y_0)$ and $\z_0^{(2)}=\mbox{vec}(\z_0^{\otimes^2}) = \mbox{vec}(\z_0 \otimes \z_0)$, we consider the larger matrix $\bs M_0$ defined as 
\begin{align}
\bs M_0 = \left(\begin{array}{ccc}
1 & \z_0 & \z_0^{(2)} \\
\z_0 & \z_0^{\otimes^2} &  \z_0\otimes\z_0^{(2)}\\
\z_0^{(2)} & \z_0^{(2)}\otimes \z_0 & \z_0^{(2)}\otimes \z_0^{(2)}
\end{array}\right),\label{secondRound}
\end{align}
In~\eqref{secondRound}, we thus have $\bs W_0 = \z_0^{\otimes^2}$. At order two, the semidefinite relaxation considers as unknowns all the entries of a positive semi-definite proxy $\bs M$ of the same structure as $\bs M_0$, but without the explicit link to a vector $\z_0$. Instead, two categories of linear constraints are intended to force the matrix $\bs M$ to inherit the structure of $\bs M_0$:
\begin{itemize}
\item \emph{Structural constraints/ total symmetry}. Due to the additional monomials that appear in~\eqref{secondRound}, there now exist corresponding additional relations between the entries of $\bs M_0$ (and thus $\bs M$ as well). In particular, all monomials in $\z_0^{(2)}$ find an exact match in the block $\z_0^{\otimes^2}$. Within the block $\z_0^{(2)}\otimes \z_0^{(2)}$, one must also list all the total symmetry constraints of a tensor of order 4. More generally, the structural constraints enforce equality of the entries that are identical in the rank-one matrix $\bs M_0$. 

\item \emph{Higher-order affine constraints}. Similarly, for any of the original affine constraints applying on the elements of $\z_0^{\otimes^2}$, one can now define higher order constraints that are jointly enforced on the elements of $\z_0$ and the elements of the block $\z_0^{(2)}\otimes \z_0$. As an example, consider that one is given the constraint $X_{ij} = x_i y_j = (X_0)_{ij}$. It is now possible to enforce the constraints $x_i y_i x_k - (X_0)_{ij} x_k$ for any monomial $x_k$. More generally, the higher-order constraints are obtained by multiplying the original constraints by any product of the entries of $x$ and $y$ of degree at most two. 
\end{itemize}

The second-order formulation in this paper consists in combining all these constraints with $\bs M \succeq 0$, and with minimization of $\mbox{Tr}(\bs M)$. The point of this paper is to prove that this formulation, of order $4$ in $\mathbf{x}$ and $\mathbf{y}$, is enough to recover every rank-one matrix in the completion problem, and to provide a scalable algorithm for it.

This idea is not new and can be found through various formulations in the work of  Parrilo~\cite{parrilo,parrilo2003semidefinite}, Shor~\cite{shor1988approach, shor1987class, shor1987quadratic}, Nesterov~\cite{nesterov2000squared}, and Lasserre~\cite{lasserre}. For now, we simply write this semidefinite programming relaxation in the following general form. More details on the constraints will be provided in section~\ref{sec:dualcert}.
\begin{align}\begin{split}
\mbox{minimize} &\quad  \mbox{Tr}(\bs M)\\
\mbox{subject to} &\quad \mathcal{A}(\bs M) =  \bs b,\\
&\quad  \bs M\succeq 0. 
\end{split}\label{stableSDPnoiseless}%
\end{align}
At this point, we just note that the linear map $\mathcal{A}$ now encodes the original constraints from~\eqref{matrixcompletion} together with the additional structural and higher order constraints mentioned above. 

To define the corresponding stable formulation for the semidefinite relaxation~\eqref{stableSDPnoiseless}, we first introduce a decomposition of the linear map into the structural part $\mathcal{A}_S$ and the remaining part $\tilde{\mathcal{A}}$. The motivation behind such a decomposition comes from the fact that structural constraints, unlike the original measurement constraints and their higher order extensions, are not affected by noise. The higher-order constraint in the Lasserrre hierarchy use (noisy) data in an essential manner in their expression, not just in a right-hand-side, hence we use the notation $\tilde{\mathcal{A}}$ as a shorthand for those constraints. Following this decomposition, the stable version of the relaxation~\eqref{stableSDPnoiseless} can be posed generally as,
\begin{align}\begin{split}
\mbox{minimize} &\quad  \mbox{Tr}(\bs M)\\
\mbox{subject to} &\quad \mathcal{A}_S(\bs M) =  0\\
& \quad \|\tilde{\mathcal{A}}(\bs M) - \tilde{\bs b}\|\leq \sigma\\
&\quad  \bs M\succeq 0. \end{split}\label{stableSDPnoisy}%
\end{align}
The vector $\tilde{\bs b}$ encodes the noisy measurements $(X_0)_{ij}+ \varepsilon_{ij}$ and the RHS corresponding to their higher order extensions, $(i,j)\in \Omega$. The next section introduces the main result of this paper.

\subsection{\label{sec:mainResult}Main Result}

The main result of this paper only requires the necessary and sufficient conditions of lemma~\ref{necessaryandsufficient} to be satisfied. Our first theorem states that the noiseless semidefinite program~\eqref{stableSDPnoiseless}, for a linear map $\mathcal{A}$ encoding both the original constraints as well as the additional higher order and structural constraints, solves the rank one completion problem exactly under minimization of the Trace norm.

\begin{theorem}\label{maintheorem}  
Consider problem~\eqref{matrixcompletion} in the context of lemma~\ref{necessaryandsufficient}, with $\bs X\in \mathcal{M}^*(1;m\times n)$. Then this problem can be solved exactly through two rounds of semidefinite programming relaxation with minimization of the trace norm.
\end{theorem}

The interest of semidefinite programs lies in their robustness vis-a-vis corruption of the data. This is what Theorem~\ref{noisycase} below makes precise. It shows that when considering observations that are corrupted by a noise $\varepsilon$, so that $(\tilde{X}_0)_{ij} = (X_0)_{ij}+ \varepsilon_{ij}$, the solution to the semidefinite programming relaxation~\eqref{stableSDPnoisy} remains within the noise level.

\begin{restatable}{theorem}{theoremCompletion}
\label{noisycase} 
Let $\bm M_0$ denote the rank one matrix introduced in~\eqref{secondRound} for $\z_0\in \mathbb{R}^{m+n-1}$. Assume that the necessary and sufficient conditions of lemma~\ref{necessaryandsufficient} are satisfied. Let $\|\bs \varepsilon\|_2 = \sqrt{\sum_{(i,j)\in \Omega}\varepsilon^2_{ij}}$.
Let $\bm M$ denote the solution to the semidefinite program~\eqref{stableSDPnoisy}. This solution satisfies  
\begin{align}
\frac{\|\bm M - \bm M_0\|_F}{\|\bm M_0\|_F} \leq C_0(m+n)^{7/2}\|\bs \varepsilon\|_2.\label{noisyCompletionEquation}
\end{align}
The constant $C_0$ depends on the entries of $\bs X_0$, but not on $\Omega$, $m,n,$ or $\varepsilon$. 


\end{restatable}


Once $\bs M$ is found, one can read off $\bs X$ from the entries of $\bs M$ corresponding to $|\alpha| = 1$. Note that for the propagation and log-system algorithms, a similar error bound can only be expected to hold provided $\| \bs \varepsilon \|_2 \leq c \min_{i,j} | (X_0)_{ij} |$ for some $c < 1$, and would otherwise become unbounded.

Most of the $(m+n)^{7/2}$ multiplicative factor in Theorem~\ref{noisycase} arises because of the propagation of noise though the certificate (i.e. the fact that the certificate relies on a chain of length $m+n$). As we don't make prior assumption regarding propagation of information through the graph, the noise along the path is constrained by a global bound of the form $\|\varepsilon\|\sqrt{(1 + \|\bs m_0\|_1)}$, instead of a bound in $\mathcal{O}(\|\varepsilon\|)$ which could be enforced if the path was explicitely given. This global rather than path specific bound results in additional $\mathcal{O}(\|\varepsilon\| (m+n))$ multiplicative factors when a path specific bound, would lead to better $\mathcal{O}(\|\varepsilon\|)$. The remaining factor affecting the constant is the size of the moments matrix. In this regard, a second possible simplification is to restrict the set of second-order moments to moments that appear in the constraints only. In this case, the first column of the moments matrix~\eqref{secondRound} reduces to $(1,\;\z_0,\;(x_0)_i(y_0)_j)$ for $(i,j)\in \Omega$ and the semidefinite program~\eqref{stableSDPnoisy} becomes as scalable as the traditional SDP relaxation, or the first round of the Lasserre hierarchy since the matrix of unknowns is now on the order of $(m+n)^2$

Those ideas are summarized through Corollary~\ref{noisycaseCorollary} below.

%

\begin{restatable}{corollary}{corollaryCompletion}
\label{noisycaseCorollary}
Assume that the paths in the bipartite graph relating each of the unknown vertices $x_i$ (resp. $y_j$) to the root node $y_\ell$ are explicitly given. Then we have the following stability estimate,
\begin{align}
\frac{\|\bm M - \bm M_0\|_F}{\|\bm M_0\|_F} \leq C_0 (m+n)^{2}\|\bs \varepsilon\|_2.\label{noisyCompletionEquationWPaths}
\end{align}
The constant $C_0$ depends on the entries of $\bs X_0$, but not on $\Omega$, $m,n,$ or $\varepsilon$.

\end{restatable}

Before introducing the proofs of Theorems~\ref{maintheorem} and~\ref{noisycase} as well as Corollary~\ref{noisycaseCorollary}, and describing how to make numerical schemes scalable, the next section lists the most relevant connections of this work with the developments of the last few years. 

\subsection{\label{sec:connections}Connections with existing work}

Low rank matrix completion and semidefinite programming relaxations have both attracted a lot of attention from various communities over the past few years. Low-rank matrix completion is a problem that has been extensively studied in the litterature and has led to numerous successful approaches. One of the most famous, nuclear norm minimization, led to the derivation of important recovery guarantees~\cite{candes2015phase, candes2010matrix,candes2009exact, candes2010power,  recht2011simpler}. In~\cite{candes2010matrix} in particular, the authors derive (probabilistic) bounds on the recovery error for low rank matrices, when the measurements are corrupted by noise. Other notable progress on this question includes the results of Keshavan et al.~\cite{keshavan2010matrix} who certify recovery with high probability (w.h.p.) given $\mathcal{O}((m+n)\log(m+n))$ measurements and some incoherence conditions in a noiseless framework. In~\cite{keshavan2009matrix,keshavan2008learning}, the same authors derive a probabilistic bound that can be used in the presence of noise and improve the results obtained in~\cite{keshavan2009matrix} to a recovery w.h.p. that scales linearly in the noise as $\mathcal{O}(m+n)$ provided that both the magnitude of the entries as well as the number of measurements are sufficiently large. The noise is assumed to be i.i.d random with zero mean and sub-gaussian tail. 

Singer et al.~\cite{singer2010uniqueness} investigate matrix completion with a non random sampling mask based on the structure of the measurements. Their paper is interested in determining whether completion is possible or not in the general rank-$r$ case using rigidity theory. Other papers that focus on characterizing the sampling patterns enabling matrix completion include~\cite{kiraly2013error,kiraly2015algebraic} where the authors extend the idea which was studied for rank one matrices in~\cite{kiraly2012combinatorial} to the more general case of low rank matrices. Those papers show that feasibility and uniqueness of the completion only depends on the structure of the measurements. They propose an algorithm based on completion of the $k$-by-$k$ minors (circuit polynomials) to complete the matrix in the noiseless case. This algorithm lacks robustness to noise and requires an additional step averaging the values over different paths across the bipartite graph to cancel out the noise. The idea therefore cannot be applied in the case of $\mathcal{O}(m+n)$ measurements. Neither can it be used with deterministic noise. The question of completability patterns is also discussed in the more recent~\cite{pimentel2015characterization} by Pimentel et al. In~\cite{pimenteladaptive}, the same authors study reconstruction under an i.i.d. zero mean noise with covariance matrix $\sigma^2 \bs I$. This last paper is based on the  concept of recommender systems where only a given number of measurements ($r+1$) are allowed per column. They show asymptotic convergence of their estimator for sufficiently large matrices, when a sufficiently large number of columns are used to generate the measurements.

Since the pioneering work of Goemans and Williamson~\cite{goemans1995improved}, which started popularizing the use of semidefinite programs as an approximation to hard quadratic optimization problems, semidefinite programming has gained a reputation as a potentially powerful candidate to derive interesting approximations to hard/nonconvex problems. This activity culminated in the now famous Unique Games Conjecture~\cite{khot2002power} in complexity theory. Examples of successful developments based on semidefinite programming or nuclear norm relaxations of nonlinear problems can be found in~\cite{candes2009exact,voroninski1,demanetjugnon, candes2014towards, bandeira2015convex}. In~\cite{demanetjugnon,demanetjugnon2}, one of the authors solves the symmetric rank one matrix completion problem when the diagonal entries are given. The proof system in this paper relies on spectral graph theory, and use the fact that the eigenvector of the exact solution $\bs X_0$ is also an eigenvector of the data weighted graph Laplacian, to derive a bound on the recovery. Finally, the noiseless result of this paper was presented in the introductory note~\cite{cosserank}.

As stated earlier, the semidefinite program~\eqref{secondRound} of this paper is in fact an instance of the more general Lasserre and sum-of-squares (SOS) hierarchies of semidefinite programs~\cite{laurent} which were introduced through the work of Parrilo~\cite{parrilo,parrilo2003semidefinite}, Shor~\cite{shor1988approach, shor1987class, shor1987quadratic}, Nesterov~\cite{nesterov2000squared}, and Lasserre~\cite{lasserre} as an extension of the basic semidefinite programming relaxation. Those hierarchies are based on making semidefinite programming relaxations gradually tighter by adding more variables and constraints, resulting in optimization on gradually larger subspaces.

Semidefinite programming hierarchies have received a lot of attention over the last few years, both positively as a potential extension to the traditional semidefinite programming relaxations, and negatively because of their practical intractability resulting form the gradually higher dimension of their successive rounds. Another important drawback associated to those hierarchies has been the lack of convincing instances for which rounds higher than one were leading to noticeable improvements. For the most recent developments on the convergence of those hierarchies, see for example the papers by Barak et al.~\cite{barak2012hypercontractivity,barak2014sum,barak2015tensor}.

A few improvements have however been made over the last few years. On the first point, in a paper which is very related to this one~\cite{tang2015guaranteed}, Tang et al. show that the tensor decomposition problem can be solved through a semidefinite programming relaxation with minimal number of rounds. In~\cite{barak2015tensor}, Barak et al. certify using the Rademacher complexity, that tensor completion can be solved with high probability with $\widetilde{\mathcal{O}}(n^{3/2})$ measurements through $6$ rounds of semidefinite programming relaxation. Finally, other results along that line recently appeared in papers by Nie et al. In a first paper tackling assymptotic convergence to the minimum~\cite{nie2013exact} for general polynomial problems, these authors introduce an updated formulation based on the Jacobian of the polynomial constraints for which convergence of the hierarchy at a sufficiently large order is certified. In a second monograph~\cite{cui2014all} which is discussed further in the last section of this paper, the same author shows that computation of the real eigenvalues of symmetric tensors can be achieved through a finite number of semidefinite programming relaxation rounds. No upper bound is provided on the number of rounds required.

Convergence of the sum-of-squares and Lasserre hierarchies are also discussed by Gouveia et al. in~\cite{gouveia2010theta}. This paper relates the sequence of theta bodies of an ideal $\mathcal{I}$ and the Lasserre hierarchy and shows that under some assumptions, the $k^{th}$ theta body of an ideal is equal to the set of solutions resulting from the $k^{th}$ round of the Lasserre hierarchy and that for real radical ideal the $k^{th}$ theta body corresponds to the closure of the convex hull of the variety of the ideal $\mathcal{I}$ as soon as every polynomial of degree one that is non negative on the variety can be represented by sum of squares of degree at most $k$ modulo the ideal.

A few papers adress simplifications of higher rounds of semidefinite programming relaxations by means of sparsity of the polynomial constraints. Among those papers, Lasserre~\cite{lasserre2006convergent} as well as Nie et al.~\cite{nie2008sparse} introduce tailored relaxations for problems where sparsity occurs in the constraints and the objective function. This adapted relaxation enables a significant reduction in the size of the matrices whenever a property known as \textit{the running intersection property} is satisfied together with some independence between the sets of variables used by the constraints. Ahmadi~\cite{ahmadi2017improving} also discusses possible reduction in the complexity of semidefinite programming hierarchies by means of the chordal extension of the graph whose cliques are defined from the polynomial constraints.

The Lasserre and sum-of-squares hierarchies are built upon the resolution of systems of polynomial equations. For this reason we also briefly address another important line of work following from computational algebraic geometry. When looking for the solution to a system of polynomial equations (in particular when the underlying ideal is zero-dimensional), the very first question one want to ask is whether it is possible to compute a Gr\"obner basis for this system (see for example~\cite{cox1992ideals} as well as~\cite{sturmfels2002solving}). When such a Gr\"obner basis can be found, the solutions can be computed as the vectors of joint eigenvalues of the companion matrices (see Theorem 2.6 in~\cite{sturmfels2002solving}).

Computing a Gr\"obner basis is at least NP-complete in the general case (see for example~\cite{bardet2005complexity}). In fact, the notion of Gr\"obner basis is somehow complementary to proving the tightness of the SOS/Lasserre hierarchies. Finding one helps understand the other and vice versa. The degree of the Gr\"obner basis is unknown before the computation and bases with higher degree are more difficult to compute. Moreover the numerical computations involved are known to be numerically unstable (see for example~\cite{byrod2009fast, heldt2009approximate} and the discussion therein) and do not scale well with the dimension. The dual polynomial that we build in this paper is in fact equivalent to showing that such a Gr\"obner basis can be constructed (modulo the Trace) with degree at most $4$ from the polynomial ideal generated by the completion constraints.

\subsection{Notations}

Let $\z = (\bs x,\bs y)\in \mathbb{R}^{m+n-1}$. When dealing with algebraic problems like~\eqref{matrixcompletion}, it will be useful to write those problems as general polynomial optimization problems (POP) of the form 
\begin{subequations}
\label{polynomialSyst}
\begin{align}
\z^* = \underset{\z}{\operatorname{argmin}} \quad & p_0(\z),\quad \z\in \mathbb{R}^{n}\\
\mbox{subject to}\quad & p_1(\z)\geq 0,\;p_2(\z)\geq 0,\ldots,p_m(\z)\geq 0.
\end{align}
\end{subequations}
For some polynomials $p_0(\z),p_1(\z),\ldots p_m(\z)\in \mathbb{R}[\z]$ where $\mathbb{R}[\z]$ is used to denote the ring of multivariate polynomials in the optimization variable $\z\in \mathbb{R}^{m+n-1}$. We will sometimes use the compact notation $\{p_j(\z)\}_{j\in [J]}$ to denote the set of polynomial constraints. This set of constraints defines a semialgebraic set $K$ of feasible points, which we write as 
\begin{align}
K = \{\z \in \mathbb{R}^{m+n-1},\;|\;p_1(\z)\geq 0,\ldots,p_m(\z)\geq 0\}\label{semialgebraicSet}
\end{align}
For $\z\in \mathbb{R}^{m+n-1}$ and $\bs\alpha\in\mathbb{N}^{m+n-1}$, we introduce the multi-index notation $\bm z^{\bs \alpha}=z_1^{\alpha_1}z_2^{\alpha_2}\ldots z _n^{\alpha_n}$ with $|\bs \alpha| = \alpha_1 + \alpha_2 + \ldots + \alpha_n$, the degree of the monomial $\z^{\bs \alpha}$. We will use $\z_{\mathcal{B}}$ to denote the sequence of all monomials in $\z\in \mathbb{R}^{m+n-1}$ for some standard ordering (standard monomial basis). Hence $\z_{\mathcal{B}}:=(1,z_1,z_2\;,\ldots, z_1^2, z_1z_2,\ldots,\z^{\bs \alpha},\ldots)$. Similarly, let $\z_{\mathcal{B}}^t$ denote the vector of all monomials $\z^{\bs \alpha}$ from the standard basis with degree bounded by $t$: $|\bs \alpha|\leq d$. 


In this paper, polynomials will be alternatively be denoted through either of the representations below,  
\begin{itemize}
\item Weigthed sums of monomials $p_i(\z) = \sum_{\bs \alpha} (p_i)_{\bs \alpha}\z^{\bs \alpha}$, where the $(p_i)_{\bs \alpha}$ thus denotes the coefficient of the monomial $\z^{\bs \alpha}$ in $p_i(\z)$. 
\item Vectors/sequences of coefficients $p_i = ((p_i)_{\bs \alpha})_{\bs \alpha\in \mathbb{N}^n}$ as $p_i(\z) =  p_i^T\z_{\mathcal{B}}$ 
\item Matrices of coefficients, $P_i$, such that $p_i(\z) = \langle P_i, \z_{\mathcal{B}}\z_{\mathcal{B}}^T\rangle = \z_{\mathcal{B}}^T P_i\z_{\mathcal{B}}$. In this case we will use the notation $\Z_{\mathcal{B}}$ (resp. $\Z_{\mathcal{B}}^d$) to represent the matrix generated from the standard basis as $\Z_{\mathcal{B}}=\z_{\mathcal{B}}\z_{\mathcal{B}}^T$ (resp. $\Z_{\mathcal{B}}^{2d} = \z^d_{\mathcal{B}}(\z^d_{\mathcal{B}})^T$). The moments matrix $\bs M_0$ encountered earlier is simply $\int \Z_{\mathcal{B}}^{2d} d\mu$ when $d = 2$ and for the measure $d\mu = \delta(\z-\z_0) d \z$.

\end{itemize}

For the general set of polynomials $p_1(\z),p_2(\z),\ldots,p_m(\z)$, we let $\mathcal{I}$ denote the ideal generated by those polynomials. This set is defined as from all the combinations that are generated by multiplying the polynomials $p_i(\z)$ by any other polynomials $h_i(\z)\in\mathbb{R}[\z]$.
\begin{align}
\mathcal{I}(p_1,\ldots, p_m):=\left\{\sum_{j\in [m]}p_j(\z)h_j(\z),\quad \mbox{for polynomials $h_j(\z)\in \mathbb{R}[\z]$}\right\}\label{IdealDefinition}
\end{align}
Equivalently, we will use $\mathcal{I}_d$ to denote the \textit{truncated} ideal, whose maximal degree is bounded by $d$, 
\begin{align}
\mathcal{I}_d(p_1,\ldots, p_m):=\left\{\sum_{j\in [m]}p_j(\z)h_j(\z),\quad \mbox{$h_j(\z)\in \mathbb{R}[\z]$, $\mbox{deg}(h_j)\leq d-\mbox{deg}(p_j)$}\right\}\label{IdealTruncatedDefinition}
\end{align}
Given the matrix $\Z^d_{\mathcal{B}} = \z_{\mathcal{B}}^d (\z^d_{\mathcal{B}})^T$ used to represent monomials of degree at most $2d$, when writing polynomials in matrix form, we will need to access monomials of a given degree. As an example, consider the univariate monomial basis $(1,z,z^2,z^3,\ldots)$. The corresponding matrix for the monomial basis truncated at degree $4$ reads
\begin{align}
\Z_{\mathcal{B}}^2 = \left(\begin{array}{ccc}
1 & z & z^2 \\
z & z^2 & z^3 \\
z^2 & z^3 & z^4\\
\end{array}\right)\label{momentsMatrixExample}
\end{align}
Now consider the polynomial $p(z) :=z^2-1$. This polynomial can be applied on the matrix $\Z_{\mathcal{B}}^2$ by introducing appropriate matrices to access the monomials. Those matrices are simply assembled from the product of two canonical basis vectors. That is, for any degree $\gamma$, one access the monomial $z^\gamma$ in $\Z$ by means of the matrices $\bs e_{\alpha}\bs e_\beta^T$ for any $\alpha,\beta$ such that $\gamma = \alpha+ \beta$. For any such matrix, we have 
$$\langle \Z_\mathcal{B}^2, \bs e_{\alpha} \bs e_{\beta}^T\rangle = z^{\gamma}.$$ 
In particular, using those matrices, the polynomial $p(z)$, $z\in \mathbb{R}$ reads 
\begin{align}
p(z) &= \langle \bs P, \Z_\mathcal{B}^2\rangle  = \sum_{\gamma \in \mathbb{N}^{\mbox{\tiny deg}(p)}}\sum_{\alpha+\beta = \gamma} C_{\gamma} p_{\gamma} \langle \Z_\mathcal{B}^2, e_{ \alpha} e_{\beta}^T\rangle.\label{eq}\end{align}
$C_\gamma$ is a normalizing constant defined from each degree $\gamma$ as 
$$C_\gamma = \frac{1}{\#\{(\alpha,\beta)|\;\alpha+\beta=\gamma\}}$$
To write expression~\eqref{eq} compactly, we introduce auxiliary matrices $B_\gamma$ relative to each of the monomials $z^\gamma$, defined as
\begin{align}
\bm B_{\gamma} = \sum_{\alpha +\beta= \gamma}e_\alpha e_\beta^T.
\end{align}
Using those matrices, the polynomial $p(z)$ can now read directly as $p(z) = \sum_{\gamma \in \mathbb{N}^{\mbox{\tiny deg}(p)}} C_\gamma p_\gamma \langle \Z_\mathcal{B}^2, B_\gamma \rangle$. Moreover, the coefficients of this polynomial can be obtained via $p_\gamma = \langle \bs P, \bs B_\gamma \rangle$. The constant $C_\gamma$ can now also be expressed more simply as $C_\gamma = \frac{1}{\|\bm B_\gamma\|_F^2}$.

As an example, the matrix $\bs B_2$ used to access the monomial $z^2$ in~\eqref{momentsMatrixExample} reads 
\begin{align}
\bm B_2 = \left(\begin{array}{ccc}
0 & 0 & 1 \\
0 & 1 & 0 \\
1 & 0 & 0\\
\end{array}\right)
\end{align}
When dealing with polynomials on $\mathbb{R}^{m+n-1}$, the same idea applies and we will denote the corresponding matrices as $\bm B_{\bs \gamma}$ where $\bs \gamma$ is the multi-index used to access the monomial $\z^{\bs \gamma}$.

\section{\label{sec:strategyMC}Proof of Theorem~\ref{maintheorem}.}


To ensure unique recovery of the matrix $\bm M_0$ from the semidefinite program~\eqref{stableSDPnoiseless}, traditional convex optimization proofs are based on satisfying first order optimality conditions\footnote{Note that in the case of convex optimization those conditions are necessary and sufficient.} by exhibiting a dual vector $\lambda$ such that $-\mathcal{A}^*\lambda -  I \in \partial \imath_\mathcal{K}(\bm M_0)$ where $\imath_\mathcal{K}$ denotes the indicator function of the positive semidefinite (PSD) cone (see for example~\cite{candes2009exact}). In section~\ref{sec:dualcert} below, we start by giving a more detailed characterization of problem~\eqref{stableSDPnoiseless} in terms of the constraints. We then provide the general conditions for the existence of such a certificate. In section~\ref{sec:polynomialform}, we show how satisfiability of these conditions can be reduced to the construction of a dual polynomial with particular structure. Section~\ref{sec:proof} finally shows how such a dual polynomial can be constructed.

\subsection{\label{sec:dualcert}Dual certificate}

In this section, we give an explicit expression for the general condition $-\mathcal{A}^*\lambda - \bm I \in \partial \imath_\mathcal{K}(\bm M_0)$ on the dual vector $\lambda$ certifying optimality in the case of problem~\eqref{stableSDPnoiseless}. We then show how this condition can be made tighter to ensure uniqueness in addition to optimality at $\bm M_0$. We start by giving the detailed expression of the constraints in~\eqref{stableSDPnoiseless}. Note that each matrix $\bm B_{\bs \gamma}$ can be decomposed into a sum of elementary matrices $\bm E_{\bs \gamma_1,\bs \gamma_2}$ with only a single non zero entry, for multi-indices $\bs \gamma_1,\;\bs \gamma_2\in \mathbb{N}^{m+n-1}$, i.e., $\bm B_{\bs \gamma} = \sum_{\bs \gamma_1+\bs \gamma_2 = \bs \gamma} \bm E_{\bs \gamma_1,\bs \gamma_2}$ with $\bm E_{\bs \gamma_1,\gamma_2} = \bs e_{\bs \gamma_1} \bs e_{\bs \gamma_2}^T$.
\begin{align}\begin{split}
\text{minimize} \quad & \text{Tr}(\bm M)\\
\text{subject to} \quad & \sum_{\bs \zeta} \frac{(h_\ell)_{\bs \zeta}}{\|\bm B_{\bs \zeta+\bs \kappa}\|^2_F}\langle \bm M, \bm B_{\bs \zeta+\bs \kappa}\rangle =0.\\
& \text{for}\quad \bs \kappa\in \mathbb{N}^K_{2(t- d_{h_\ell} )}, 1\leq \ell \leq L\\
& \bm M\succeq 0,\quad \bm M_{11} = 1,\\
& \langle \bm M, \bm E_{\bs \delta_1,\bs \delta_2} - \bm E_{\bs \gamma_1,\bs \gamma_2}\rangle = 0,\\ 
& \mbox{for all $(\bs \delta_1,\bs \delta_2), (\bs \gamma_1,\bs \gamma_2)\;\mbox{s.t.}\;\bs \delta_1+ \bs \delta_2= \bs \gamma_1+ \bs \gamma_2\leq 2t$}.\end{split}\label{momentProblem}
\end{align}
In the proof of Theorem~\ref{maintheorem}, we will write the last constraint of~\eqref{momentProblem} together with the normalization constraint $\bm M_{11} = 1$ compactly as $\bm M = \sum_\gamma m_\gamma \bm B_\gamma + \bs e_1\bs e_1^T$ by introducing additional variables $m_{\gamma}$. This enables us to get rid of the last structural constraint in~\eqref{momentProblem}. The resulting structure of $\bm M$ is Hankel-type (and would be exactly Hankel in the one-dimensional case as we saw earlier). The first sum in~\eqref{momentProblem} is taken over all the coefficients of each constraint $h_\ell(x) = 0$, $h_\ell(x) = \sum_{\bs \zeta} (h_\ell)_{\bs \zeta}{\bf x}^{\bs \zeta}$. 

We now derive the first order optimality conditions $-\mathcal{A}^*\lambda - \bm I \in \partial \imath_\mathcal{K}(\bm M_0)$ for problem~\eqref{momentProblem} in terms of the Lagrangian dual function $\mathcal{L}$. Introducing multipliers for each of the polynomial constraints, the Lagrangian can be written as
\begin{align}&\mathcal{L}(\bm M, \bm m, \bs \lambda, \bs \xi) = \text{Tr}(\bm M) + \langle \bm M- \sum_{\bs \gamma} m_{\bs \gamma} \bm B_{\bs \gamma} - \bs e_1\bs e_1^T,\bs \xi\rangle\nonumber \\
 & + \sum_\ell\sum_{\bs \kappa\in \mathbb{N}^K_{2(t- d_{h_\ell} )}}\lambda_{\ell, \bs \kappa} \left(\sum_{\bs \zeta} \frac{(h_\ell)_{\bs \zeta}}{\|\bm B_{\bs \zeta+\bs \kappa}\|_F^2}\langle \bm  B_{\bs \zeta+\bs \kappa}, \bm M\rangle\right)\label{Lagrangian0}\\ &+\imath_{\mathcal{K}}(\bm M).\nonumber
\end{align}
The multipliers $\lambda_{\ell,\bs \kappa}$ correspond to each of the original and shifted polynomial constraints while $\bs \xi$ encode the Hankel-type structure of the matrix $\bm M$. Usual convex optimization theory states that $\bm M_0 = \bs m_0\bs m_0^T$ is a minimizer for problem~\eqref{momentProblem} if and only if one can find dual vectors $(\bs \xi, \bs \lambda)$ such that $0\in \partial \mathcal{L}(\bm M_0, \bm m_0, \bs \lambda,\bs \xi)$. The dual variables $\bs \xi,\bs \lambda$ combine into a dual certificate $Z$, and must obey the following three conditions. Let 
\begin{align}
T =\left\{\bs m_0\bs v^T + \bs v\bs m_0^T,\quad \bs v\in \mathbb{R}^{|\mathbb{N}_2^{K}|}\right\},\label{Tspace}
\end{align}
$T^\perp$ being its orthogonal complement, and let $Y_T$ denote the projection of the matrix $Y$ onto the subspace $T$. 

\begin{enumerate}[1)]
\item $\displaystyle Y = \bm I - \bs \xi - \sum_{\ell}\sum_{\bs \kappa\in \mathbb{N}^K_{2(t- d_{h_\ell} )}}\lambda_{\ell\bs \kappa}\left(\sum_{\bs \zeta} \frac{(h_\ell)_{\bs \zeta}}{\|\bm B_{\bs \zeta+\bs \kappa}\|^2_F}\bm B_{\bs \zeta+ \bs \kappa} \right)$\\
\item $Y_T = 0, \quad Y_{T^\perp}\succeq 0$\\
\item $\langle \bm B_{\bs \gamma}, \bs \xi\rangle =0,\qquad \forall \bs\gamma\neq 0.$
\end{enumerate}

Conditions 1 and 2 are obtained by requiring the derivative of this Lagrangian with respect to the moments matrix $\bs M$ belongs to the normal cone (subdifferential of the indicator of the PSD cone) at $\bs M_0$, and Condition 3 is obtained by requiring that the derivative of the Lagrangian with respect to the vector of moments, $m$, vanishes.

The following proposition guarantees unique recovery in addition to the optimality ensured by the satisfiability of conditions $1)$ to $3)$.
\begin{proposition}\label{propositionUniqueness}
To ensure \textit{unique} recovery of $\bm M_0$, in addition to the conditions 1), 2), and 3) mentioned above, it is sufficient to require
$Y_{T^\perp} \succ 0$ as well as injectivity on $T$ of all the linear constraints $\mathcal{A}(\bm M) = \bs b$ arising from the measure version of the polynomial constraints $h_\ell(\bm z) = 0$ as well as from the structure of the moments matrix.
\end{proposition}
\begin{proof} We will now use the decompositions $Y = \bm I - Y_2 = I - Y_1 - \bs \xi$
\begin{align}
\text{Tr}(\bm M_0)&= \langle I,\bm M_0 \rangle = \langle I_T,\bm M_0 \rangle  =\langle (Y_2)_T, \bm M_0\rangle\\
& = \langle Y_2,\bm M_0-\bm M\rangle +  \langle Y_2,\bm  M\rangle = \langle Y_2, \bm M \rangle\label{boundTrace1}\\
& = \langle I_T, \bm M_T\rangle  + \langle (Y_2)_{T^\perp}, \bm M \rangle \\
& = \text{Tr}(\bm M_T)  + \langle (Y_2)_{T^\perp}, \bm M \rangle \\
& < \text{Tr}(\bm M)\qquad\text{for $\bm M_{T^\perp}\neq 0$} 
\end{align}
In~\eqref{boundTrace1}, we use $\langle \bs \xi, \bs M - \bs M_0\rangle = 0$ as well as the fact that $Y_1$ belongs to the range of $\mathcal{A}^*$ and $\mathcal{A}(\bm M) = \mathcal{A}(\bm M_0)$. The last inequality follows from $ (Y_2)_{T^\perp} \prec I_{T^\perp} $ which since $\bm M\succeq 0$ implies $\langle (Y_2)_{T^\perp}, \bm M\rangle < \text{Tr}(\bm M_{T^\perp})$ for $\bm M_{T^\perp}\neq 0$. This last inequality thus implies $\bm M_{T^\perp}=0$. Finally $\bm M_T  = (\bm  M_0)_T$ by injectivity of the constraints on $T$.
\end{proof}
Note that, to satisfy $Y_{T^\perp} \succ 0$ and $Y_T = 0$, it is sufficient to ask for $\bs m_0\in \text{Null}(Y)$ and to require $Y$ to be positive semidefinite and exact rank $|\mathbb{N}_2^{K}|-1$. In the next section, we show how the duality between sum-of-squares polynomials and positive semidefinite matrices can help us construct a dual certificate satisfying those conditions.

\subsection{\label{sec:polynomialform}Sum-of-squares and positive semidefinite matrices}

We call sum-of-squares (SOS) polynomial, any polynomial $p(z)$ for which there exists a decomposition $p(z) = \sum_{j=1}^m s_j^2(z)$ for some polynomials $s_j\in \mathbb{R}[\bm z]$.
Introducing a polynomial version of proposition~\ref{propositionUniqueness} requires the following lemma from~\cite{laurent} relating SOS and semidefinite programming (SDP). For completeness we also provide a proof.

\begin{proposition}[\label{sosissdp}Equivalence between SOS and SDP]Let $\mathbb{N}^K_{2t}$ denoe the set set of $K$-tuples  $\bs \alpha\in \mathbb{N}^K$ such that $\sum_i \alpha_i \leq 2t$ and let $\x^{\bs \alpha} = x_1^{\alpha_1}x_2^{\alpha_2}\ldots x_K^{\alpha_K}$. Let $p(\z)\in \mathbb{R}[\bm z]$ with $\displaystyle p(\z)=\sum_{\bs \alpha\in \mathbb{N}_{2t}^K} p_{\bs \alpha} \x^{\bs \alpha}$ be a polynomial of degree $\leq 2t$, the following assertions are equivalent,
\begin{enumerate}[1)]
\item $p(\z)$ is a sum-of-squares polynomial
\item There exists a positive semidefinite matrix $\bm A$ such that 
\begin{equation}
p(\z) = \z_{\mathcal{B}}^T \bm A\b\z_{\mathcal{B}},
\end{equation}
\end{enumerate}
\end{proposition}
\begin{proof}	
If $p(\x)$ is SOS then $p(\x) = \sum_j s_j^2(\x)$ for some polynomial $s_j(\x)$. Let $d=\max\{\deg(s_j)\}$ denote the maximum degree of the $s_j(\x)$. Further let $\delta = \lceil d/2\rceil $. For each of the $s_j(\x)$, for some ordering of the monomials, construct the corresponding vector of coefficients $\bs s_j\in \mathbb{R}^\delta$ with ${\bm s}_j(\x)=s_j^T\x_{\mathcal{B}}$, then the positive semidefinite matrix $\bm A = \sum_j \bs s_j\bs s_j^T$ satisfies $\x^T\bm A\x = p(\x)\geq 0$ for all $\x$. Conversely, let $\bm A$ be a matrix such that $p(\x) = \x^T\bm A\x$. Since $\bm A \succeq 0$, it has the spectral decomposition $\bm A = \sum_j \mu_j \bm v_j \bm v_j^T$ for some $\mu_j \geq 0$. Then we write $p(\x) = \sum_j \mu_j (\x^T \bm v_j)^2$, which is a sum of squares. 


\end{proof}

It is important to notice that proposition~\ref{sosissdp} doesn't provide a strict equivalence between a matrix certificate and a polynomial certificate. Observe that the existence of a sum-of-squares polynomial $p(\bs x)$ such that $p(\bs x) = \mu_{\bs x}^T\bm A\mu_{\bs x}$ doesn't imply that $\bm A$ is positive semidefinite. In other words, not all matrices encoding sum-of-squares polynomials are PSD.
As an illustration, consider the following example:

\begin{example}[sum-of-squares and positive semidefiniteness]
$$\bm A = \left(\begin{array}{ccc}
0& 0 &0\\
0& 1 &0\\
0& 0 &0\\
\end{array}\right),\;\bm B = \left(\begin{array}{ccc}
0& 0 &1/2\\
0& 0 &0\\
1/2& 0 &0\\
\end{array}\right),\;\bm C = \left(\begin{array}{ccc}
0& 0 &1/4\\
0& 1/2 &0\\
1/4& 0 &0\\
\end{array}\right).$$

For the vector of monomials $\z_\mathcal{B} = (1\quad z\quad z^2)$. All those matrices are encoding the same SOS polynomial $p(z)$
$$\z_\mathcal{B}^T\bm A \z_\mathcal{B}= \z_\mathcal{B}^T\bm B \z_\mathcal{B} = \z_\mathcal{B}^T\bm C \z_\mathcal{B}= p(z)= z^2.$$ 
However, only one of them is positive semidefinite. The second and third ones can therefore not be used as the matrix form of a SOS-type certificate. However, note that there exists a matrix $\bs \xi$ such that $\langle \bs \xi,\bm B_\gamma\rangle = 0 $ for all $\bs \gamma$, satisfying $\bm C + \bs \xi = \bm A$ or equivalently $\bm B + \bs \xi = \bm A$. Indeed, for $\bm C$ it suffices to take 
$$\bs \xi = \left(\begin{array}{ccc}
0 & 0 & -1/4\\
0 & 1/2 & 0\\
-1/4 & 0 & 0
\end{array}\right) $$
This is the point of the following lemma which formalizes and closes the gap between matrix and polynomial certificate.
\end{example}

The following lemma proves equivalence of the matrix certificates up to a $\bs \xi$ provided that the corresponding polynomials are the same.
\begin{lemma}\label{auxlemma1}
Let $Y_1$ and $Y_2$ be two matrices such that $\z_{\mathcal{B}}^T Y_1\z_{\mathcal{B}} = \z_{\mathcal{B}}^T Y_2\z_{\mathcal{B}}$ for all $z$, i.e., the polynomials corresponding to $Y_1$ and $Y_2$ are identical. Then there exists a matrix $\bs \xi$ with $\langle \bs \xi, \bm B_{\bs \gamma}\rangle =0 $ for all $\bs \gamma$ and such that $Y_1 = Y_2 + \bs \xi$. 
\end{lemma}
\begin{proof}
$\z_{\mathcal{B}}^T(Y_1-Y_2)\z_{\mathcal{B}} =0 \quad \forall z \; \Leftrightarrow \; \langle Y_1-Y_2, \z_{\mathcal{B}} \z_{\mathcal{B}}^T\rangle =0 \quad \forall z  \; \; \Rightarrow \; \langle Y_1-Y_2, \bm B_{\bs \gamma}\rangle = 0 \quad \forall \bs \gamma$. This last implication holds in the reverse direction: if a polynomial $p(\z)$ has all zero coefficients, then it must be the zero polynomial. 
\end{proof}

The conditions of proposition~\ref{propositionUniqueness}, together with proposition~\ref{sosissdp} and lemma~\ref{auxlemma1} imply the following result, arising from the polynomial nature of problem~\eqref{matrixcompletion},

\begin{proposition}[Polynomial Form]\label{certificateuniquePoly}
To ensure \textit{unique recovery} of $\bm M_0 = \bs m_0 \bs m_0^T$ with $(\bs m_0)_{\bs \gamma} = \z_0^\gamma$, in addition to the injectivity of the constraints on $T$, it is sufficient to find a sum of $(|\mathbb{N}_2^{K}|-1)$ linearly independent squares $s_j^2(z)$ of degree less than or equal to $4$, polynomials $\lambda_\ell(z)$ of degree less than or equal to $4-2d_{h_\ell}$ and constant $\rho$ such that
\begin{equation}
q(z) = \sum_j s_j^2(z) = \sum_{\bs \alpha\in \mathbb{N}^K_2} {\bf z}^{2\bs \alpha}- \rho + \sum_{\ell} h_\ell(z) \lambda_\ell(z),\label{Polycert1}\end{equation}
and such that $q(z_0) = 0$.
\end{proposition}

\begin{proof}
The form of the polynomial $q(z)$ in~\eqref{Polycert1} implies the existence of a matrix $Y_1$ in the range of $\mathcal{A}^*$ such that $\z_{\mathcal{B}}^T(I-Y_1)\z_{\mathcal{B}} = \sum_{j}s_j^2(z)$. By lemma~\ref{auxlemma1}, we can then add a matrix $\bs \xi$ to $I-Y_1$ to get the positive semidefinite matrix $\sum_{j}\bs s_j\bs s_j^T = I-Y_2 = I-Y_1-\bs \xi \succeq 0$ which now satisfies the condition $Y_{T}^\perp\succeq 0$ of section~\ref{sec:dualcert}. Note that such a $\bs \xi$ always exists, by lemma~\ref{auxlemma1}, as we have  
$$q(z) = \z_{\mathcal{B}}^T\sum_{j} \bs s_j\bs s_j^T \z_{\mathcal{B}}=\z_{\mathcal{B}}^T( I-Y_2)\z_{\mathcal{B}} = \z_{\mathcal{B}}^T( I-Y_1-\bs \xi)\z_{\mathcal{B}} = \z_{\mathcal{B}}^T( I-Y_1)\z_{\mathcal{B}}$$
Finally, as indicated by proposition~\ref{certificateuniquePoly}, since $q(z)$ is SOS, to satisfy the last condition, $Y_T = 0$, it suffices to require $s_j(z_0)= 0$. Indeed, for $(\bs m_0)_{\bs \gamma} = \z_0^\gamma$, we have 
$$\{s_j(z_0)=0, \; \forall j \}\iff\{\langle \sum_j \bs s_j\bs s_j^T,\bs m_0 \bs y^T + \bs y \bs m_0^T\rangle, \; \forall \bs y\in \mathbb{R}^{|\mathbb{N}_2^{K}|}\}= 0.\;$$
The value of the constant $\rho$, which derives from the one degree of freedom of $\bs B_{0}$, is fixed by enforcing $q(z_0)=0$. The last term on the RHS of~\eqref{Polycert1} is a contribution of degree $\leq 4$ from the ideal $\;\mathcal{I} := \{\sum_{j=1}^L u_j(z)h_j(z)\;|\;u_1,\ldots,\;u_L\in \mathbb{R}[\bm z]\}$ generated from the constraints $h_j(z)$.
\end{proof}

Because of proposition~\ref{sosissdp} and lemma~\ref{auxlemma1}, we can now just focus on finding a (dual) polynomial $q(z)$ with the structure~\eqref{Polycert1}.

\subsection{\label{sec:proof}Construction of the dual polynomial}

In this section we show how to construct the dual polynomial satisfying the decomposition~\eqref{Polycert1}. As explained above, such a polynomial implies the existence of a matrix $Y$ satisfying the conditions 1) to 3) and proposition~\ref{propositionUniqueness} and serves as the first part of the proof of Theorem~\ref{maintheorem}. We then prove injectivity on $T$ to conclude this proof.

Remember that $\bs z$ is given by the concatenation $\bs z= (\bs x,\;\bs y)$ of all first order monomials arising in problem~\eqref{matrixcompletion}. Our construction of the certificate is based on choosing the squares on the LHS of~\eqref{Polycert1} to be the canonical polynomials $(\z^{\bs \alpha}-\z_0^{\bs \alpha})^2$ for all $|\bs \alpha|\leq 2$ and to show that those canonical squares can be obtained from the ideal; the squared monomials arising from the trace norm and the constant $\rho = \sum_{\bs \alpha} \z_0^{2\bs \alpha}$. The resulting expression for the certificate is simply
$$q(z) = \sum_{|\bs \gamma|\leq 2} (\z^{\bs \gamma} - \z_0^{\bs \gamma})^2.$$
First, let us show that for all monomials $\z^{\bs \alpha}$ with $|\bs \alpha|=1$ one can build the polynomial $-2\z^{\bs \alpha} \z_0^{\bs \alpha}+2(\z^{\bs \alpha}_0)^2$ by using a decomposition from the ideal of degree at most $3$. 

\begin{itemize}
\item Either the constraint $\z^{\bs \alpha} = \z^{\bs \alpha}_0$ is present explicitly ($\z^{\bs \alpha} = y_\ell$ corresponds to an element of the first row of $\bm X$ and $h_\ell( z) \equiv y_\ell - (y_0)_\ell$ is a constraint in $\Omega$) and one can then just multiply this constraint by $-2(\z_0^{\bs \alpha})$ to get the desired polynomial $-2(\z_0)^{\bs \alpha} \z^{\bs \alpha} + 2(\z_0^{\bs\alpha})^{2}$

\item Or, since the bipartite graph is connected, the first order monomial $\z^{\bs \alpha}, \; |\bs \alpha|=1,$ appears in a chain like~\eqref{generalchain}, such that if we denote the corresponding numerical values by $(z_0)_{i_1}$, $(z_0)_{i_1}(z_0)_{i_2}$, \ldots,$(z_0)_{i_{\ell-1}} (z_0)_\ell$, the constraints $z_{i_1} - (z_0)_{i_1},\ldots, z_{i_\ell-1} z_\ell - (z_0)_\ell (z_0)_{i_{\ell-1}}$ belong to $\Omega$ and thus to the ideal $\mathcal{I}$. Using~\eqref{generalchain}, one can thus recursively combine the elements of the chain in the following way, 
\begin{align}
(z_0)_{i_{\ell-2}}(z_0)_{i_{\ell-1}} (z_\ell - (z_0)_\ell) &= (z_\ell z_{i_{\ell-1}} - (z_0)_\ell (z_0)_{i_{\ell-1}})z_{i_{\ell-2}}\nonumber \\
&- (z_{i_{\ell-2}} z_{i_{\ell-1}} - (z_0)_{i_{\ell-2}} (z_0)_{i_{\ell- 1}})z_{\ell}\nonumber \\
&+ (z_0)_{\ell} (z_0)_{i_{\ell-1}} (z_{i_{\ell-2}} - (z_0)_{i_{\ell-2}}).\label{chain2}
\end{align}
This telescoping relation holds for all $\ell$ throughout the chain until the second element, ($z_{i_2}$), for which we have 
$(z_0)_{i_1}(z_{i_2} - (z_0)_{i_2})=(z_{i_2}z_{i_1} - (z_0)_{i_2}(z_0)_{i_1}) -z_{i_2} (z_{i_1}-(z_0)_{i_1})\in \mathcal{I}.$ 
The key here is that one can make use of the bilinear constraints to get a propagation argument which remains degree-$3$ since the multiplicative factor $(z_0)_{\ell}(z_0)_{i_{\ell-1}}$ in front of the propagation term $(z_{i_{\ell-2}} - (z_0)_{i_{\ell-2}})$ remains constant. In particular, note that we never use the third order constraints $z^{\bs \alpha}(z_{i_1} - (z_0)_{i_1})$ for $|\bs \alpha|=2$, namely the highest degree of the monomials multiplying the first order constraints is one. This will be important later when establishing the stability result.

\end{itemize}

Now that we can build the polynomials $-2(z_0)_k z_k + 2 (z_0)_k^2$ for all $k$ as degree-$3$ decompositions from the ideal $\mathcal{I}$, one can just add those polynomials to the trace and constant $\rho$ contributions $z_k^2 - (z_0)_k^2$ in order to get the squares $(z_k - (z_0)_k)^2$. We thus get $|\mathbb{N}_1^{K}|-1$ of the required squares. The remaining $K \choose 2$ decompositions for the second order squared polynomials $(\z^{\bs \alpha} - \z_0^{\bs \alpha})^2$ for $|\bs \alpha|=2$, are built from the first order decompositions, the trace, and constant $\rho$ as follows. $\forall \bs \alpha,\bs \beta$ with $|\bs \alpha|,|\bs \beta|=1$, 
\begin{align}\begin{split}
(\z^{\bs \alpha} \z^{\bs \beta} - \z_0^{\bs \alpha} \z_0^{\bs \beta} )^2 &= (\z^{\bs \alpha} \z^{\bs \beta})^2 -  (\z_0^{\bs \alpha} \z_0^{\bs \beta})^2 \\
& -2 \z_0^{\bs \alpha} \z_0^{\bs \beta} (\z^{\bs \alpha} \z^{\bs \beta} - \z_0^{\bs \alpha} \z_0^{\bs \beta}), \end{split}\label{secondOrderDec1}  
\end{align}
where the first two terms arise from the contribution of the trace and $\rho$, and the third one can be expressed from the ideal $\mathcal{I}$ with degree at most 4, as
\begin{align}\begin{split}
-2 \z_0^{\bs \alpha} \z_0^{\bs \beta} (\z^{\bs \alpha} \z^{\bs \beta} - \z_0^{\bs \alpha} \z_0^{\bs \beta}) 
&= (-2\z^{\bs \alpha} \z_0^{\bs \alpha} + 2(\z_0^{\bs \alpha})^2) (\z_0^{\bs \beta})^2 \\
&+ (\z^{\bs \beta} - \z^{\bs \beta}_0)(-2\z^{\bs \alpha} \z^{\bs \alpha}_0)\z_0^{\bs \beta}
\end{split}\label{secondOrderDec2}
\end{align}
The first term is of degree at most $3$ and the second one is of degree at most $4$.\\

To conclude the proof of Theorem~\ref{maintheorem}, we show that the linear map $\mathcal{A}$ grouping the linear constraints derived from the polynomials $h_\ell$ and the structure of the moments matrix,
is injective on $T$. For this purpose, let us show that the nullspace of $\mathcal{A}$ is empty on $T$.  Let us consider any $\bm H = \bs m_0 \bs v^T + \bs v \bs m_0^T$. Normalization of $\bm H_{11}$ implies $v_1=0$ and reduces $\bm H$ to a matrix for which the first column equals the first row and is given by $(v_2\ldots, v_{|\mathbb{N}_2^K|})$. Then recall that there is a least one constraint setting to zero one of the elements of the first column. So there exists $\ell$ s.t. $v_\ell = 0$. Accordingly the whole corresponding row and column reduce to $((z_0)_\ell v_k)_{k\leq |\mathbb{N}_2^K|}$. Since $(z_0)_\ell\neq 0$\footnote{Recall that we assumed $(\bm X_0)_{ij}\neq 0$ for all $(i,j)$}, one can then  apply the next constraint $z_\ell z_m = 0$ which implies $v_m = 0$. By recursively applying this idea, one can show that the first block of $\bm H$ corresponding to the monomials of degree at most two is zero. The remaining part of the matrix can then be set to $0$ as well trough the structural constraints (equality of corresponding monomials) for the first row/column and then using the fact that $\bm H$ is defined as $\bm m_0 \bs v^T + \bs v \bm m_0^T$. 

\section{\label{sec:stability}Stability}

In this section, we prove Theorem~\ref{noisycase} and Corollary~\ref{noisycaseCorollary}. We let the noisy measurements be given by $(\tilde{X}_0)_{ij} = (X_0)_{ij} + \varepsilon_{ij}$ for $(i,j)\in \Omega$. We further let $\tilde{h}_\ell$ denote the corresponding noisy constraints. If $h_\ell(\z):=\z^{\bs \alpha} - \z^{\bs \alpha}_0$ denotes a constraint in $\Omega$ with either $\z^{\bs \alpha} = x_{i+1}y_{j}$ and $(i,j)\in \Omega$ or $\z^{\bs \alpha} = y_{\ell}$, $(1,\ell)\in \Omega$, we let $\tilde{h}_\ell(\z):=\z^{\bs \alpha} - \tilde{\z}^{\bs \alpha}_0$ denote the corresponding noisy constraint with $\tilde{\z}_0^{\bs \alpha} = \z_0^{\bs \alpha} + \varepsilon_{ij}$. Hence, $\tilde{h}_\ell(\z) = h_\ell(\z) - \varepsilon_{ij}$ with $(i,j)$ relative to the constraint indexed by $\ell$, or with a slight abuse of notation, 
$\tilde{h}_\ell(\z) = h_\ell(\z) - \varepsilon_{\ell}$.

Let $\eta\geq \|\varepsilon\|_2\sqrt{(1+\|\z_0\|_1^2|)}$ with $\|\varepsilon\|_2 = \sqrt{\sum_{ij\in \Omega} \varepsilon_{ij}^2}$. The stable version of~\eqref{momentProblem} reads, 
\begin{align}\begin{split}
\text{minimize} \quad & \text{Tr}(\bm M)\\
\text{subject to} \quad & \sqrt{\sum_{\bs \kappa}\sum_{\ell}\left|\sum_{\bs \zeta} \frac{(\tilde{h}_\ell)_{\bs \zeta}}{\|\bm B_{\bs \zeta+\bs \kappa}\|^2_F}\langle \bm M, \bm B_{\bs \zeta+\bs \kappa}\rangle \right|^2}\leq \eta.\\
& \text{for}\quad \bs \kappa\in \mathbb{N}^K_{2(t- d_{h_\ell} )}, 1\leq \ell \leq L\\
& \bm M\succeq 0,\quad \bm M_{11} = 1,\\
& \langle \bm M, \bm E_{\bs \delta_1,\bs \delta_2} - \bm E_{\bs \gamma_1,\bs \gamma_2}\rangle = 0,\\ 
& \mbox{for all $(\bs \delta_1,\bs \delta_2), (\bs \gamma_1,\bs \gamma_2)\;\mbox{s.t.}\;\bs \delta_1+ \bs \delta_2= \bs \gamma_1+ \bs \gamma_2\leq 2t$}.
\end{split}
\label{noisy}
\end{align}
The first constraint in formulation~\eqref{noisy} is simply the $\ell_2$ norm of the constraints appearing in~\eqref{momentProblem}. In this first constraint, the first sum is taken over the different noisy polynomials $h_\ell$ and the second is taken over all the ``shifts" of those polynomials. For a given $\bs \kappa$, the corresponding shifted polynomial is simply obtained by multiplying $h_\ell$ by the corresponding monomial $\z^{\bs \kappa}$. 

It is worth pointing out that formulation~\eqref{noisy} is not unit-independent, since the moment matrix $M$ mixes different powers of the original variables. This can be remedied by assigning dimensional weights $w_{\gamma}$ to the $B_\gamma$ matrices -- an operation that modifies the numerics and the theory in an obvious way. Formulation~\eqref{noisy} leads to the recovery result of Theorem~\ref{noisycase} which is restated below for clarity.

\theoremCompletion*

The stability result of Theorem~\ref{noisycase} can be improved if a path is known that relates one entry to all the others. In this last case, the scalings can be reduced from $\mathcal{O}((m+n)^{7/2})$ to $\mathcal{O}((m+n)^{2})$. This is the point of Corollary~\ref{noisycaseCorollary} which is proved in section~\ref{proofCorol},

\corollaryCompletion*

Let $\mathcal{P}_1,\ldots, \mathcal{P}_P$ denote the sets of constraints that appear along each path between the root node and the leaf nodes in the sense of~\eqref{generalchain}. As explained in section~\ref{sec:proof}, the dual certificate only relies on the monomials appearing along each of the paths multiplied either by the previous missing variable or the next one. For each path $\mathcal{P}_i$, let $\mathcal{K}_i$ denote the subset of multi-indices corresponding to variables that are multiplying the constraints in the chain in the expression of the certificate~\eqref{chain2}. Let $\eta'\geq \|\varepsilon\|_2\sqrt{1+\sup_{|\bs \alpha|\leq 1} \z_0^{2\bs \alpha}}$. The formulation for Corollary~\ref{noisycaseCorollary} is obtained by replacing the $\ell_2$ constraint in~\eqref{noisy} by a corresponding $\ell_2$ term minimizing the noise along the paths,
\begin{align}
\sqrt{\sum_{\bs \kappa\in \mathcal{K}_i}\sum_{\ell\in \mathcal{P}_i}\left|\sum_{\bs \zeta} \frac{(\tilde{h}_\ell)_{\bs \zeta}}{\|\bm B_{\bs \zeta+\bs \kappa}\|^2_F}\langle \bm M, \bm B_{\bs \zeta+\bs \kappa}\rangle \right|^2}\leq \eta', \quad i=1,\ldots, P.\label{constraintsPathsL2}
\end{align}
The improvement in the prefactors of Theorem~\ref{noisycase} essentially arises from the tighter bound $\eta'$ on the $\ell_2$ constraints in~\eqref{constraintsPathsL2}. This tighter bound is due to the fact that along a given path, the constraints are always distinct and that following the discussion in section~\eqref{sec:proof}, one can express each first order monomial that appear in the path from the constraints along the path multiplied by either the first or the previous or next missing degree one monomial. This idea is expressed through section~\ref{proofCorol}. 

\subsection{\label{sec:proofStabilityTh}Proof of Theorem~\ref{noisycase}}

Let $\|\bm M\|_p$ to denote the Schatten $p$-norm of $\bm M$, 
$$\|\bm M\|_p = \left(\sum_k \sigma_k^p\right)^{1/p}.$$
We therefore have $\|\bm M\|_1 = \|\bm M\|_*$ which denotes the nuclear norm of $\bm M$, $\|\bm M\|_2 = \|\bm M\|_F$ which is used to denote the Frobenius norm of $\bm M$ and $\|\bm M\|_\infty=\|\bm M\|$ which denotes the operator norm of $\bm M$. %

Any solution $\bm M$ to~\eqref{noisy} reads $\bm M = \bm H + \bm M_0$. 
To prove stability of the recovery, we first highlight the following, 
\begin{itemize}
\item $\text{Tr}(\bm M_0+ \bm H)\leq \text{Tr}(\bm M_0)$ and therefore $\text{Tr}(\bm H)\leq 0$. 
\item Both $\bm M$ and $\bm M_0$ are feasible points for~\eqref{noisy}, and hence both satisfy the normalization constraint $\bm M_{11} = (\bm M_0)_{11}$ which can be exactly enforced. As a consequence, $\bm H_{11} = (\bm M_0)_{11} - \bm M_{11} = 0$, and all degree zero terms in the constraints $\tilde{h}_1,\ldots,\tilde{h}_L$ vanish when those constraints are applied to $\bm H$. We have
\begin{align}
\frac{(\tilde{h}_\ell)_{\bs 0}}{\|\bm B_{\bs 0}\|^2_F}\langle \bm H, \bm B_{\bs 0}\rangle = 0, \qquad \ell=1,\ldots, L.
\end{align}
More generally, for both $\bs M$ and $\bs M_0$, as $\eta$ is bounding the vector $(\varepsilon_{\ell} \z_0^{\bs \kappa})_{\ell,|\bs \kappa|\leq 2}$ of weighted residuals, we must have,
$$\sum_{\ell,\kappa}\left|\sum_\zeta (\tilde{h}_\ell)_{\zeta } \langle \bs B_{\zeta+\kappa} , \bs M_0\rangle \right|^2= \sum_{\ell,\kappa}\left|\sum_\zeta (h_\ell)_{\zeta } \langle \bs B_{\zeta+\kappa} , \bs M_0\rangle - \varepsilon_\ell \z_0^{\kappa}\right|^2 \leq\sum_{\ell,\kappa}( \varepsilon_\ell \z_0^{\bs \alpha})^2\leq \eta^2$$
$$\sum_{\ell,\kappa}\left|\sum_\zeta (\tilde{h}_\ell)_{\zeta } \langle \bs B_{\zeta+\kappa} , \bs M\rangle \right|^2 \leq\sum_{\ell,\kappa}( \varepsilon_\ell \z_0^{\bs \alpha})^2\leq \eta^2$$
From those relations we can derive a similar bound on $\bs H$,
\begin{align*}
\sqrt{\sum_{\bs \kappa}\sum_{\ell}\left|\sum_{\bs \zeta} \frac{(\tilde{h}_\ell)_{\bs \zeta}}{\|\bm B_{\bs \zeta+\bs \kappa}\|^2_F}\langle \bm H, \bm B_{\bs \zeta+\bs \kappa}\rangle \right|^2}& = \sqrt{\sum_{\bs \kappa}\sum_{\ell}\left|\sum_{\bs \zeta} \frac{(\tilde{h}_\ell)_{\bs \zeta}}{\|\bm B_{\bs \zeta+\bs \kappa}\|^2_F}\langle \bm M - \bm M_0, \bm B_{\bs \zeta+\bs \kappa}\rangle \right|^2}\\
& \leq 2\eta.
\end{align*}
\item Finally, note that $\bm H + \bm M_0\succeq 0$ implies $\langle \bm H+ \bm M_0, \bm W\rangle \geq 0$ for all $\bm W\succeq 0$ including all $\bm W\in T^\perp$ which implies $\bm H_{T^\perp}\succeq 0$. 
\end{itemize}

The polynomial form of $\bm Y_2$ belongs to the range of $\mathcal{A}^*$ (i.e, its polynomial form belongs to the ideal $\mathcal{I}$) modulo a $\bs \xi$ and $\bm Y_2$ is written as 
\begin{align}\bm Y_2 =  \bs \xi +\sum_{\ell}\sum_{\bs \kappa\in \mathbb{N}^K_{2(t- d_{h_\ell} )}}\lambda_{\ell\bs \kappa}\left(\sum_{\bs \zeta} \frac{(h_\ell)_{\bs \zeta}}{\|\bm B_{\bs \zeta+\bs \kappa}\|^2_F}\bm B_{\bs \zeta+ \bs \kappa} \right),\label{reminderCertY2}\end{align}
where $\bs \xi$ is orthogonal to the $\bm B_{\bs \gamma}$. 


The next section derives a bound on $\|\bs H_{T}^\perp\|$. For this, we start by bounding $|\langle \bs H, \bs Y_1\rangle |$. Note that $|\langle \bs Y_2,\bs H\rangle| = |\langle \bs Y_1,\bs H\rangle|$, as $\langle \bs \xi, \bs H\rangle = 0$.

\subsubsection[bound]{Bound on $\bm H_T^\perp$}

The certificate~\eqref{reminderCertY2} is built from the noiseless constraints $h_\ell$, while the solutions $\bs M$, $\bs M_0$ and thus $\bs H$ are bounded with respect to the corrupted constraints from~\eqref{noisy}. As we saw earlier, the noisy constraints relate to the noiseless constraints as
\begin{align}
\tilde{h}_\ell(z)= \sum_{\zeta} (\tilde{h}_\ell)_\zeta \z^{\zeta} = \sum_{\zeta\neq 0} (h_\ell)_\zeta \z^{\zeta} + (h_\ell)_0 - \varepsilon_\ell   = h_\ell(z) - \varepsilon_\ell.
\end{align}
Let $\bs Y_1^{(1)}$ and $\bs Y_1^{(2)}$ denote the contributions to $Y_1$ corresponding to the first and second order squares in the sos certificate of section~\ref{sec:proof} respectively. For any constraint $h_\ell$, using the recursion~\eqref{chain2}, the difference between $h_\ell$ and $\tilde{h}_\ell$ will only affect the entries in $\bs H$ corresponding to first and zero order moments. Since $\bs H_{1,1} = 0$ (see the discussion above), this discrepancy will thus only affect first order entries. Let $W_\ell$ denote the number of times that each constraint is used in the construction of $Y_1^{(1)}$ 
\begin{align}
|\langle \bs Y_1^{(1)},\bs  H\rangle | &=\left|\sum_{\kappa} \sum_\ell W_{\ell,\kappa} \sum_{\zeta} (h_\ell)_{\zeta} \langle \bs B_{\zeta+\kappa}, \bs H\rangle \right|\\
&= \left|\sum_{\kappa} \sum_\ell W_{\ell,\kappa} \sum_{\zeta} (\tilde{h}_\ell)_{\zeta} \langle \bs B_{\zeta+\kappa}, \bs M\rangle - \sum_{|\kappa|\leq 1} \sum_\ell W_{\ell,\kappa} \varepsilon_\ell \langle \bs B_{\kappa}, \bs M\rangle \right|
\end{align}
\begin{align}
|\langle \bs Y_1^{(1)},\bs  H\rangle |& \leq \left|\sum_{\kappa} \sum_\ell W_{\ell,\kappa} \sum_{\zeta} (\tilde{h}_\ell)_{\zeta} \langle \bs B_{\zeta+\kappa}, \bs M\rangle \right| + \left|\sum_{|\kappa|\leq  1} \sum_\ell W_{\ell,\kappa} \varepsilon_{\ell} \langle \bs B_{\kappa}, \bs M\rangle \right|\label{finalBoundHY0} \\
&\leq \left|\sum_{\kappa} \sum_\ell W_{\ell,\kappa} \sum_{\zeta} (\tilde{h}_\ell)_{\zeta} \langle \bs B_{\zeta+\kappa}, \bs M\rangle \right| + \mathcal{O}((m+n)^{3/2}) \|\bs M_{11}\|_F \|\varepsilon\|_\infty\label{finalBoundHY1}\\
&\leq \left|\sum_{\kappa} \sum_\ell W_{\ell,\kappa} \sum_{\zeta} (\tilde{h}_\ell)_{\zeta} \langle \bs B_{\zeta+\kappa}, \bs M\rangle \right| + \mathcal{O}((m+n)^{3/2}) \|\bs m_0\| \|\varepsilon\|_\infty\label{finalBoundHY1b}\\
&\leq \mathcal{O}((m+n)^{3/2})\eta + \mathcal{O}((m+n)^{3/2})  \|\bs m_0\|  \|\varepsilon\|_\infty\label{finalBoundHY2}
\end{align}


%
In~\eqref{finalBoundHY2} we use the fact that $\bs Y_1^{(1)}\in \mathcal{R}an(\mathcal{A}^*)$ and $\left\{\mathcal{A}(\bs M_0)\right\}_{\ell\neq 0}= 0$ (the first constraint is simply $(\bs M_0)_{11} = 1$ and does not appear in $\bs Y_1$). Let us introduce the following decomposition for $\bs M/\bs M_0$,
\begin{align}
\bs M = \left[\begin{array}{cc}
M_{11} & M_{12}\\
M_{21} & M_{22}\end{array}\right].\label{blockDecomposition}
\end{align}
%
%
%
For both $\bs M$ and $\bs M_0$, because of the structural constraints and PSD constraint,  one can write $\tr(\bs M)\geq \|\bs M_{11}\|^2_F + \tr(\bs M_{11})$. Moreover, we have $\tr(\bs M_0)\geq \tr(\bs M)$ so in particular, we have 
$$\|\bs M_{11}\|_F\leq \sqrt{\tr(\bs M) - \tr(\bs M_{11})}\leq \sqrt{\tr(\bs M)}\leq \sqrt{\tr(\bs M_0)} = \|\bs m_0\|.$$
In~\eqref{finalBoundHY0}, since all the $\bs B_{\bs \kappa}$ are accessing moments of order at most one in $\bs M$, the sum $\sum_{\ell}\sum_{\kappa} \bs B_{\kappa}$ in the second term of~\eqref{finalBoundHY0} has the form $ve_1^* + e_1v^*$ where $v_i \leq  \mathcal{O}(m+n)\|\varepsilon\|_\infty$ and the norm $\|\sum_{\ell}\sum_{\kappa} \bs B_{\kappa}\|_F$ can thus be bounded as $\|\sum_{\ell}\sum_{\kappa} \bs B_{\kappa}\|_F = \mathcal{O}((m+n)^{3/2})$.  We can also replace $\bs M$ by $\bs M_{1,1}$. 
Equations~\eqref{finalBoundHY1} and~\eqref{finalBoundHY1b} then follow from Cauchy-Schwarz. For~\eqref{finalBoundHY2} simply note that
\begin{itemize}
\item In the sum, every constraint appears at most $\mathcal{O}(m+n)$ times (as an example, the first constraint $z_{i_1} - (z_0)_{i_1} = 0$ will appear exactly $m+n$ times as it is used to express every square in the chain), i.e, $W_{\ell,\kappa} = \mathcal{O}(m+n)$ if $\kappa$ corresponds to either of the two monomials multiplying the constraints in~\eqref{chain2} and $0$ otherwise. 
\item Each of the constraints is multiplied by at most two different monomials leading to two distinct entries in the vector $(\varepsilon_{\ell} \z_0^{\bs \alpha})_{\ell,|\bs \alpha|\leq 2}$ whose $\ell_2$ norm is bounded by $\eta$.
\end{itemize}

Equation~\eqref{finalBoundHY2} follows from Cauchy-Schwarz, noting that the first term in~\eqref{finalBoundHY1} can be written as $|\langle \tilde{\mathcal{A}}(\bs M), \bs \lambda\rangle |$ where $\|\bs \lambda\|_\infty = \mathcal{O}(m+n)$ and the $\ell_2$ norm of $\tilde{\mathcal{A}}(\bs H)$ is bounded from the constraints in~\eqref{noisy}. 

We now bound the second order contributions gathered in $\bs Y_1^{(2)}$. From~\eqref{secondOrderDec1}, this contribution can be decomposed as $\bs Y_2^{(2)} = \bs V_1 + \bs V_2$, where $\bs V_1$ only involves the decomposition of first order monomials (first term on the RHS of~\eqref{secondOrderDec1}), and $\bs V_2$ denotes the higher order contributions (second term on the RHS of~\eqref{secondOrderDec1}). The contribution of $\bs V_1$, corresponding to the first term in~\eqref{secondOrderDec1} is identical to the contribution from $\bs Y_1^{(1)}$ except that it is now also summed $|\left\{\bs \beta\;|\;|\bs \beta|\leq 1\right\}|$ times. We thus have 
\begin{align}
|\langle \bs H, \bs V_1\rangle| &= \left|\sum_{|\bs \beta|\leq 1} \sum_{\kappa} \sum_\ell W_{\ell,\kappa} \sum_{\zeta} (h_\ell)_{\zeta} \langle \bs B_{\zeta+\kappa}, \bs H\rangle \right| \\
&= \left|\sum_{|\bs \beta|\leq 1} \sum_{\kappa} \sum_\ell W_{\ell,\kappa} \sum_{\zeta} (\tilde{h}_\ell)_{\zeta} \langle \bs B_{\zeta+\kappa}, \bs M\rangle \right| + \left|\sum_{|\bs \beta|\leq 1} \sum_{\kappa} \sum_\ell W_{\ell,\kappa} \sum_{\zeta = 0} \varepsilon_\ell \langle \bs B_{\zeta+\kappa}, \bs M\rangle \right|  \\
& \leq \max\left\{\mathcal{O}((m+n)^{2})\eta , \mathcal{O}((m+n)^{2})  \|\bs m_0\|  \|\varepsilon\|_\infty\right\}\label{lastBoundSecondOrder1Y}
\end{align}

In~\eqref{lastBoundSecondOrder1Y}, we use the fact that every constraint of the form $h_j(\z)\z^{\bs \beta}$ corresponds to a distinct entry in $\eta$. The sum over the multi-indices $\beta$ is thus included into $\mathcal{O}(m+n)$ entries of $\eta$ which have to be multiplied by $\mathcal{O}(m+n)$ as each constraint $h_j(\z)$ appears at most $\mathcal{O}(m+n)$.

For the second term $\bs V_2$, it suffices to note that this term corresponds to multiplying all the polynomials appearing in $\bs Y_1^{(1)}$ by $\z^{\bs \beta}$ and summing up all the resulting polynomials over all possible first order multi-indices $\bs \beta$. In terms of $\bs V_2$, for the term which is multiplying $\varepsilon$, this means shifting the first column in $\sum_\ell\sum_{\kappa}\bs B_{\kappa}$ into  $\sum_\ell\sum_{\kappa}\bs B_{\kappa+\alpha}$ and summing all resulting matrices over $\bs \alpha$. If we let $\mathcal{C}(\beta)$ denote the set of pairs $(\ell,\bs \kappa)$ representing the constraints $h_\ell(z) z^{\bs \kappa}$ that appear in the expression of $\z^{\bs \beta} - \z_0^{\bs \beta}$ following the decomposition~\eqref{chain2}, for the second order contribution $\bs V_2$, using the decomposition given by the second term on the RHS of~\eqref{secondOrderDec1}, we can write 
\begin{align}
|\langle \bs H, \bs V_2\rangle| &= \left|\sum_{|\bs \alpha|\leq 1}\sum_{|\bs \beta|\leq 1} \sum_{(\ell,\kappa)\in \mathcal{C}(\beta)} \sum_{\zeta} (h_\ell)_{\zeta} \langle \bs B_{\zeta+\kappa+\alpha}, \bs H\rangle\right|\\
 &= \left|\sum_{|\bs \alpha|\leq 1} \sum_{|\bs \beta|\leq 1}\sum_{(\ell,\kappa)\in \mathcal{C}(\beta)} \sum_{\zeta\neq 0} (\tilde{h}_{\ell})_{\zeta} \langle \bs B_{\zeta+\kappa+\alpha}, \bs M\rangle\right. \\
&+ \left.\sum_{|\bs \alpha|\leq 1}\sum_{|\bs \beta|\leq 1} \sum_{(\ell,\kappa)\in \mathcal{C}(\beta)} \sum_{\zeta=0} ((\tilde{h}_{\ell})_{\zeta} - \varepsilon_{\ell}) \langle \bs B_{\zeta+\kappa+\alpha}, \bs M\rangle\right|\\
&\leq \left| \sum_{|\bs \alpha|\leq 1}\sum_{|\bs \beta|\leq 1}\sum_{(\ell,\kappa)\in \mathcal{C}(\beta)} \sum_{\zeta} (\tilde{h}_{\ell})_{\zeta} \langle \bs B_{\zeta+\kappa+\alpha}, \bs M\rangle\right| \\
&+ \left| \sum_{|\bs \alpha|\leq 1}\sum_{|\bs \beta|\leq 1}\sum_{(\ell,\kappa)\in \mathcal{C}(\beta)} \sum_{\zeta=0} \varepsilon_{\ell} \langle \bs B_{\kappa+\alpha}, \bs M\rangle\right|\\
&\leq \left| \sum_{|\bs \alpha|\leq 1}\sum_{|\bs \beta|\leq 1}\sum_{(\ell,\kappa)\in \mathcal{C}(\beta)} \sum_{\zeta} (\tilde{h}_{\ell})_{\zeta} \langle \bs B_{\zeta+\kappa+\alpha}, \bs M\rangle\right|\label{noisySecondOrder} \\
&+\mathcal{O}((m+n)^2) \|\bs m_0\| \|\varepsilon\|_\infty\label{bound}
\end{align}
In~\eqref{noisySecondOrder} we use the discussion above and the fact that,again since $\bs B_{\bs \kappa}$ only targets monomials of order at most $2$, we can focus on the submatrix $\bs M_{1,1}$ from the decomposition~\eqref{blockDecomposition}. Moreover, the norm $\|\sum_{|\bs \beta|\leq 1} \sum_{|\bs \alpha|\leq 1}\sum_{(\ell,\kappa)\in \mathcal{C}(\beta)} \varepsilon_{\ell} \bs B_{\kappa+\alpha}\|_F = \|\varepsilon\|_\infty\mathcal{O}(m+n) \|\bs 1\bs 1^*\|_F = \|\varepsilon\|_\infty\mathcal{O}((m+n)^2)$. Equation~\eqref{bound} follows from Cauchy-Schwarz. To bound~\eqref{noisySecondOrder}, simply use the result of~\eqref{finalBoundHY2} (first term), noting that in each $\z^{\beta} - \z_0^{\beta}$ each polynomial from the ideal appears at most $\mathcal{O}(m+n)$ times. Then use the fact that every $\alpha$ in $(\z^{\beta} - \z_0^{\beta})\z^{\bs \alpha}$ gives a different constraint in~\eqref{noisy} so that the sum over $\alpha$ can be included within $\eta$. Apply Cauchy-Schwarz to $|\langle \lambda_2, \tilde{\mathcal{A}}(M)\rangle |$ with the bound on $\tilde{\mathcal{A}}(H)$ given by~\eqref{noisy} and $\|\bs \lambda\| = \mathcal{O}(m+n)$. This gives the following bound $|\langle \bs H, \bs V_2\rangle|$
\begin{align}
|\langle \bs H, \bs V_2\rangle| \leq \max\left\{(m+n)^{2}\eta,\mathcal{O}((m+n)^2) \|\bs m_0\|_2 \|\varepsilon\|_\infty\right\}\label{lastBoundSecondOrder2Y}
\end{align}
%


Consider the sum-of-squares certificate of section~\eqref{sec:proof}. In polynomial form, we have seen that this certificate reads $\sum_{|\bs \gamma|\leq 2} (\z^{\bs \gamma} - \z_0^{\bs \gamma})^2 = \sum_{|\bs \gamma|\leq 2} (\z^{\bs \gamma} - (\bs m_0)_{\bs \gamma})^2$. One possible matrix representation\footnote{An alternative representation would be given by the decomposition Trace + ideal of section~\eqref{sec:proof} and encoded as $\bs I - \bs Y_1$.} of this certificate is thus given by $\bar{Z} = \sum_j \bs s_j\bs s_j^T$ where each $\bs s_j$ denote a vector of the form $-(m_0)_{\bs \gamma}\bs e_1 + \bs e_{\bs \gamma}$. Using this form, we get
\begin{align}
\bar{Z} &= \bm I - \bs m_0\bs e_1^T - \bs e_1\bs m_0^T + \|\bs m_0\|^2\bs e_1\bs e_1^T = \bm I - \bm Y_2. 
\end{align}
For $\bs m\perp \bs m_0$, with $\|\bs m\|=1$ and $m_1 = \bs e_1^T \bs m$, we have
$$\langle \bm Y_2, \bs m\bs m^T\rangle  = - m_1^2\|\bs m_0\|^2\leq 0.$$
This last equation implies that $ (\bm Y_2)_{T^\perp}\preceq 0$. 


Let $\bs I - \bs Y_1$ denote the matrix form of the polynomial certificate constructed in section~\ref{sec:proof}. As we have $\z_{\mathcal{B}}^*(\bs I - \bs Y_1)\z_{\mathcal{B}} = \z_{\mathcal{B}}(\bm I - \bm Y_2)\z_{\mathcal{B}} = \sum_{|\bs \gamma|\leq 2} (\z^{\bs \gamma} - \z^{\bs \gamma}_0)^2 = \sum_j s_j(\z)^2$, proposition~\eqref{auxlemma1} applies and there exists a matrix $\bs \xi$ satisfying $\bs \xi + Y_1 = Y_2$.
Recall that the certificate reads 
\begin{align}
Y &= \bm I - \bm Y_2 = \bm I - \bs \xi - \bm Y_1 = \bm I - \bs \xi - \sum_{\ell}\sum_{\bs \kappa\in \mathbb{N}^K_{2(t- d_{h_\ell} )}}\lambda_{\ell\bs \kappa}\left(\sum_{\bs \zeta} \frac{(h_\ell)_{\bs \zeta}}{\|\bm B_{\bs \zeta+\bs \kappa}\|^2_F}\bm B_{\bs \zeta+ \bs \kappa} \right)  
\end{align}	
Now using $\mbox{Tr}(\bm H)\leq 0$, we can write, 
\begin{align}
0 \; \geq \; &\text{Tr}(\bm H_T) + \text{Tr}(\bm H_{T^\perp})\\
= \; &\langle \bm H, \bm I_T\rangle + \langle \bm H, \bm I_{T^\perp}\rangle \\
= \; &\langle \bm H, \bm I_T\rangle - \langle \bm H, \bm Y_2 \rangle  + \langle \bm H, \bm Y_2\rangle  +\langle \bm H, \bm I_{T^\perp}\rangle \\
= \; &\langle \bm H_T, \bm I_T - (\bm Y_2)_T\rangle - \langle \bm H_{T^\perp}, (\bm Y_2)_{T^\perp} \rangle  + \langle \bm H, \bm Y_2\rangle  +\langle \bm H, \bm I_{T^\perp} \rangle \\
\geq \;  & - \langle \bm H_{T^\perp}, (\bm Y_2)_{T^\perp} \rangle  -| \langle \bm H, \bm Y_2\rangle|  +\langle \bm H, \bm I_{T^\perp} \rangle \\
\geq \;  & - \langle \bm H_{T^\perp}, (\bm Y_2)_{T^\perp} \rangle  -| \langle \bm H, \bm Y_1\rangle|  +\langle \bm H, \bm I_{T^\perp} \rangle \label{TraceBound1a}\\
\geq \; &   -| \langle \bm H, \bm Y_1\rangle|  +\text{Tr}(\bm H_{T^\perp}) 
\label{Tracebound1b}
\end{align}
As explained above, $\bm Y_1$ is used to denote the component of the dual certificate which is in the range of $\mathcal{A}^*$, i.e. $\bm Y_1= \mathcal{A}^*\lambda$. In~\eqref{TraceBound1a}, we use the fact that $\bs Y_1 = \bs Y_2 + \bs \xi$ and $\langle \bs \xi, \bs B_{\bs \gamma}\rangle  = 0$, for all $\bs \gamma$. Since both $\bs M$ and $\bs M_0$ are solutions to problem~\eqref{noisy}. Both of these matrices thus satisfy the structural constraints exactly, and read $\bs M = \sum_{\bs \gamma} m_{\gamma}\bs B_{\bs \gamma}$, $\bs M_0 = \sum_{\bs \gamma} (m_0)_{\bs \gamma}\bs B_{\gamma}$ for some $m_{\gamma}$. Together with lemma~\ref{auxlemma1}, this implies $\langle \bs H, \bs \xi\rangle = \langle \bs M-\bs M_0,\bs \xi\rangle  = 0$. 
Finally, in~\eqref{Tracebound1b}, we use the fact that for a positive semidefinite matrix $\bm H_{T^\perp}$, and a matrix $(\bm Y_2)_{T^\perp}$ such that $(\bm Y_2)_{T^\perp}\preceq 0$, $\langle \bm H_{T^\perp}, (\bm Y_2)_{T^\perp}\rangle  \leq 0$. 

We also use the fact that both $\bm M$ and $\bm M_0$ satisfies the structural constraints so that $\langle \bs \xi, \bm H\rangle =0$. The last line implies 
%
\begin{align}
\text{Tr}(\bm H_{T^\perp}) \leq |\langle \bm H, \bm Y_1 \rangle |&\leq \max\left\{(m+n)^{2}\eta,\mathcal{O}((m+n)^2) \|\bs m_0\|_2 \|\varepsilon\|_\infty\right\}
\label{traceBound}\end{align}
Finally $|\langle \bs H, \bs Y_1\rangle|$ is bounded from~\eqref{finalBoundHY2}, ~\eqref{lastBoundSecondOrder1Y} and~\eqref{lastBoundSecondOrder2Y}.

\subsubsection[bound HT]{Bound on $\bs H_T$}
We now use a more quantitative version of injectivity of the linear map $\mathcal{A}$, encoding the polynomial constraints, on $T$ to derive a bound on $\bm H_T$. Let $\bm H_T$ be expressed as $\bm H_T = \bs y\bs m_0^T + \bs m_0\bs y^T$ for some $\bs y\in \mathbb{R}^{\mathbb{N}_2^K}$ (see~\eqref{Tspace}). 


Using this decomposition for $\bs H_T$, and letting $h_{i_1}\rightarrow h_{i_2}\rightarrow \ldots$ denote the ordered series of constraints making the chain~\eqref{generalchain}, we have
\begin{align*}
y_{i_1} + y_1(\bs z_0)_{i_1} &= \left\{\tilde{\mathcal{A}}(\bm H_T)\right\}_1= \sum_{|\bs \zeta|>0} \frac{(\tilde{h}_{i_1})_{\bs \zeta}}{\|\bm B_{\bs \zeta}\|^2_F}\langle \bm H_T, \bm B_{\bs \zeta}\rangle  \\
y_{i_1}(\bs z_0)_{i_2} +  y_{i_2}(\bs z_0)_{i_1} &= \left\{\tilde{\mathcal{A}}(\bm H_T)\right\}_2= \sum_{|\bs \zeta|>0} \frac{(\tilde{h}_{i_2})_{\bs \zeta}}{\|\bm B_{\bs \zeta}\|^2_F}\langle \bm H_T, \bm B_{\bs \zeta}\rangle\\
y_{i_2}(\bs z_0)_{i_3} +  y_{i_3}(\bs z_0)_{i_2} &= \left\{\tilde{\mathcal{A}}(\bm H_T)\right\}_3= \sum_{|\bs \zeta|>0} \frac{(\tilde{h}_{i_3})_{\bs \zeta}}{\|\bm B_{\bs \zeta}\|^2_F}\langle \bm H_T, \bm B_{\bs \zeta}\rangle\\
y_{i_3}(\bs z_0)_{i_4} +  y_{i_4}(\bs z_0)_{i_3}& = \ldots
\end{align*}
%
To derive a bound for $\bm H_T$, we then isolate each of the entries in $\bs y$ as,
\begin{align}\begin{split}
y_{i_1} &= \sum_{|\bs \zeta|>0} \frac{(\tilde{h}_{i_1})_{\bs \zeta}}{\|\bm B_{\bs \zeta}\|^2_F}\langle \bm H_T, \bm B_{\bs \zeta}\rangle - y_1(z_0)_{i_1} \\
y_{i_2}& = \frac{1}{(z_0)_{i_1}} \left[ \sum_{|\bs \zeta|>0} \frac{(\tilde{h}_{i_2})_{\bs \zeta}}{\|\bm B_{\bs \zeta}\|^2_F}\langle \bm H_T, \bm B_{\bs \zeta}\rangle  - (z_0)_{i_2}y_{i_1}\right] \\
y_{i_3}& =  \frac{1}{(z_0)_{i_2}}\left[\sum_{|\bs \zeta|>0} \frac{(\tilde{h}_{i_3})_{\bs \zeta}}{\|\bm B_{\bs \zeta}\|^2_F}\langle \bm H_T, \bm B_{\bs \zeta}\rangle -y_{i_2}(z_0)_{i_3} \right] \\
y_{i_4}&=\ldots\end{split}\label{path1NoPath}\end{align}
\eqref{path1NoPath} thus gives a general expression for every first order entry $y_{i_\ell}$ of $\bs y$, as a weighted combination of the constraints which can be considered as a noisy version of~\eqref{Polycert1} or~\eqref{chain2}. Generally, every first order $y_{i_\ell}$ can thus be expressed as the weighted combination
\begin{align}
y_{i_\ell}& = \frac{1}{(\bs z_0)_{i_{\ell-1}}}\sum_{|\bs \zeta|>0} \frac{(\tilde{h}_{i_\ell})_{\bs \zeta}}{\|\bm B_{\bs \zeta}\|^2_F}\langle \bm H_T, \bm B_{\bs \zeta}\rangle \pm \sum_{j=1}^{\ell-1} \frac{ (\bs z_0)_{i_\ell}}{\bs z_0(i_{\ell - j-1})\bs z_0(i_{\ell-j})} \sum_{|\bs \zeta|>0} \frac{(\tilde{h}_j)_{\bs \zeta}}{\|\bm B_{\bs \zeta}\|^2_F}\langle \bm H_T, \bm B_{\bs \zeta}\rangle\pm (z_0)_{i_\ell}(\bm H_{T^\perp})_{11}.\label{injectivityChain}
\end{align} 
The last term in~\eqref{injectivityChain} follows from $y_1(z_0)_1 = (H_T)_{11} = H_{11} - (H_T^\perp)_{1,1} = - (H_T^\perp)_{1,1}$. Let $(\bs y_{\bs \alpha})_{|\bs \alpha|\leq 1}$ denote the entries in $\bs y$ corresponding to the multi-indices that give rise to degree one monomials. Let $C_3$ bound each of the weights appearing in front of the constraints making up the chain in~\eqref{injectivityChain}. The first order part of $\bs y$, $(\bs y_{\bs \alpha})_{|\bs \alpha|\leq 1}$, has length $m+n$ and each of its entry is bounded by at most a sum of all the constraints making the connected path in the bipartite graph. One can thus write
\begin{align}\|(\bs y_{\bs \alpha})_{|\bs \alpha|\leq 1}\|_2 &\leq C_3 (m+n)^{1/2} \left(\sum_{\ell}\left|\sum_{|\bs \zeta|>0} \frac{(\tilde{h}_\ell)_{\bs \zeta}}{\|\bm B_{\bs \zeta}\|^2_F}\langle \bm H_T, \bm B_{\bs \zeta}\rangle \right| + \|\bm H_{T^\perp}\|_1\right)\nonumber\\
&\leq C_3 (m+n)^{1/2}\left(\sum_{\ell}\left|\sum_{|\bs \zeta|>0} \frac{(\tilde{h}_\ell)_{\bs \zeta}}{\|\bm B_{\bs \zeta}\|^2_F}\langle \bm H_T, \bm B_{\bs \zeta}\rangle \right| + \|\bm H_{T^\perp}\|_1\right).\label{step1a}\end{align}
The constant $C_3$ depends on the entries of $\bs X_0$ as 
$$C_3 = \mathcal{O}\left( \sup_{(k,\ell)\in \Omega} \frac{|z_0(m)|}{|z_0(k)z_0(\ell)|}\right) = \mathcal{O}(1).$$
To bound the second order components of $\bs y$, we use the structural constraints $\langle \bm M, \bm E_{\bs \delta_1,\bs \delta_2} - \bm E_{\bs \gamma_1,\bs \gamma_2}\rangle = 0$, as those are not affected by the noise. Those constraints are enforcing equality between the entries $y_{\ell}m_0(k) + y_{k}m_0(\ell)$ and the second order entries of the first column of $\bm H_T$, namely $y(\ell,k)m_0(1) + y_1m_0(\ell,k)$. For multiindices $\bs \alpha$ and $\bs \beta$ such that $|\bs \alpha| = |\bs \beta|=1$, using those relations, we therefore have
\begin{align}
y_{\bs \alpha + \bs \beta} (\z_0)_1 + (y)_{1} (\z_0)^{\bs \alpha + \bs \beta} & =  y_{\bs \alpha} (m_0)_{\bs \beta} + y_{\bs \beta} (m_0)_{\bs \alpha} + \langle \bm H_T, \bm E_{\bs \alpha+ \bs \beta, 1} - \bm E_{\bs \alpha,\bs \beta}\rangle,\\
&=  y_{\bs \alpha} (m_0)_{\bs \beta} + y_{\bs \beta} (m_0)_{\bs \alpha} + \langle \bm H - \bm H_{T^\perp}, \bm E_{\bs \alpha+ \bs \beta, 1} - \bm E_{\bs \alpha,\bs \beta}\rangle,\\
&=  y_{\bs \alpha} (m_0)_{\bs \beta} + y_{\bs \beta} (m_0)_{\bs \alpha} - \langle \bm H_{T^\perp}, \bm E_{\bs \alpha+ \bs \beta, 1} - \bm E_{\bs \alpha,\bs \beta}\rangle.\label{BoundHigherOrdermomentsHT}
\end{align} 
The last line follows from the fact that $\langle \bm M- \bm M_0, \bm E_{\bs \delta_1,\bs \delta_2} - \bm E_{\bs \gamma_1,\bs \gamma_2}\rangle = 0$ for any $\delta_1+\delta_2 = \gamma_1 +\gamma_2$. Using~\eqref{step1a} as well as H\"{o}lder's inequality and the fact that, for a constant $C_4$, $\|\bm E_{\bs \delta_1,\bs \delta_2} - \bm E_{\bs \gamma_1,\bs \gamma_2}\|_\infty\leq C_4$ for any $(\bs \delta_1,\bs \delta_2),\; (\bs \gamma_1,\bs \gamma_2)$, one can write,
\begin{eqnarray}
\displaystyle \|y_{(\bs \alpha, \bs \beta)}\|_2&\displaystyle\lesssim (m+n)\left(\left\|\left(\sum_{|\bs \zeta|>0} \frac{(\tilde{h}_\ell)_{\bs \zeta}}{\|\bm B_{\bs \zeta}\|^2_F}\langle \bm H_T, \bm B_{\bs \zeta}\rangle\right)_\ell\right\|_1 + \|\bm H_{T^\perp}\|_1 + |y_1| \right)\\
&\displaystyle\lesssim (m+n)\left(\left\|\left(\sum_{|\bs \zeta|>0} \frac{(\tilde{h}_\ell)_{\bs \zeta}}{\|\bm B_{\bs \zeta}\|^2_F}\langle \bm H_T, \bm B_{\bs \zeta}\rangle\right)_\ell\right\|_1 + \|\bm H_{T^\perp}\|_1\right) \label{step2a}
\end{eqnarray} 
In~\eqref{step2a}, we again use $|y_1(m_0)_1| = |(\bm H_T)_{11}| = |\bm H_{11} - (\bm H_{T^\perp})_{11}| = |(\bm H_{T^\perp})_{11}|$. Combining~\eqref{step1a} and~\eqref{step2a}, we get the following bound on $\bm H_T$, 
\begin{align}
\|\bm H_T\|_F&\leq \|\bs m_0\|_2\|\bs y\|_2\nonumber\\
&\leq C_5 \|\bs m_0\|_2 (m+n)\left(\sum_{\ell}\left|\sum_{|\bs \zeta|>0} \frac{(\tilde{h}_\ell)_{\bs \zeta}}{\|\bm B_{\bs \zeta}\|^2_F}\langle \bm H_T, \bm B_{\bs \zeta}\rangle \right| + \|\bm H_{T^\perp}\|_1\right).
\label{BT2Th2}\end{align}
The second term can be bounded by~\eqref{traceBound}. To bound the first term, note that we have $\bm H_T = \bm H - \bm H_{T^\perp}$ and that the expression obtained by considering the first term above and substituting $\bs H$ for $\bs H_T$ can be bounded through the $\ell_2$ constraint in~\eqref{noisy}. We can thus focus on bounding this term when replacing $\bs H_T$ by $\bs H_T^\perp$. In each first order constraint $\tilde{h}_\ell$, there is only one non zero coefficient $(\tilde{h}_\ell)_{\bs \zeta}$ for $|\bs \zeta|>0$ and each $\bm B_{\bs \gamma}$ only has $\mathcal{O}(1)$ non zero entries. Moreover none of the constraints in the chain are targeting the same entry in $\bs H_T^\perp$. Let $\Omega_P$ denote the moments matrix defined as $(\Omega_P)_{\bs \zeta} = \sign (h_\ell)_{\bs \zeta}$, for all $\bs \zeta $ such that there exists a $\ell$ with $(h_\ell)_{\bs \zeta}\neq 0$, and $0$ otherwise. This matrix has Frobenius norm at most $\mathcal{O}(\sqrt{m+n})$. Using this matrix, we can write,
\begin{align}
\sum_{\ell}\left|\sum_{|\bs \zeta|>0} \frac{(\tilde{h}_\ell)_{\bs \zeta}}{\|\bm B_{\bs \zeta}\|^2_F}\langle \bm H_{T^\perp}, \bm B_{\bs \zeta}\rangle \right|&\leq C_6 \left|\langle \Omega_P, \bs H_T^\perp\rangle \right|\label{BoundOnAHTperp1}\\
&\leq C_6 (m+n)^{1/2}\|\bm H_{T^\perp}\|_F\label{BoundOnAHTperp2}\\
&\leq C_6 (m+n)^{1/2}\|\bm H_{T^\perp}\|_1\label{BoundOnAHTperp}
\end{align}
In~\eqref{BoundOnAHTperp1}, we use H\"{o}lder's inequality, together with the fact that the chain has length $\mathcal{O}(m+n)$. Using the trace bound~\eqref{traceBound}, and substituting~\eqref{BoundOnAHTperp} into~\eqref{BT2Th2}, we finally get the bound on $\bm H_T$ as,
\begin{align}
\|(\bm H_T)\|_F &\leq C_7  (m+n)\|\bs m_0\|_2 \left((m+n)^{1/2}\|\bm H_{T^\perp}\|_1 + 
\sum_{\ell}\left|\sum_{\bs \zeta} \frac{(\tilde{h}_\ell)_{\bs \zeta}}{\|\bm B_{\bs \zeta}\|^2_F}\langle \bm H, \bm B_{\bs \zeta}\rangle \right|\right)\\
&\leq C_7  (m+n) \|\bs m_0\|_2 \bigg((m+n)^{1/2}|\langle \bm H,\bm Y_1 \rangle |+ 2\eta (m+n)^{1/2}  \bigg)\\
&\leq C_7  (m+n)^{3/2} \|\bs m_0\|_2 \max\left\{(m+n)^{2}\eta,\mathcal{O}((m+n)^2) \|\bs m_0\|_2 \|\varepsilon\|_\infty\right\}\\
&\leq C_7\|\bs m_0\|^2_2 (m+n)^{5/2}  \eta .
\label{boundHT}
\end{align}
In~\eqref{boundHT}, we use $\|\bs m_0\|_2 = \mathcal{O}(m+n)$. Using this last bound together with~\eqref{traceBound}, we finally get,
\begin{align*}
\|\bm H\|_F &\leq \|\bm H_T\|_F + \|\bm H_{T^\perp}\|_F\leq \|\bm H_T\|_F + \|\bm H_{T^\perp}\|_1\lesssim (m+n)^{5/2}\eta \|\bm M_0\|_F.
\end{align*}
By definition of $\eta$, we also have  $\eta = \mathcal{O}(m+n)\|\varepsilon\|$ which enables to concludes.

The next section shows how the scaling factor can be reduced to $(m+n)^{2}\|\varepsilon\|_2$ when paths are known between any root node and the corresponding leaf nodes in the bipartite graph, and the noise can be constrained along those paths.

\subsection{\label{proofCorol}Proof of Corollary~\ref{noisycaseCorollary}}

The proof of corollary~\ref{noisycaseCorollary} follows the idea of section~\ref{sec:proofStabilityTh} with the difference that we now constrain the noise along the path and consider a reduced SDP. When considering the reduced (sparse) formulation, we only consider monomials of order $2$ that are appearing in the constraints. There are $\mathcal{O}(m+n)$ such monomials. The moments matrix has now size $\mathcal{O}(m+n)\times \mathcal{O}(m+n)$. The part of the certificate expressing first order squares remain unchanged. The second order squares can be written directly from the constraints, trace and constant $\rho$ without the need for any propagation, i.e. $(\z^{\bs \alpha}\z^{\bs \beta} - \z_0^{\bs \alpha}\z_0^{\bs \beta})^2 = -2\z_0^{\bs \alpha}\z_0^{\bs \beta}(\z^{\bs \alpha}\z^{\bs \beta} - \z_0^{\bs \alpha}\z_0^{\bs \beta}) + (\z^{\bs \alpha}\z^{\bs \beta})^2 - (\z_0^{\bs \alpha}\z_0^{\bs \beta})^2$. The certificate thus becomes much sparser. Conditions 1) to 3) still hold for this certificate as it still has the exact same structure $\sum_j \bs s_j\bs s_j^*$ as before except that the number of such squares is reduced. The squared polynomials are now given by $(\z^{\bs \alpha}_0 - \z^{\bs \alpha}_0)^2$, for all $|\bs \alpha|\leq 1$ and $(\z^{\bs \gamma}_0 - \z^{\bs \gamma}_0)^2$ for all $|\bs \gamma|=2$ such that $\z^{\bs \gamma} - \z^{\bs \gamma}_0$ appears in the constraints. As the matrix has now size $|\left\{\bs \alpha, \; |\bs \alpha|\leq 1\right\}| + |\left\{\bs \gamma\;|\;\bs\gamma\in \Omega\right\}|$, the rank condition in 2) still holds as well.

Moreover, from~\eqref{constraintsPathsL2}, we now have for each path $\mathcal{P}_i$, $i=1, \ldots, P$,  
\begin{align}
\sqrt{\sum_{\bs \kappa\in \mathcal{K}_i}\sum_{\ell\in \mathcal{P}_i}\left|\sum_{\bs \zeta} \frac{(\tilde{h}_\ell)_{\bs \zeta}}{\|\bm B_{\bs \zeta+\bs \kappa}\|^2_F}\langle \bm H, \bm B_{\bs \zeta+\bs \kappa}\rangle \right|^2}\leq 2\eta'.
\end{align}
As for the proof of Theorem~\ref{maintheorem}, we let $\bs Y_1^{(1)}$ and $\bs Y_1^{(2)}$, with $\bs Y_1 = \bs Y_1^{(1)} + \bs Y_1^{(2)}$ denote the contributions of first and second order squares to the certificate $\bs Y_1$. To bound the inner product $|\langle \bm H, \bm Y\rangle|$, we once again replace the noiseless constraints appearing in the expression of the SOS certificate with the noiseless constraints that are bounded through~\eqref{constraintsPathsL2}. For the first order contribution $\bs Y_1^{(1)}$, we have, 
\begin{align}
|\langle \bs Y_1^{(1)},\bs  H\rangle | &\leq \left|\sum_{\kappa} \sum_\ell W_{\ell,\kappa} \sum_{\zeta} (\tilde{h}_\ell)_{\zeta} \langle \bs B_{\zeta+\kappa}, \bs M\rangle \right| + \left|\sum_{|\kappa|\leq  1} \sum_\ell W_{\ell,\kappa} \varepsilon_{\ell} \langle \bs B_{\kappa}, \bs M\rangle \right|\\
&\leq \left|\sum_{\kappa} \sum_\ell W_{\ell,\kappa} \sum_{\zeta} (\tilde{h}_\ell)_{\zeta} \langle \bs B_{\zeta+\kappa}, \bs M\rangle \right| + \mathcal{O}((m+n)^{3/2}) \|\bs m_0\| \|\varepsilon\|_\infty\label{boundEpsilon1Cor3}\\
&\leq \mathcal{O}((m+n)^{3/2})\eta' + \mathcal{O}((m+n)^{3/2}) \|\bs m_0\| \|\varepsilon\|_\infty\label{boundEpsilon1Cor3b}
\end{align}
The bound~\eqref{boundEpsilon1Cor3} on the second term follows the exact same reasoning as~\eqref{finalBoundHY1b}. The difference is for the first term, for which we now use the bound on the given path $\mathcal{P}$. For the second term, following the proof of Theorem~\eqref{maintheorem}, noting that we now only use second order moments appearing in the constraints, and using~\eqref{constraintsPathsL2}, we can write 
\begin{align}
|\langle \bs Y_1^{(2)},\bs  H\rangle |& \leq \max\left\{\eta' , \mathcal{O}((m+n)^{3/2})  \|\bs m_0\|  \|\varepsilon\|_\infty\right\}
\end{align}
%
%
%
So that $|\langle \bm H, \bm Y \rangle |\leq (m+n)\max\left\{\mathcal{O}((m+n)^{3/2})\eta' , \mathcal{O}((m+n)^{3/2})  \|\bs m_0\|  \|\varepsilon\|_\infty\right\}$. In a similar way,  the expression for the $y_{i_k}$ in~\eqref{path1NoPath} also relies on the first order constraints making up the path from the root node to $y_{i_k}$ so that the relations and bound in~\eqref{path1NoPath} and~\eqref{step1a} can now be reduced to 
\begin{align}
\|\bs y_{|\bs \alpha|\leq 1}\|_2&\lesssim \sqrt{m+n}\sup_{\mathcal{P}_i}\sum_{\ell\in \mathcal{P}_i}\left|\left(\sum_{|\bs \zeta|>0} \frac{(\tilde{h}_\ell)_{\bs \zeta}}{\|\bm B_{\bs \zeta}\|^2_F}\langle \bm B_{\bs \zeta},\bm H_T\rangle \right)\right|.\label{FirstOrderMomPath}
\end{align}
Relation~\eqref{BoundHigherOrdermomentsHT} still holds. We still only need to account for $(m+n)$ second order monomials corresponding to the $\mathcal{O}(m+n)$ constraints and one can thus simply bound the second order part of $\bs y$ as 
\begin{align}
\|y_{\bs \alpha + \bs \beta}\|_2\displaystyle\lesssim \sqrt{m+n}\sup_{\mathcal{P}_i}\sum_{\ell\in \mathcal{P}_i}\left|\left(\sum_{|\bs \zeta|>0} \frac{(\tilde{h}_\ell)_{\bs \zeta}}{\|\bm B_{\bs \zeta}\|^2_F}\langle \bm B_{\bs \zeta},\bm H_T\rangle \right)\right|  + \sqrt{m+n}\|\bm H_{T^\perp}\|_1\label{SecondOrderMomPath}
\end{align}
The square root in~\eqref{SecondOrderMomPath} comes from the problem-depedent formulation. Grouping~\eqref{FirstOrderMomPath} and~\eqref{SecondOrderMomPath}, the bound on $\bm H_T$ can therefore read,
\begin{align}
\|\bm H_T\|_F&\lesssim \sqrt{m+n}\|\bs m_0\|_2\left(\sup_{\mathcal{P}_i}\sum_{\ell\in \mathcal{P}_i}\left|\left(\sum_{|\bs \zeta|>0} \frac{(\tilde{h}_\ell)_{\bs \zeta}}{\|\bm B_{\bs \zeta}\|^2_F}\langle \bm B_{\bs \zeta},\bm H_T\rangle \right)\right| + \|\bm H_{T^\perp}\|_1\right)\\
&\lesssim \sqrt{m+n}\|\bs m_0\|_2\left(\sup_{\mathcal{P}_i}\sum_{\ell\in \mathcal{P}_i}\left|\left(\sum_{|\bs \zeta|>0} \frac{(\tilde{h}_\ell)_{\bs \zeta}}{\|\bm B_{\bs \zeta}\|^2_F}\langle \bm B_{\bs \zeta},\bm H\rangle \right)\right| +  (m+n)^{1/2} \|\bm H_{T^\perp}\|_1 \right)\\
\label{BT2}\end{align}
Using~\eqref{constraintsPathsL2}, we have 
\begin{align}
\|(\bm H_T)\|_F &\lesssim  \|\bs m_0\|_2(m+n)\max\left\{\mathcal{O}((m+n)^{3/2})\eta' , \mathcal{O}((m+n)^{3/2})  \|\bs m_0\|  \|\varepsilon\|_\infty\right\}
\end{align}
Noting that in the reduced formulation, $\|\bs m_0\|_2 = \mathcal{O}(\sqrt{m+n})$, $\|\bs M_0\|_F = \mathcal{O}(m+n)$ and using $\eta'  =\mathcal{O}(\|\varepsilon\|)$ gives the desired result.

\section{\label{numericalMethodsSec}Numerical methods}

Section~\ref{propagationVSSDPnumerics} starts by providing a comparison of the stability and recovery guarantees of the convex formulation against traditional  approaches such as nuclear norm minimization, nonlinear propagation, and ridge regression. 

Sections~\ref{scalableNumericalImp} through ~\ref{subsamplingAffCons} discuss scalable numerical schemes. Simply listing the moments up to order 4 has complexity $\mathcal{O}(N^4)$ where $N = m+n$, hence is not a scalable representation of the moments matrix. The traditional remedy is the factorized gradient approach due to Burer and Monteiro~\cite{burer,burer2005local}, but our first numerical observation will not be a surprise to the specialist: difficult instances of matrix completion lead to the presence of spurious local minimizers. With adequate compression of the variables and constraints, and provided convergence is to the global minimizer, we show how the problem can be solved in an empirical $\mathcal{O}(N^2)$ complexity. 

The conclusions of section~\ref{numericalMethodsSec} can be summarized as follows.

\begin{itemize}
\item Factorization approaches sometimes introduce spurious minimizers for sufficiently difficult (small $\delta$) problems. When convergence to such minimizers occur, it is sometimes possible to add an additional rounding step and to extract the solution from the second order block rather than considering the whole matrix.

\item Factorizing the moment matrix in low rank form still has storage complexity $\mathcal{O}(N^2)$, hence is not fully scalable. We propose to instead view the moment matrix as a tensor, and upgrade to a more efficient hierarchical low-rank factorization with storage complexity $\mathcal{O}(N)$. This factorization seems to always work when the simpler factorized gradient works. 

\item The hierarchical factorization is in itself not sufficient to guarantee scalability, as formulation~\eqref{momentProblem}, and in particular total symmetry, still requires encoding a combinatorial ($\mathcal{O}(N^4)$) number of constraints. Section~\ref{traceContraction} then introduces three different trace relations, which are derived from the third and fourth order total symmetry constraints. Enforcing those relations in place of the original total symmetry constraints reduces the computational cost required to enforce these constraints from $\mathcal{O}(N^4)$ to $\mathcal{O}(1)$ in the best case. This compression of the total symmetry constraint thus reduces the global complexity to a factor $\mathcal{O}(N^3)$. Empirically, we again observe that those trace relations can be used as a substitute for the more expensive total symmetry constraints as soon as the traditional factorized gradient method works. 


\item Given the $\mathcal{O}(1)$ trace relations and the hierarchical low rank factorization of the moments tensor, a last bottleneck that prevent reducing the global computational cost from $\mathcal{O}(N^3)$ to $\mathcal{O}(N^2)$ is given by the Higher Order Affine constraints which enforce the moments constraints derived from multiplying any of the original constraint by any monomial of degree at most two, to be satisfied. Encoding those constraints requires storing matrices of size $\mathcal{O}(N^3)$ (i.e product of $\bs b\in \mathbb{R}^{N}$) by $\bs m_0\in \mathbb{R}^{N^2}$. We propose to encode these constraints through random sampling, minimizing over distinct batches of size $\mathcal{O}(N)$ iteratively. Such formulation does not seem to modify the convergence properties and enables us to apply the semidefinite program~\eqref{momentProblem} to matrices $\bs X$ of sizes up to $100\times 100$ without making use of the reduced sparsity based formulation of corollary~\ref{noisycaseCorollary}. Dealing with such matrices is not practical in the original framework of the Lasserre hierarchy with two rounds of lifting.

\end{itemize}

\subsection{\label{propagationVSSDPnumerics}Lipschitz stability}

To illustrate how the noise can affect a nonlinear reconstruction in the propagation framework, we conduct the following experiments. We consider a noise vector $\bs \varepsilon =\gamma  \bs n/\| \bs n\|$ for $ \bs n\sim \mathcal{N}(\bs 0,\bs I_3)$. We gradually increase the amplitude $\gamma$ of the noise vector. For those noise vectors, we let $\bs M_P$ denote the solution obtained through propagation and $\bs M_L$ the solution obtained through the stable semidefinite relaxation~\eqref{stableSDPnoisy}. We consider the matrix of example~\eqref{matrixEpsilon} for which we let $\delta=.01$. The numerical experiments are then repeated as follows.

\begin{itemize}	
\item We randomly draw the noise vector $ n\sim \mathcal{N}(\bs 0,\bs I_{3})$. 
\item The noise vector is multiplied by the scaling coefficient $\gamma$ taking values between $.001$ and $.01$, so that the corruption is at most $100\%$ of the signal. The noise is added to the entries $(X_0)_{11}$, $(X_0)_{22}$ and $(X_0)_{21}$ of $\bs X_0$ to define the (noisy) measurements.
\item  Our semidefinite programming relaxation is then solved with \textsc{cvx}\footnote{\texttt{http://cvxr.com/about/}} for the noisy measurements. We compute the difference between the returned solution $\bs M_L$ and the optimal solution to the noiseless problem $\bs M_0$ through the Frobenius norm as $\|\bs M_L - \bs M_0\|_F/\|\bs M_0\|_F$.
\item The equivalent solution obtained through nonlinear propagation is computed and compared to $M_0$ as $\|\bs M_P -  \bs M_0\|_F/\|\bs M_0\|_F$.

\end{itemize}
Those various steps are repeated for the various noise levels and for a collection of random vectors $ n$. Note that, because we consider example~\eqref{matrixEpsilon}, nuclear norm fails even in the absence of noise. For each choice of $\gamma$ the relative errors $\|\bs M_P - \bs M_0\|/\|\bs M_0\|$ and $\|\bs M_L - \bs M_0\|_F/\| \bs M_0\|_F$ are averaged over all the noise vectors. The results  are shown in Fig.~\ref{MCvsProp}. This figure thus illustrates the evolution  of the averaged relative errors $\mathbb{E}_{\bs n}\|\bs M_L - \bs M_0\|/\|\bs M_0\|$ and $\mathbb{E}_{\bs n}\| \bs M_P - \bs M_0\|_F /\|\bs M_0\| _F$ for our semidefinite programming relaxation, as well as for nonlinear propagation in an instance where nuclear norm minimization fails. The Figure on the Right is truncated above to enable the comparison between both figures. The relative errors corresponding to low signal to noise ratio (SNR) were otherwise rising above $200\%$. The SNR is measured in [dB] as $20\log(\varepsilon/\gamma)$.

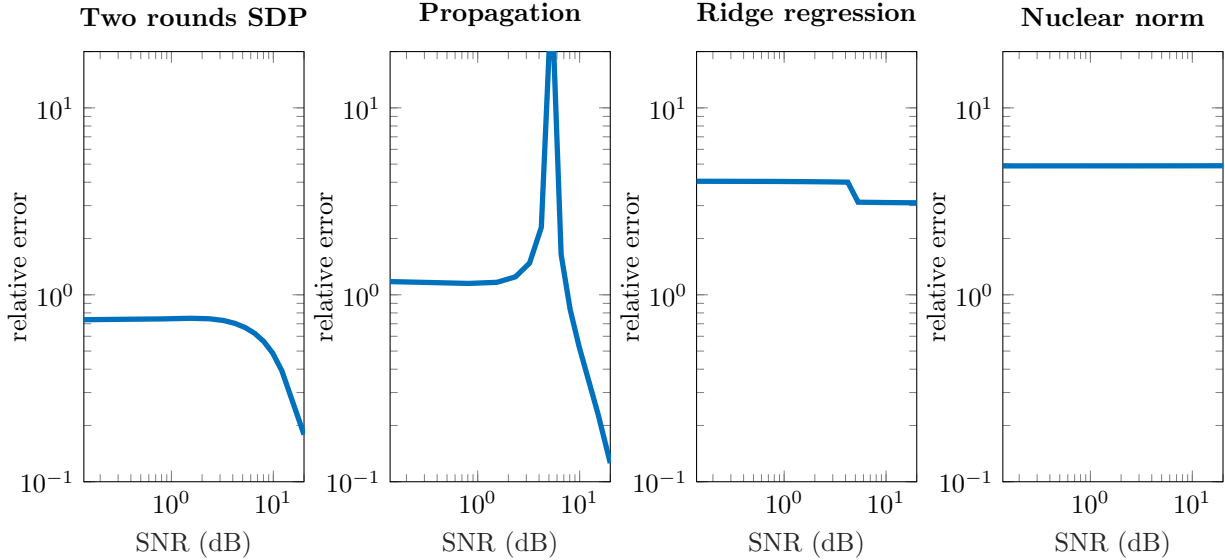
\begin{figure}[h!]\centering
%
%
\definecolor{mycolor1}{rgb}{0.00000,0.44700,0.74100}%
\begin{tikzpicture}

\begin{axis}[%
width=2.881in,
height=5.633in,
at={(1.393in,0.731in)},
scale only axis,
xmode=log,
xmin=0.138239888326602,
xmax=20,
xminorticks=true,
xlabel style={font=\color{white!15!black}},
xlabel={SNR (dB)},
ymode=log,
ymin=0.1,
ymax=20,
yminorticks=true,
ylabel style={font=\color{white!15!black}},
ylabel={relative error},
axis background/.style={fill=white},
title style={font=\bfseries},
title={Two rounds SDP},scale=.4,ylabel near ticks, y label style={at={(-0.22,0.5)}},
]
\addplot [color=mycolor1, line width=2.0pt, forget plot]
  table[row sep=crcr]{%
20	0.178740851921305\\
15.2047932214988	0.275005530382252\\
12.1331148603422	0.391960385082545\\
9.86847531884124	0.487115679809755\\
8.07384675122258	0.562013721797218\\
6.58727188615832	0.62076817795617\\
5.31832752495314	0.667030272912008\\
4.21135478413335	0.703061941519455\\
3.2296461059413	0.729168965762864\\
2.34771197435708	0.744567567795422\\
1.54712953264755	0.748805685280375\\
0.814149956480673	0.743231387112018\\
0.138239888326602	0.735849539798331\\
-0.488849129353197	0.752509280404121\\
-1.07369717925553	0.782831882427249\\
-1.62163762774118	0.775507847677809\\
-2.13705345290966	0.784628751275199\\
-2.62359044756931	0.799986622312457\\
-3.08431379843153	0.816056769560004\\
-3.52182518111362	0.830756125261068\\
};
\end{axis}

\begin{axis}[%
width=2.881in,
height=5.633in,
at={(2.996in,0.731in)},
scale only axis,
xmode=log,
xmin=0.138239888326602,
xmax=20,
xminorticks=true,
xlabel style={font=\color{white!15!black}},
xlabel={SNR (dB)},
ymode=log,
ymin=0.1,
ymax=20,
yminorticks=true,
ylabel style={font=\color{white!15!black}},
ylabel={relative error},
axis background/.style={fill=white},
title style={font=\bfseries},
title={Propagation}
,scale=.4,ylabel near ticks, y label style={at={(-.22,0.5)}}
]
\addplot [color=mycolor1, line width=2.0pt, forget plot]
  table[row sep=crcr]{%
20	0.12579930703916\\
15.2047932214988	0.229030681714786\\
12.1331148603422	0.354923198099704\\
9.86847531884124	0.530798001438384\\
8.07384675122258	0.835001935102578\\
6.58727188615832	1.64316851665472\\
5.31832752495314	41.0843877463453\\
4.21135478413335	2.29551601133221\\
3.2296461059413	1.47775351057757\\
2.34771197435708	1.2479709097181\\
1.54712953264755	1.16629751287035\\
0.814149956480673	1.15031033911491\\
0.138239888326602	1.17672437891902\\
-0.488849129353197	1.24159905768106\\
-1.07369717925553	1.35342018064814\\
-1.62163762774118	1.53793484528214\\
-2.13705345290966	1.86180476139394\\
-2.62359044756931	2.53486576986809\\
-3.08431379843153	4.67988241717611\\
-3.52182518111362	51.9710778099863\\
};
\end{axis}

\begin{axis}[%
width=2.881in,
height=5.633in,
at={(4.6in,0.731in)},
scale only axis,
xmode=log,
xmin=0.138239888326602,
xmax=20,
xminorticks=true,
xlabel style={font=\color{white!15!black}},
xlabel={SNR (dB)},
ymode=log,
ymin=0.1,
ymax=20,
yminorticks=true,
ylabel style={font=\color{white!15!black}},
ylabel={relative error},
axis background/.style={fill=white},
title style={font=\bfseries},
title={Ridge regression}
,scale=.4, ylabel near ticks, y label style={at={(-0.22,0.5)}},
]
\addplot [color=mycolor1, line width=2.0pt, forget plot]
  table[row sep=crcr]{%
20	3.10599270394253\\
15.2047932214988	3.10872781324128\\
12.1331148603422	3.11185963866312\\
9.86847531884124	3.11541467199597\\
8.07384675122258	3.11928454876823\\
6.58727188615832	3.1235705721347\\
5.31832752495314	3.12821233410445\\
4.21135478413335	4.00341837035117\\
3.2296461059413	4.01247071088839\\
2.34771197435708	4.02155968234711\\
1.54712953264755	4.0306649133933\\
0.814149956480673	4.03980761769196\\
0.138239888326602	4.04895419538938\\
-0.488849129353197	4.05811578470394\\
-1.07369717925553	4.06727437138396\\
-1.62163762774118	4.07642745848094\\
-2.13705345290966	4.08556315079389\\
-2.62359044756931	4.09466263774661\\
-3.08431379843153	4.10375716204813\\
-3.52182518111362	4.87283576751572\\
};
\end{axis}

\begin{axis}[%
width=2.881in,
height=5.633in,
at={(6.203in,0.731in)},
scale only axis,
xmode=log,
xmin=0.138239888326602,
xmax=20,
xminorticks=true,
xlabel style={font=\color{white!15!black}},
xlabel={SNR (dB)},
ymode=log,
ymin=0.1,
ymax=20,
yminorticks=true,
ylabel style={font=\color{white!15!black}},
ylabel={relative error},
axis background/.style={fill=white},
title style={font=\bfseries},
title={Nuclear norm}
,scale=.4, ylabel near ticks, y label style={at={(-0.22,0.5)}},
]
\addplot [color=mycolor1, line width=2.0pt, forget plot]
  table[row sep=crcr]{%
20	4.9000049105361\\
15.2047932214988	4.89929746145258\\
12.1331148603422	4.89860629307985\\
9.86847531884124	4.89792571240399\\
8.07384675122258	4.89725928083358\\
6.58727188615832	4.89660529477442\\
5.31832752495314	4.89596547083114\\
4.21135478413335	4.89533957294089\\
3.2296461059413	4.89472934734964\\
2.34771197435708	4.89413324053469\\
1.54712953264755	4.89355124880793\\
0.814149956480673	4.89298147580784\\
0.138239888326602	4.8924246650713\\
-0.488849129353197	4.89188428749171\\
-1.07369717925553	4.8913569925797\\
-1.62163762774118	4.89084710233335\\
-2.13705345290966	4.89034922107522\\
-2.62359044756931	4.88986540730652\\
-3.08431379843153	4.88939503821367\\
-3.52182518111362	4.88894097802459\\
};
\end{axis}
\end{tikzpicture}%
\caption{\label{MCvsProp} Evolution of the relative error $\|\bm M-\bm M_0\|_F/\|\bm M_0\|_F$ as a function of the noise level (SNR [dB]) for the semidefinite program; nonlinear propagation; ridge regression; and nuclear norm minimization. Blowup can occur for nonlinear propagation whenever the noise takes on values that are close, yet opposite in sign to the small entries in the matrix.
Both ridge regression and nuclear norm minimization are known to fail, even in the absence of noise.}
\end{figure}

\subsection{\label{scalableNumericalImp}Toward scalability: low-rank factorization}

Despite its interest in terms of stability, the semidefinite program~\eqref{noisy} remains difficult to implement for practical problems because of the size of the second order moments matrix involved. Solving the completion problem on a matrix of size $N\times N$ through~\eqref{noisy} requires storing a matrix of size $N^4$ which is often out of reach for typical numerical solvers, on sufficiently interesting instances. In this section we introduce and discuss more scalable numerical methods based on low rank factorizations of the moment matrix~\eqref{secondRound}.  As is usual in semidefinite programming, the recovery guarantees are however lost when passing to such formulations. This phenomenon is illustrated by Fig~\ref{landscapeStudy1} to~\ref{landscapeStudy4}. 


\begin{figure}\centering
\input{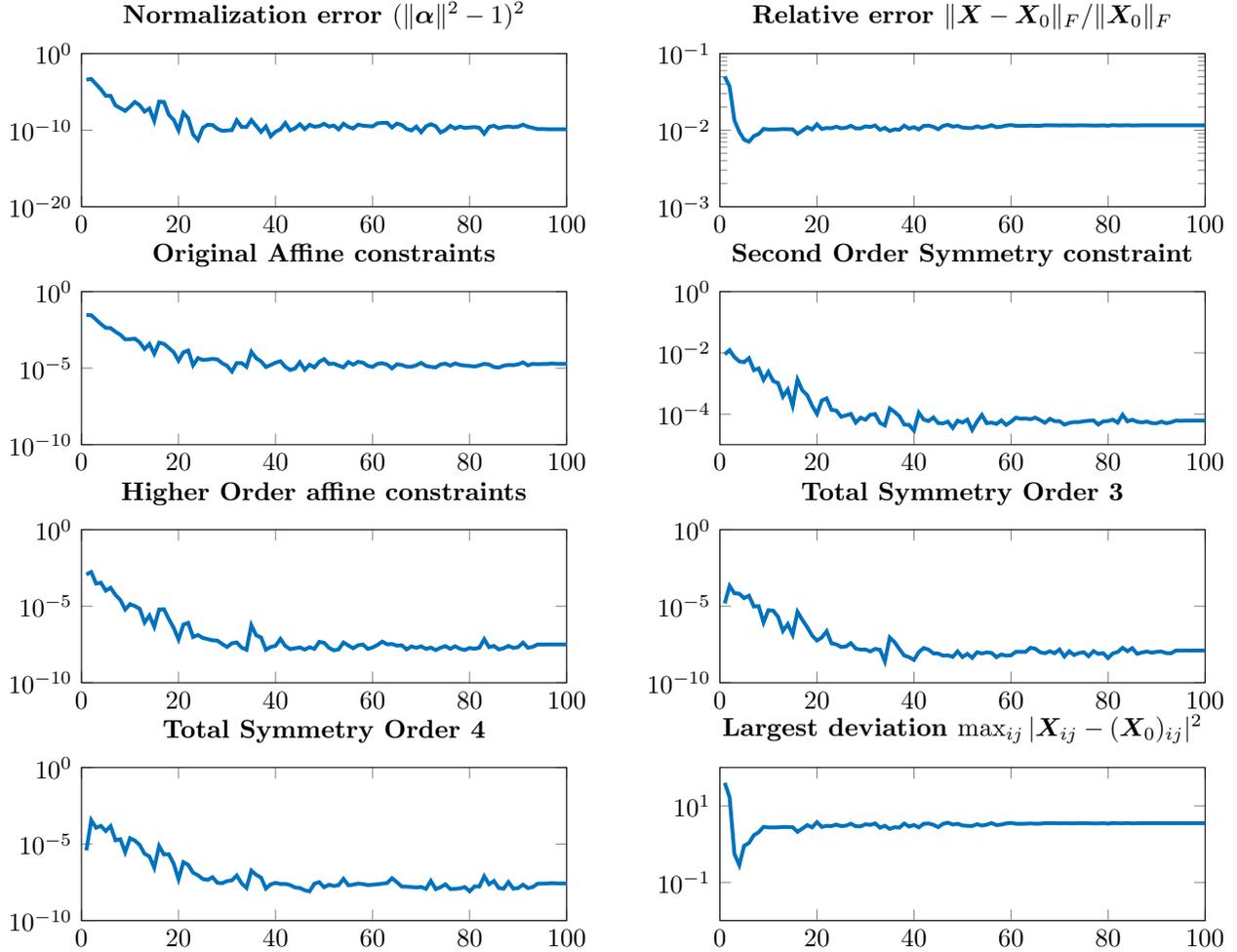}
\caption{\label{landscapeStudy1}Evolution of each of the error terms appearing in the augmented Lagrangian formulation~\eqref{augmentedLagrangian} on the benchmark problem~\eqref{matrixEpsilon} with $\delta = 0.3$. The relative error (top right) together with the largest deviation highlight the convergence to a local minimizer, while all the other error terms, certifying feasibility, have already reached small thresholds. Those plots can also be compared to the evolution of the trace and global misfit shown in Fig.~\ref{landscapeStudy2} and the comparison of the structures of the global minimizer with the returned local minimizer shown in Fig.~\ref{landscapeStudy4}.}  
\end{figure}

\begin{figure}\centering
\input{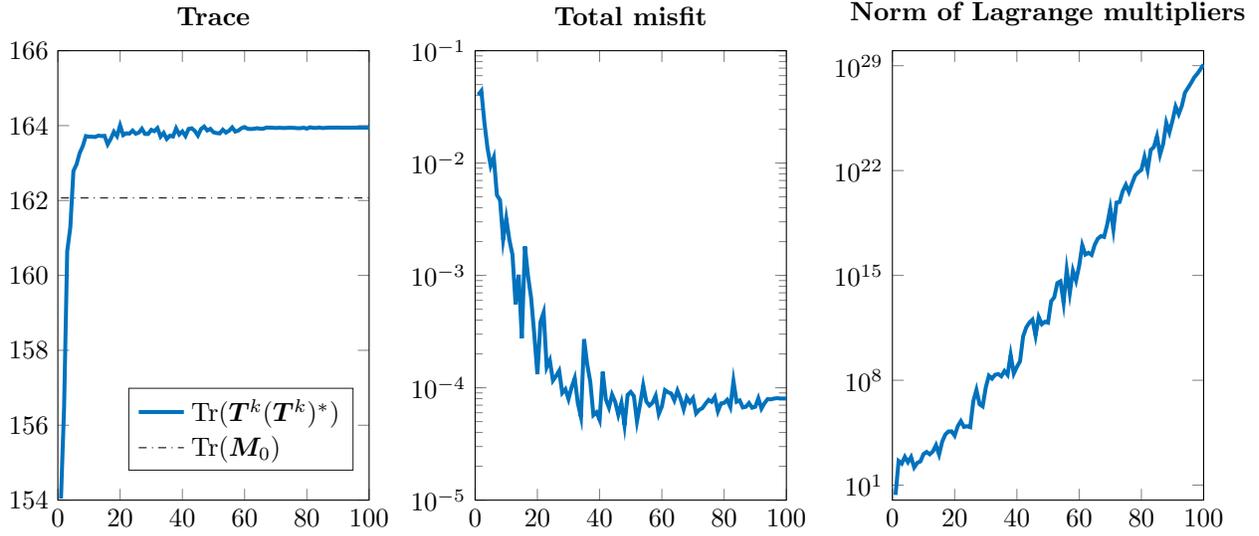}
\caption{\label{landscapeStudy2}Evolution of the Trace, misfit and norm of the dual Lagrange multipliers. The slow blowup in the norm of the those multipliers confirms that we leave the regime in which Theorem 5.4 in~\cite{burer2005local} works, and thus loses the recovery guarantees for the factorized gradient formulation that are following from this theorem.}  
\end{figure}

\begin{figure}\centering
\includegraphics[width=\linewidth, trim = 3cm 4cm 2cm 4cm, clip=true]{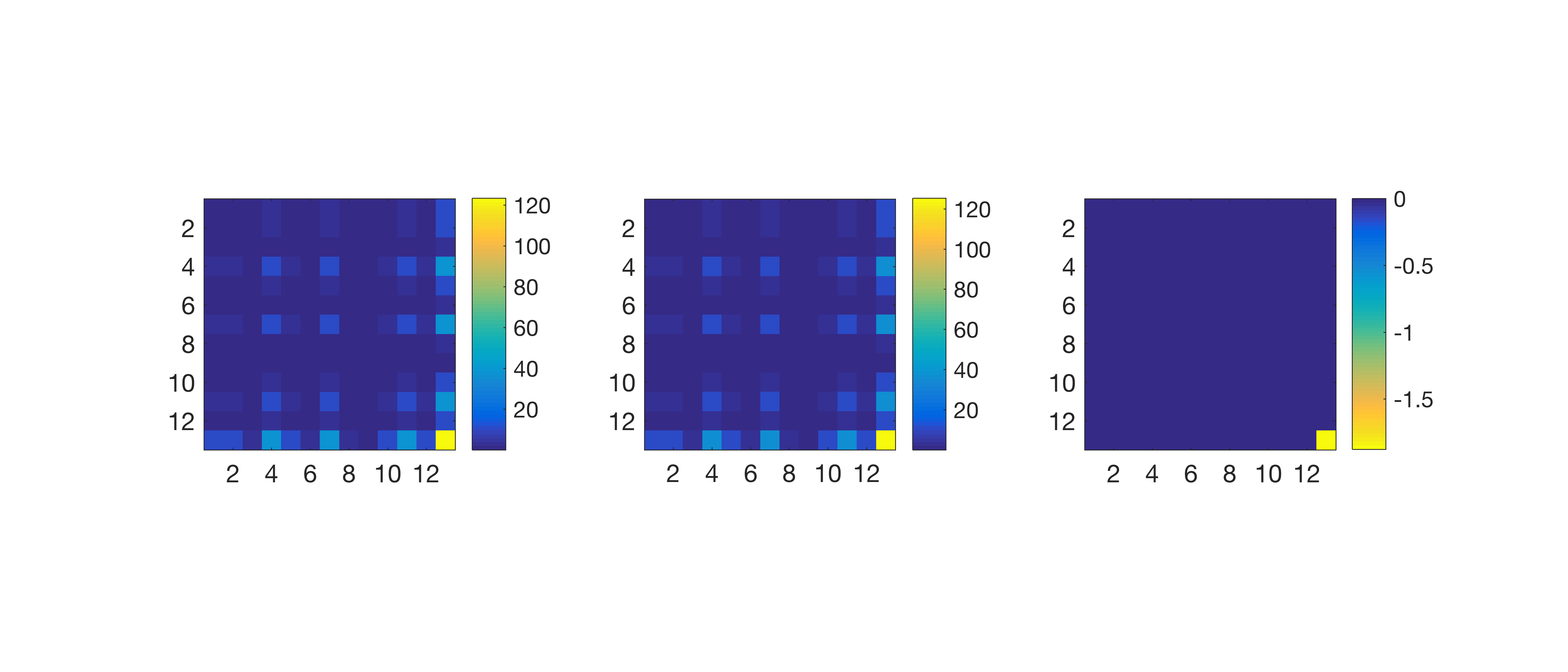}
\caption{\label{landscapeStudy4} 
Despite the existence of local minima for sufficiently difficult formulations (small $\delta$), the solutions returned by the minimization of the factorized augmented Lagrangian, or hierarchical low rank Lagrangian sometimes remain relatively close to the global solution except for the highest order moments. This figures illustrates this phenomenon. The exact moments matrix $\bs M_0$ for the simple example of~\eqref{matrixEpsilon} with $\delta=0.3$ is shown on the left. The moments matrix $\bs M$ recovered through low rank factorization is shown in the center and the difference $\bs M-\bs M_0$ is shown on the right. In practice, for a problem that is not too difficult (i.e. a sufficiently large value of $\delta$ in~\eqref{matrixEpsilon}) it is thus empirically possible to recover the global solution by simply extracting the lower order moments. }
\end{figure}

Among the most popular approaches of the last few years, one the of the most efficient, popularized by~\cite{burer2003nonlinear} encodes the unknown positive semidefinite matrix $\bs M$ from~\eqref{momentProblem} as a low rank factorization $\bs M \approx \bs T\bs T^*$, with $\bs T$ of size $N + N^2$ by $r$ for small $r$, and then minimizes the augmented Lagrangian over the factor $\bs T$. Note that in our case $\bs T$ is of the form $\bs T =[\bs \alpha^{(k)}, \bs R^{(k)}, \bs \Pi^{(k)}]$ where $\bs \alpha^{(k)}\in \mathbb{R}^r$ encodes the normalizing constant, $\bs R^{(k)}\in \mathbb{R}^{n\times r}$ and $\bs \Pi\in \mathbb{R}^{n^2\times r}$. We further let $\bs \Pi^{(k)} = [\bs \Pi^{(k)}_1,\ldots \bs \Pi^{(k)}_r]$ where we use $\bs \Pi^{(k)}_j \in \mathbb{R}^{N^2}$ to denote each of the full rank matrices of size $N\times N$ making up the factors in the low rank factorization of the matricization of the fourther order moments tensor. The moments matrix $\bs M$ then reads,
\begin{align}
\bs M = \left[\begin{array}{c}
\bs \alpha\\
\bs R\\
\bs \Pi
\end{array}\right]\left[\begin{array}{ccc}
\bs \alpha^T,\quad 
\bs R^T, \quad \bs \Pi^T
\end{array}\right] = \bs T\bs T^T,\label{LowrankEncoding}
\end{align}
In the factorization above, $\bs X$ and $\bs X^T$ are meant to appear as off-diagonal blocks of $\bs R\bs R^T$. The rank of each of the factors can be constrained, and is increased, when reaching local minimizers. If we let $r$ denote the rank of the compressed matrix $M$, such a formulation thus results in a reduction of the number of unknown from $\mathcal{O}(N^2)$ to only $\mathcal{O}(Nr)$ unknowns. For a set of constraints defined as $\langle \bs A_i, \bs X\rangle = b_i $ and encoded in the linear map $\mathcal{A}: \bs X\mapsto \mathcal{A}(\bs X) = \{\langle \bs A_i, \bs X\rangle \}_{i=1}^m$, a vector of multipliers $\lambda \in \mathbb{R}^m$ and penalty term $\sigma\in \mathbb{R}$, the augmented Lagrangian function corresponding to a minimization of the trace under the linear constraints $\mathcal{A}(\bs T\bs T^*) = b$ reads 
\begin{align}
\mathcal{L}(\bs T,\lambda,\sigma) = \|\bs T\|_F^2 -2 \sum_{i=1}^m \lambda_i \left(\langle \bs A_i, \bs T\bs T^*\rangle - b_i\right) + \sigma \sum_{i=1}^m|\langle \bs A_i,\bs T\bs T^*\rangle - b_i|^2  \label{augmentedLagrangian}
\end{align}
For some initial guess $\bs T^{(0)}$, we let $v_0$ be initialized as $v_0 = \rho(\bs T^{(0)}(\bs T^{(0)})^*) =\sum_{i=1}^m (\langle \bs A_i, \bs T^{(0)}(\bs T^{(0)})^T\rangle - b_i )^2$. Finally set $k$ to $0$. The augmented Lagrangian algorithm iteratively minimizes the Lagrangian over $\bs T$ (step 1) and updates the multipliers (step 2) according to the following rule (see~\cite{burer2003nonlinear}). Let $\rho$ denote the norm of the vector of residuals following from step 1, $\rho :=\sum_{i=1}^m (\langle \bs A_i, \bs T^{(k)}(\bs T^{(k)})^T\rangle - b_i )^2$. We set $\gamma = 2$ (when dealing with more difficult cases, this update parameter should be increased) and $\eta = 0.25$. If $\rho < \eta v_k$, $y^{k+1} \leftarrow y^k - \sigma^k (\mathcal{A}(\bs T\bs T^*) - b)$, $\sigma^{k+1} \leftarrow \sigma^k$, and $v_{k+1}\leftarrow \rho$. Otherwise, $y^{k+1} \leftarrow y^k$, $\sigma^{k+1} \leftarrow\gamma \sigma^k$, $v_{k+1}\leftarrow v_k$. Finally set $k \leftarrow k+1$ and repeat step 1.

When dealing with problems~\eqref{momentProblem} and~\eqref{noisy} we should favor penalty formulations over Lagrangian formulations, as the number of symmetry constraints is combinatorial in the dimension and therefore requires large vectors of multipliers. The convex formulation~\eqref{stableSDPnoiseless} then turns into
\begin{align}\begin{split}
\min \quad & \|\bs T\|^2_F + \sigma \left\| \mathcal{A}(\bs T \bs T^*) - \bs b\right\|^2
\end{split}\label{lowRankFormulationPenaltyMethod}
\end{align}
%




In difficult cases (e.g., when $\delta$ is sufficiently small), convergence of iterative methods can suffer for~\eqref{lowRankFormulationPenaltyMethod} and the Lagrangian formulation is thus more appropriate. A hybrid formulation, intermediate between the penalty formulation~\eqref{lowRankFormulationPenaltyMethod} and the more expensive Lagrangian~\eqref{augmentedLagrangian} is to consider an incomplete set of Lagrange multipliers. In the rest of this section, we will focus on making formulation~\ref{augmentedLagrangian} more tractable.

\subsection{\label{hierarchicalLR}From low rank to hierarchical low rank}
When considering large matrices, such as used by the Lasserre hierarchy, even wen using rank constrained factorization, an optimization framework such as~\eqref{augmentedLagrangian} with the factorization~\eqref{LowrankEncoding} still requires storing $\mathcal{O}(N^2r)$ unknowns. It is however possible to factorize in low-rank form the higher order blocks in~\eqref{secondRound}. This leads to a multi-level or \textit{hierarchical encoding} of the moments matrix underlying the Lasserre/sos hierarchies. Within the completion framework, it means that each of the factors $\bs \Pi^{(k)}_\ell$, $\ell=1,\ldots, r_1$ can be encoded as a symmetric low rank factorization. This idea is known as hierarchical Tucker decomposition in tensor analysis.

The hierarchical factorization thus relies on two dynamic ranks. The first rank $r_1$ controls the factorization of the moments matrix as a whole. The second rank, $r_2$ controls the factorization of the fourth order tensors $\bs \Pi_{\ell}$ which are thus stored as the tuples $\{\bs S_{k,\ell}\}_{(k,\ell)\in [r_2]\times[r_1]}$, i.e., 
\begin{align}
\bs \Pi_{\ell} = \sum_{k=1}^{r_2} \bs S_{k,\ell}\bs S_{k,\ell}^T,\qquad \ell=1,\ldots, r_1.\label{hierarchicalEncoding}
\end{align}
Optimization is then performed on the augmented Lagrangian obtained by substituting this nested low rank factorization. The power of the hierarchical low-rank idea lies in its scalability, and the fact that it can be applied recursively to higher-degree moment matrices, thus potentially enabling scalable optimization over higher rounds of semidefinite programming hierarchies. Function and gradient derivation are given for the hierarchical factorization on the penalty formulation~\eqref{lowRankFormulationPenaltyMethod} in appendix~\ref{gradientFunctionCalculation}. The derivations on the augmented Lagrangian formulation follow the exact same idea. The next section discusses how the combinatorial total symmetry constraints can be enforced efficiently.

\subsection{\label{traceContraction}Replacing total symmetry with trace relations}

In this section, we discuss three trace relations whose linearizations can be used as scalable substitutes to the more computationally expensive third and fourth order total symmetry constraints. We provide numerical evidence that whenever the factorized gradient method works, enforcing those trace relations in place of total symmetry works just as well, yet reduces the computational (combinatorial) cost of those constraints from $\mathcal{O}(N^4)$ to $\mathcal{O}(1)$. Those relations seem to work best when applied to the multilevel low rank decomposition introduced in section~\ref{hierarchicalLR}. 

The total symmetry constraints are used to encode correspondence of the entries of $\bs M$ that correspond to the same moments (see the discussion in section~\ref{introductionToAlgo}). When applied on the third and fourth moments tensors, those constraints enforce equality between any permutation of the multi-index. I.e if $\bs M^{(3)}$ and $\bs M^{(4)}$ encode the third and fourth order blocks in $\bs M$, then those constraints require that for any $3$-tuple $(i,j,k)$ and permutation $\pi$, $(\bs M^{(3)})_{i,j,k} = (\bs M^{(3)})_{\pi(i,j,k)}$. Similarly, on the fourth order block, for any $4$-tuple $(i,j,k,\ell)$ and any permutation $\pi$, the moments matrix must satisfy $(\bs M^{(4)})_{i,j,k\ell} = (\bs M^{(4)})_{\pi(i,j,k\ell)}$.

One of the implications of total symmetry constraints is that the contraction of any fourth order block does not depend on the indices over which this contraction is taken. In other words, the sum $\sum_{i=1}^N \bs M^{(4)}_{iijk}$ is the same as the sum $\sum_{i=1}^N \bs M^{(4)}_{jiik}$, and so is it for any of the sums $\sum_{i=1}^N \bs M^{(4)}_{\pi(iijk)}$ for any permutation operator $\pi:[N]\mapsto [N]$. When assuming that the tensor $\bs M^{(4)}$ is rank one, that is $\bs M^{(4)}_0= \mbox{vec}(\z_0 \otimes \z_0) \mbox{vec}(\z_0\otimes \z_0)^T$, those constraints can be used to derive interesting trace relations on the second order tensor $\bs M^{(2)}$. For $\bs M^{(4)} = \bs M^{(2)}\otimes \bs M^{(2)}$, $\sum_{i=1}^N \bs M^{(4)}_{iijk} = \sum_{i=1}^N \bs M^{(4)}_{jiik}$ in particular implies the following trace relation on $\bs M^{(2)}$,
\begin{align}
\tr(\bs M^{(2)})\bs M^{(2)} = (\bs M^{(2)})^2.\label{traceRelation1}
\end{align}
Linearizing this trace relation brings us back to enforcing equality of the contractions $\sum_{i=1}^N \bs M^4_{iijk} - \sum_{i=1}^N \bs M^4_{jiik} =0$. Moreover this first contraction can be enforced very efficiently on the (hierarchical) low rank factorization of $\bs M^{(4)}$, $\bs M^{(4)} = \sum_{k=1}^r \bs S_k\bs S_k$,
\begin{align}
\sum_{k=1}^{r_1} \tr(\mata(\bs \Pi_{k}))\mata(\bs \Pi_{k}) = \sum_{k=1}^r \mata(\bs \Pi_{k}) \mata(\bs \Pi_{k}) 
\end{align}
%
%
%
The natural extension to~\eqref{traceRelation1} is to go one step further and take a second contraction with respect to the indices remaining in this first constraint. This gives a second trace relation that requires the trace of the squared matrix to match the square of this matrix trace,  
\begin{align}
\tr(\bs M^{(2)})^2 = \tr((\bs M^{(2)})^2).\label{traceRelation2}
\end{align}
This last relation reduces the set of $\mathcal{O}(N^4)$ symmetry constraints to a single constraint that can be enforced efficiently on the hierarchical low rank factors. Note that when enforced on positive semidefinite matrices,~\eqref{traceRelation2} is in fact equivalent to enforcing an exact rank one constraint, as it requires $(\sum_i \lambda_i)^2 = \sum_i \lambda_i^2$, for $\lambda_i\geq 0$,
\begin{align}
\left\{X\in \mathbb{S}_N^+\;:\;\tr(\bs X)^2 = \tr(\bs X^2)\right\} =\left\{\bs X\in \mathbb{S}_N^+\;:\;\mbox{rank}(\bs X)\leq 1\right\}
\end{align}
Again, Equation~\eqref{traceRelation2} can be written compactly for the low rank as well as for the hierarchical low rank formulations. For this last factorization, we get 
\begin{align}
\sum_r \|\bs S_r\|^2_F \bs S_r\bs S_r^* &=\sum _r \bs S_r (\bs S_r^*\bs S_r)\bs S_r\\
& = \sum_r \bs S_r \bs \Delta_{r,r}\bs S_r^*
\end{align}
where we let $\bs \Delta_{r,r}$ denote the $r$ by $r$ matrix encoding the products $\bs S_r^*\bs S_r$. 

 

As the fourth order symmetry constraints are not as important as a means to express the first order monomials as the third order symmetry constraints (especially in the reduced framework of Corollary~\ref{noisycaseCorollary}), one could argue that replacing those constraints with a simpler contraction does not have a significant impact on the outcome of the relaxation. It is in fact possible to consider a trace contraction for third order total symmetry constraints as well. At order $3$, following from the constraints $\bs M^{(3)}_{ijk} = \bs M^{(3)}_{\pi_1(i,j,k)}$, one possible contraction can be taken over the first $2$ indices in the third order tensor $\bs M^{(3)}$,  for any permutation $\pi$. This gives the following relation
\begin{align}
\tr(\bs M^{(2)})\bs M^{(1)} = \bs M^{(2)}\bs M^{(1)}\label{traceRelation3}
\end{align}
where we again let $\bs M^{(2)}$ denote the matrix encoding the second order moments $\bs M^{(2)} \approx \bs M^{(1)}\otimes \bs M^{(1)}$ and $\bs M^{(1)}$ denotes the vector of first order moments. Again, this third relation can be expressed compactly for both the low rank and hierarchical low rank formulations. For this last factorization, we can write
\begin{align}
\sum_{r=1}^{r_1} \tr(\mata(\bs \Pi_r)) (\bs R_r) = \sum_{r=1}^{r_1} \mata(\bs \Pi_r)\bs R_r
\end{align}
%
In each of these examples, we provide the evolution of each of the error terms appearing in the Lagrangian~\eqref{augmentedLagrangian} with the iterations. We also represent the global misfit, the trace and the evolution of the Lagrange multipliers as the main recovery guarantees provided so far on low rank factorization require those multipliers to remain bounded. 

To study the result of replacing third and fourth order total symmetry constraints by the trace relations above, we apply those relations on the simple example~\eqref{matrixEpsilon} for $\delta = .5$ with the single low rank and hierarchical low rank factorization. The results are shown in Figs.~\ref{traceComparisonTotalLRa} (low rank) and~\ref{hierarchicalLowRankandTraceA} (hierarchical low rank). When considering the simpler factorized gradient approach, it seems that replacing the full set of $4^{th}$ order symmetry constraints with the corresponding trace contraction can lead to a slight reduction in the accuracy. The total symmetry constraints are not entirely satisfied as highlighted by Fig.~\ref{traceComparisonTotalLRa} and this results in a partial recovery of the fourth order tensor. It remains possible to extract the solution $\bs X_0$ from the second order moments. A comparison of the iterations of Fig.~\ref{traceComparisonTotalLRa} and~\ref{hierarchicalLowRankandTraceA} seems to suggest that replacing total symmetry by the relations~\eqref{traceRelation1},~\eqref{traceRelation2} and~\eqref{traceRelation3} performs best when used on the hierarchical low rank factorization.

Generally speaking, it again seems that when the factorized gradient approach converges, which typically happens on problems that are not too difficult (i.e $\delta$ not too small), it always seems possible to replace the combinatorial Total Symmetry constraints by the more tractable trace contractions on both the $3^{rd}$ and $4^{th}$ order tensors, and to recover the solution for both the low rank and hierarchical low rank frameworks. As we don't have empirical evidence that choosing contraction~\eqref{traceRelation1} over contraction~\eqref{traceRelation2} will lead to better convergence properties, we will always favor the former over the latter, as this one reduces to a single equation. On the remaing large scale examples of this paper, we thus always replace total symmetry constraints with relation~\ref{traceRelation3} (third order moments) and relation~\eqref{traceRelation2} (fourth order moments).

To illustrate the interest of the combination of a multi-level low rank decomposition and of the trace relations~\eqref{traceRelation2} and~\eqref{traceRelation3} for large rank one recovery problems, we now apply this combination on a first large scale example. On this example, both nuclear norm and ridge regression fail at recovering the solution. For this example, we take the moments to be bounded as $\delta \leq \z_0^{\bs \alpha} \leq 1$, for any $|\bs \alpha|\leq 1$ and with $\delta=.1$ (i.e with a possibly larger gap between smallest and largest entries). The bipartite graph defining the measurements is represented in Fig.~\ref{bipartiteGraphTikhonovFail}. This graph is generated at random  while enforced to span the $(m+n)$ vertices with a minimal number of edges. For this particular problem, the solution returned by nuclear norm minimization gives a relative error $\|\bs X-\bs X_0\|_F/\|\bs X_0\|_F = 0.7382$. The iterates returned by the ridge regression formulation~\eqref{quartic} are displayed in Fig.~\ref{TikhoFailure1} (relative error and data misfit. In this case there is no need for any regularization as the problem is noiseless). 

The iterations following from optimization over the hierarchical low rank augmented Lagrangian with the trace relations~\eqref{traceRelation2} and~\eqref{traceRelation3} is displayed in Fig.~\ref{LasserreSucceedsWhereTikhoFails}.

%
%
%

\begin{figure}\centering
\input{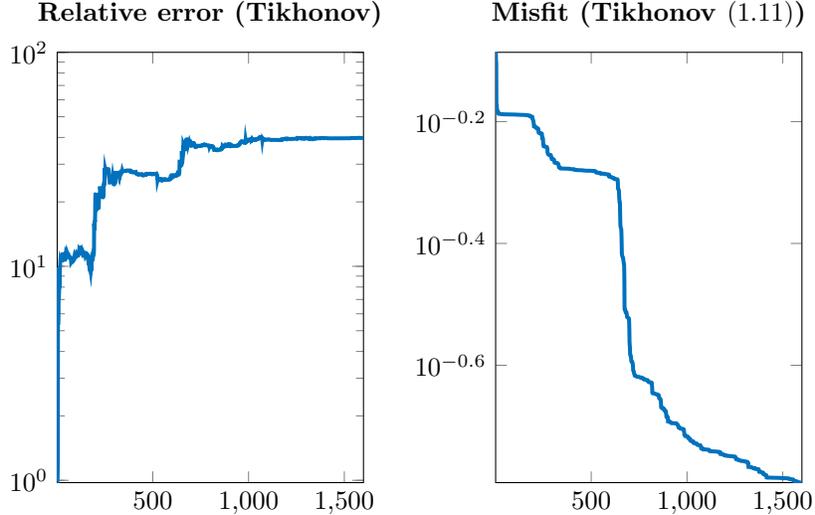}
\caption{\label{TikhoFailure1}It is possible to find instances of the rank one matrix completion problem for which even the ridge regression formulation will not be able to return the global minimizer. Such instances are more frequent when considering matrices of large size. In this particular case, ridge regression is clearly shown to converge to the wrong minimizer. The higher semidefinite relation, on the other hand, returns the true solution, even when considering a hierarchical low rank factorization with strong constraint on the rank ($2$ in this case) (see Fig.~\ref{LasserreSucceedsWhereTikhoFails}). The matrix considered here is $20\times 20$ with moments bounded as $\delta \leq (X_0)_{ij}\leq 1$ for $\delta = 0.01$ and a mask $\Omega$ whose underlying bipartite graph is shown in Fig.~\ref{bipartiteGraphTikhonovFail}. The corresponding rank constrained iterations for the higher order relaxation are displayed in Fig.~\ref{LasserreSucceedsWhereTikhoFails}.}
\end{figure}

\begin{figure}\centering
\input{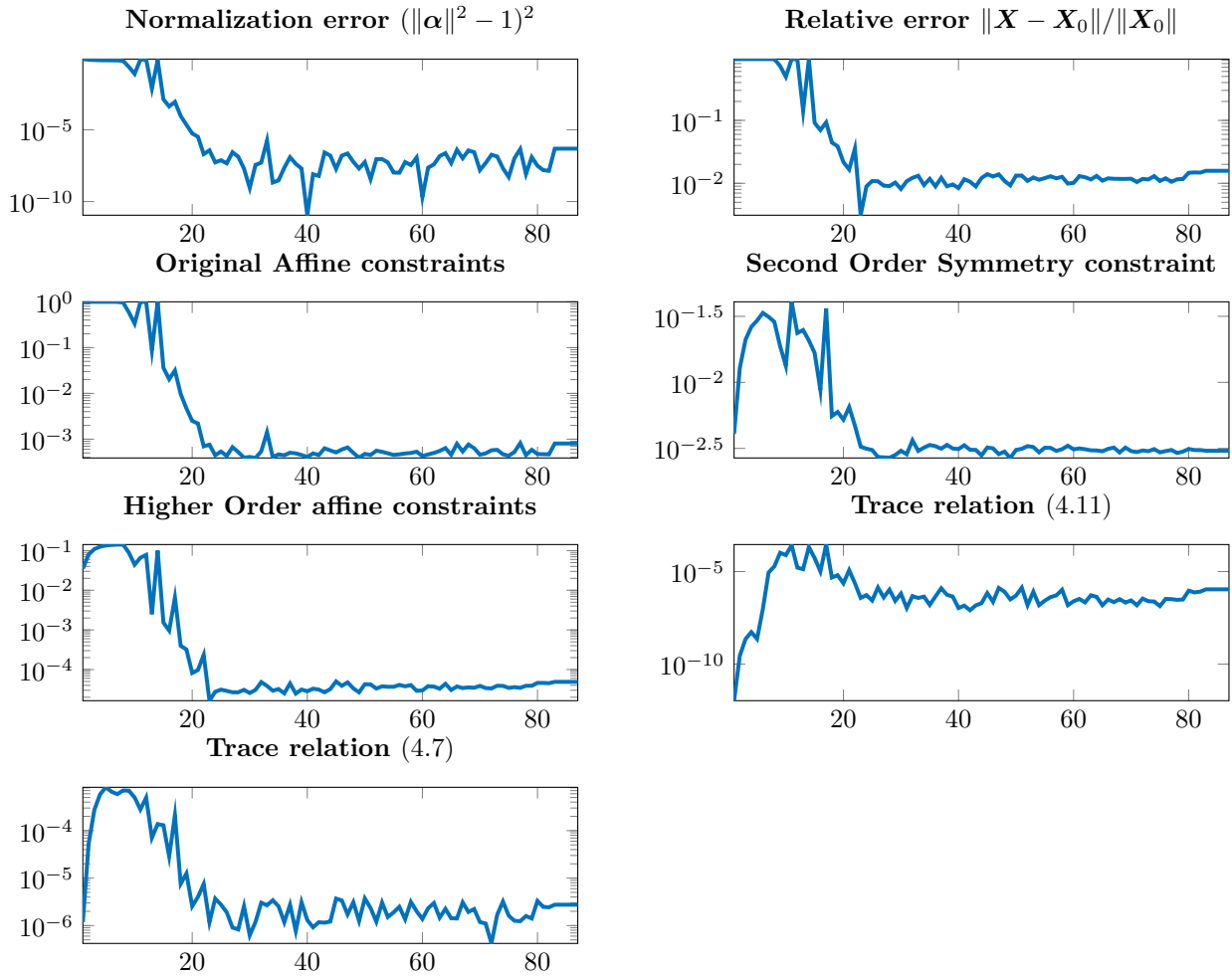}
\caption{\label{LasserreSucceedsWhereTikhoFails}Evolution of the various constraints and error terms appearing in the augmented Lagrangian, for the hierarchical low rank factorization (here $r_1$ and $r_2$ are set to $4$) with Trace contractions~\eqref{traceRelation2} and~\eqref{traceRelation3} for the $20\times 20$ example of Fig.~\ref{TikhoFailure1} with measurements defined from the bipartite graph of Fig~\ref{bipartiteGraphTikhonovFail}.}
\end{figure}

\begin{figure}\centering
\includegraphics[width=.7\linewidth, trim=0cm 3cm 0cm 3cm,clip=true]{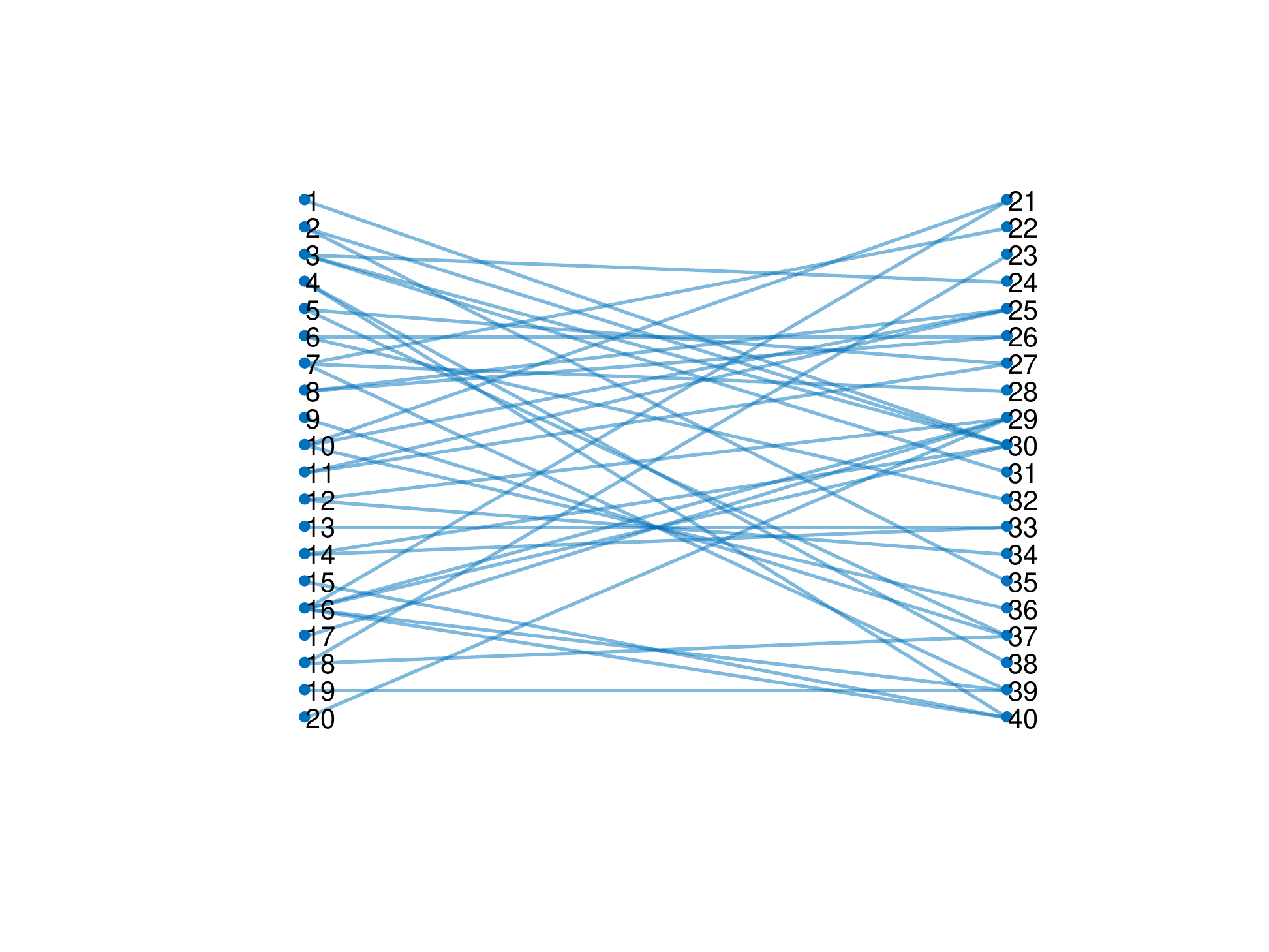}
\caption{\label{bipartiteGraphTikhonovFail}Bipartite graph corresponding to the $20\times 20$ example of Figs~\ref{TikhoFailure1} and~\ref{LasserreSucceedsWhereTikhoFails} used to illustrate the failure of nuclear norm minimization and ridge regression.}
\end{figure}

%
%
%


\begin{figure}
\input{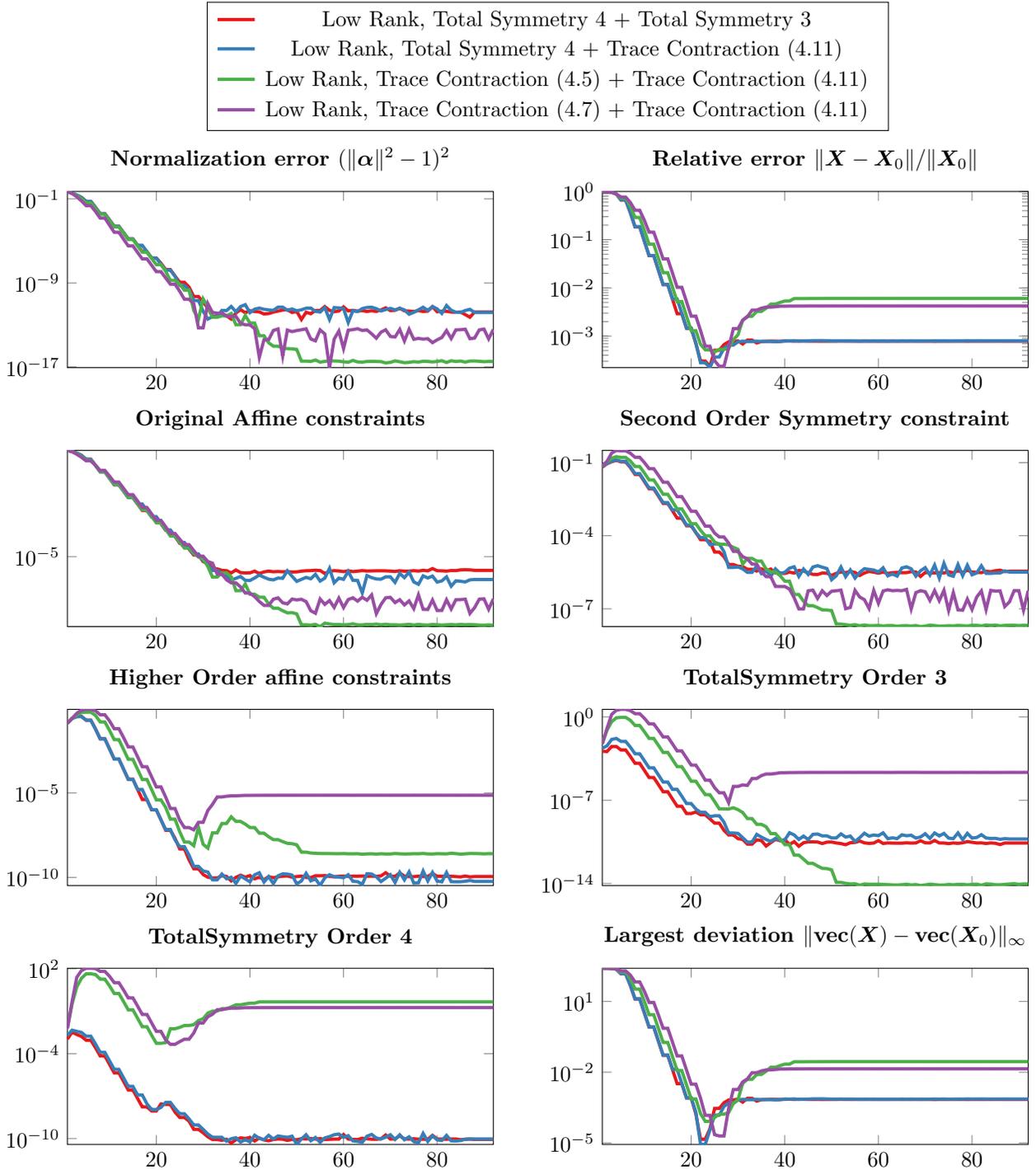}
\caption{\label{traceComparisonTotalLRa}Comparison of the various Trace contractions~\eqref{traceRelation1},~\eqref{traceRelation2} and~\eqref{traceRelation3} on the simple low rank decomposition, as substitute for the combinatorial total symmetry constraints. As shown by the evolution of the relative error, all $4$ approaches lead to estimates that are very close to the optimal solution $M_0$. Adding some of the $4^{th}$ order Total symmetry might help improving the estimate although the solution returned by the $4^{th}$ order trace contraction already are sufficiently close to the ground truth to recover the solution.
}
\end{figure}

\begin{figure}\centering
\input{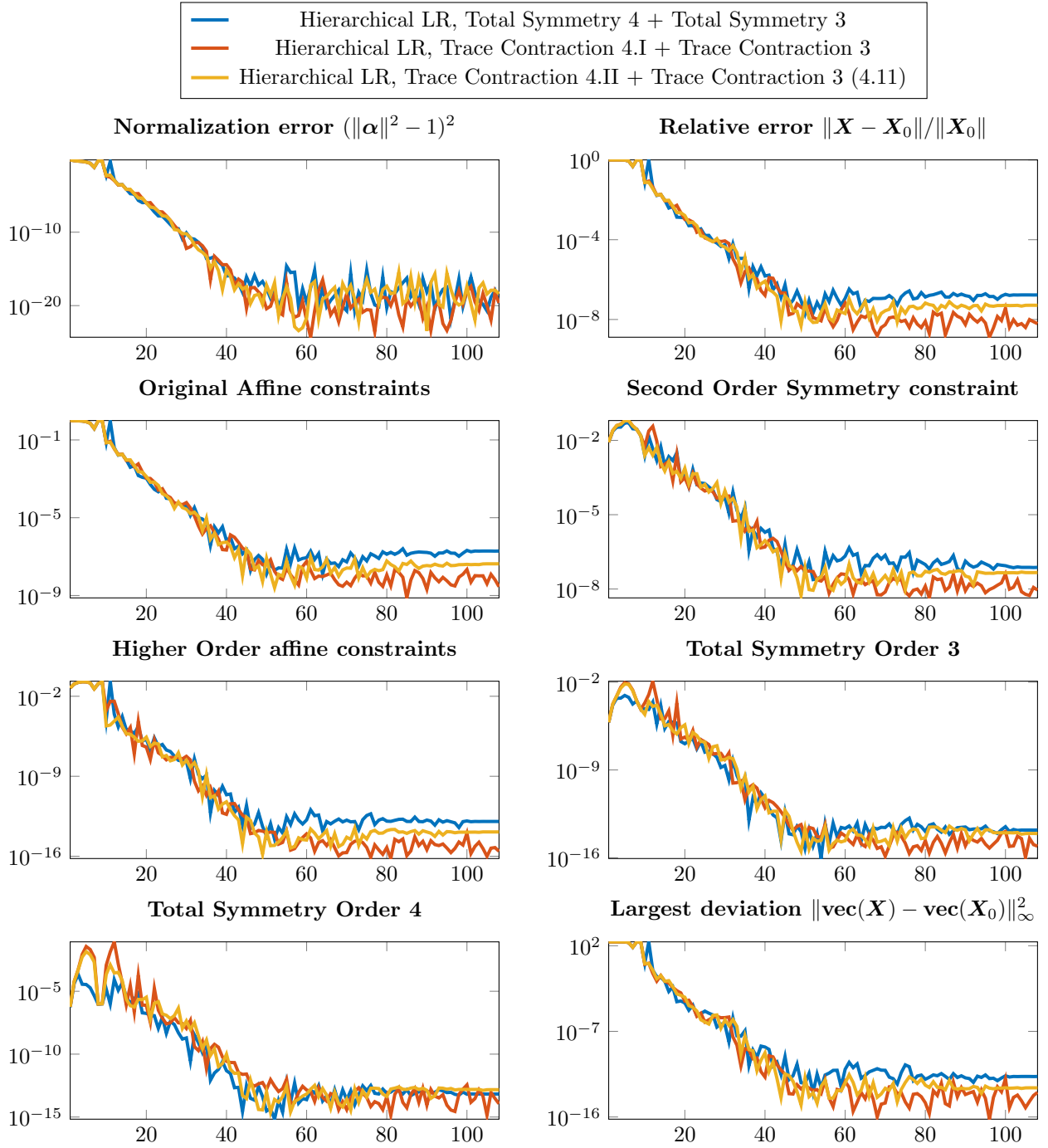}
\caption{\label{hierarchicalLowRankandTraceA}Minimization of the augmented Lagrangian, on the hierarchical low rank factorization (here we consider ranks $4$ and $2$) for various combinations of the total symmetry constraints and trace relations. Note that total symmetry and trace contractions are never imposed simultaneously, i.e., either we impose total symmetry or the corresponding trace contraction. The evolution of the total symmetry constraints shows that on the hierarchical factorization, enforcing the trace relations of section~\ref{traceContraction} is sometimes exactly equivalent to requiring total symmetry by means of the combinatorial constraints.}
\end{figure}

\subsection{\label{subsamplingAffCons}Subsampling the higher order affine constraints}

A last computational bottleneck that hinders the application of the hierarchical formulation of section~\eqref{hierarchicalLR} to larger matrices comes from the higher-order affine constraints. Those constraints have the form $\bs A\bs T\bs T^* = 0$ where $\bs T$ denotes the whole low rank factor of size $\mathcal{O}(N^2)$ and $\bs A$ simply applies the affine constraints to the second order part of $\bs T$, columnwise. To further reduce the computational cost, we propose to draw smaller $\mathcal{O}(N)$ "batches" of moments $\mathcal{S}_i\subseteq [N]\times [N]$ from the full set of second order moments. We then minimize the resulting reduced augmented Lagrangian functions defined from each $\mathcal{S}_i$ sequentially. Let $\bs T_0$ denote the intial iterate chosen at random. The procedure can be summarized as follows

\begin{enumerate}
\item Randomly select a subset $\mathcal{S}_i$ of size $\mathcal{O}(N)$, without replacement, from the set of all second order moments $[N]\times[N]$.
\item Minimize the augmented Lagrangian~\eqref{augmentedLagrangian}, considering only the higher order affine constraints of the form $\bs A\bs T \bs T_{\mathcal{S}_i}^* = 0$, where $\bs T_{\mathcal{S}_i} \equiv [\bs \alpha^*, \bs R^*, \bs \Pi_{\mathcal{S}_i}^*]$ resulting from the second order moments appearing in $\mathcal{S}_i$.
\item Let $\bs T^\sharp$ denote the solution resulting from step 2. If $\|\mathcal{A}(\bs T^{\sharp}(\bs T^{\sharp})^*) - \bs b\|$ is sufficiently small then stop. Otherwise, repeat step 1 with $\bs T_0 \leftarrow \bs T^\sharp$.
\end{enumerate}

To illustrate this last algorithm, we provide numerical experiments on a $100\times 100$ matrix. For the algorithm to be fully efficient, we combine the trace relations~\eqref{traceRelation2} (on the fourth order block) and~\eqref{traceRelation3} (on the third order block), and take advantage of the subsampling scheme discussed above. On a $100\times 100$ matrix, the factorized gradient method would require storing matrices of size at least $\mathcal{O}(N^3) = 8e6$, thus leading to poor performace in terms of runtime. On convex solvers such as {\sc cvx} or {\sc glotipoly}, this example would require storing matrices of size $\mathcal{O}(N^4)$ ($\mathcal{O}(N^3)$ in the reduced formulation of corollary~\eqref{noisycaseCorollary}). The iterations on this example are displayed in Fig.~\ref{hugeExample2} and the corresponding bipartite graph used as mask is shown in Fig.~\ref{hugeExampleBipartite}. Solving this problem takes no more than $10$ mins on a laptop with $2$ GHz Intel Core i5. 

\begin{figure}\centering
\input{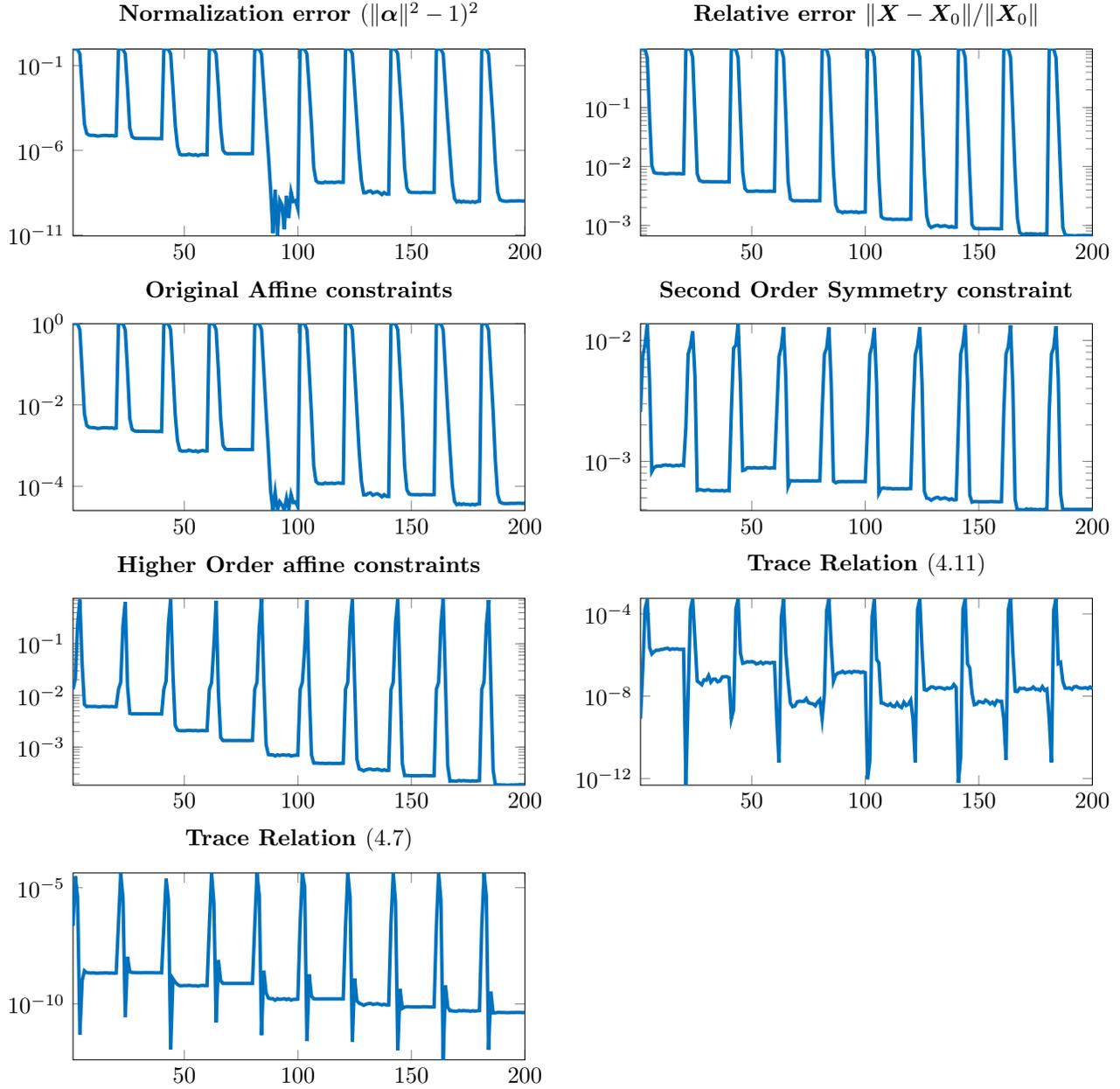}
\caption{\label{hugeExample2} When considering sufficiently large (e.g. $100\times 100$) completion problems, minimizing the whole set of higher order affine constraints is not efficient anymore because those constraints require storing matrices of size $\mathcal{O}(N^3)$ while the hierarchical low rank decomposition only requires storing matrices of size $\mathcal{O}(N)$. For this reason, we divide the vector of second order moments into smaller batches $\mathcal{S}_i$ of size $\mathcal{O}(m+n)$ sampled at random, and minimize the resulting augmented Lagrangians sequentially as explained in section~\ref{subsamplingAffCons}. Each jump in the figures above corresponds to a resampling of the moments. For each resampling of the moments, we reset the relative weight of the trace with respect to the misfit, whence the jump occuring at the transition between two batches. Here the moments of $\bs X_0$ are controlled as $\delta \leq (X_0)_{ij}\leq 1$ with $\delta =0.25$ and the ranks of the hierarchical factorization are set as $r_1=r_2=2$. For large matrices, only the trace relations~\eqref{traceRelation3} and~\eqref{traceRelation2} are represented as the total symmetry constraints are too expensive to compute and are thus not enforced. }
\end{figure}

\begin{figure}\centering
\includegraphics[width=.7\linewidth, trim=0cm 3cm 0cm 3cm,clip=true]{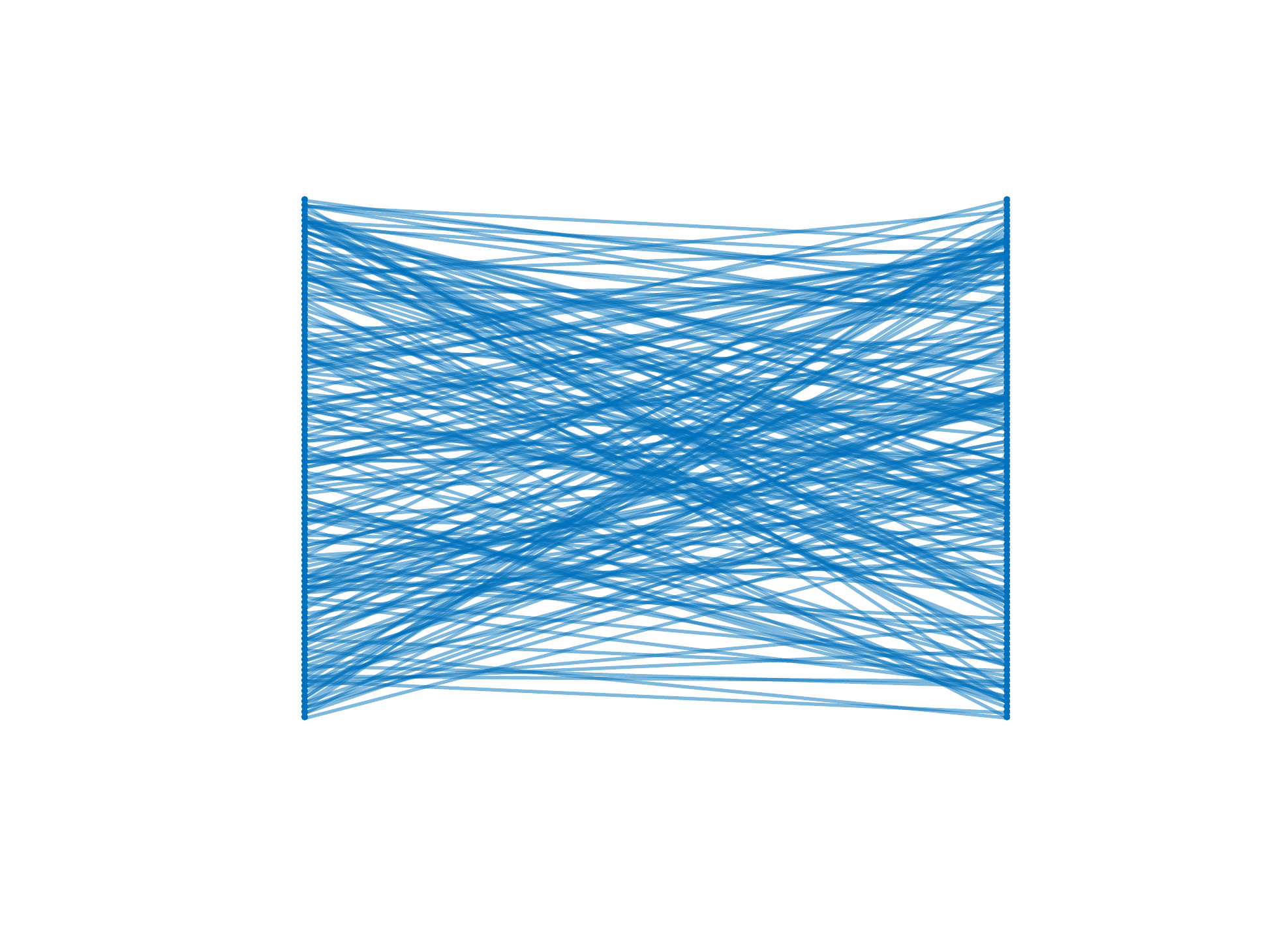}
\caption{\label{hugeExampleBipartite} Bipartite graph corresponding to the $100\times 100$ example of Fig.~\ref{hugeExample2}. The number of edges is $\mathcal{O}(N)\approx 200$. The graph is generated at random while required to span the whole set of vertices minimally.}
\end{figure}

\section{\label{r1tc}Stable completion of rank-one tensors}
 


Theorem~\ref{maintheorem} and~\ref{noisycase} both have a direct extension to the tensor completion problem. Given a rank one $d$-tensor $T\in \mathbb{R}^{n^{d}}$, one can always write $T = x_1\otimes x_2\otimes \ldots\otimes x_d$ with $(x_2)_1 = (x_3)_1 = \ldots = (x_{d-1})_1 = 1$. To this tensor, one can associate a $d$-uniform hypergraph $\mathcal{H}(\mathcal{V},\mathcal{E})$ whose set of vertices is given by the set of indices associated to each dimension and whose edges are defined from the measurements $\mathcal{P}_\Omega(T)$. For this hypergraph, we consider the following property.

\begin{definition}[Definition 1 in~\cite{berke2009propagation}]
Let $\mathcal{H}(\mathcal{V},\mathcal{E})$ be a $d$-uniform hypergraph on $N_1\times N_2\times\ldots N_d=|\mathcal{V}|$ vertices. A sequence $E_1,\ldots, E_{K}\in \mathcal{E}$ of hyperedges is called a propagation sequence if for any $1\leq \ell\leq K$, $|E_{\ell+1}\bigcap_{i=1}^\ell E_\ell| = d-1$. If the hypergraph has a propagation sequence, then it is called propagation connected.
\end{definition}

Stable deterministic completion of rank-one, propagation-connected tensors directly follows from the definition of propagation connectivity and the normalization $(x_2)_{1} = \ldots = (x_{d-1})_{1} $ of the vectors in the decomposition. Indeed, as in section~\ref{SectionIntroduction}, let $\z= (\bs x_1,\bs x_2,\ldots, \bs x_{d-1})$ denote the concatenation of the monomials arising from the tensor decomposition. For every monomial $z_L$ there always exists a sequence of hyperedges, each defined from its corresponding set of indices in $[N_1]\times[N_2]\times\ldots\times [N_d]$, such that $E_1 = (1,1,\ldots, i_1)$, $E_2 = (1,\ldots,1, i_1, i_2)$, $\ldots$ and $|E_i\cap E_{i+1} | = d-1$. Let us denote by $z_{i_1},z_{i_2}, \ldots, z_{L}$ the variables defined from the chain as $i_1 \in E_{1}$ and $i_k\in E_{k-1}\cap E_{k}$ for $k=2,\ldots, L$ with $i_L = L$. Assume that we can express the polynomials $z_{i_k}-(z_0)_{i_k}$ for $k=1,\ldots L-1$, then the $k^{th}$ canonical polynomial $z_{L} - (z_0)_L$ can be expressed from $z_{i_{L-1}}$ and the constraints corresponding to the edges $E_L$ and $E_{L-1}$ as 
\begin{align*}
(z_L - (z_0)_L)\prod_{i\in E_{L-1}} (z_0)_i  \; = & \; z_{E_{L-1}\setminus (E_L\cap E_{L-1})}\left(\prod_{i\in E_L} z_i -  \prod_{i\in E_L} (z_0)_i\right)- z_{L}\left(\prod_{i\in E_{L-1}} z_i - \prod_{i\in E_{L-1}} (z_0)_i \right)\\
&+   z_{E_{L-1}\setminus (E_L\cap E_{L-1})}\prod_{i\in E_L} (z_0)_i
\end{align*}
This discussion naturally leads to the following corollary (a corresponding stability result can be derived).
\begin{corollary}
Let $T\in \mathbb{R}^{N_1\times\ldots\times N_d}$ denote an order $d$ rank-one tensor. Assume that we are given the entries $T_{i_1,\ldots,i_d}$ for $(i_1,i_2,\ldots,i_d)\in \Omega$. Further assume that the hypergraph corresponding to $\Omega$ is propagation connected. Then the tensor can be efficiently completed through $d$ rounds of semidefinite programming relaxation with minimization of the trace norm of the moments matrix.   
\end{corollary}	

\appendix
\section{\label{gradientFunctionCalculation}Hierarchical low rank gradient and function}

Before introducing the compressed function and gradient resulting from the multi-level low rank encoding of the moments matrix, recall that we normalize the first entry of $\bs x$ and work with the matrix 
\begin{align}
\bs X &= \bs y\bs (1,\bs x^T) = \left(\begin{array}{cccc}
y_1 & x_1 y_1 & \ldots & x_n y_1\\
y_2 & x_1 y_2 & \ldots &x_{n}y_2\\
\vdots & &  &\\
y_m & x_1 y_m & \ldots & y_m x_n
\end{array}\right)\label{normalization1}
\end{align}
The trace can be computed efficiently as 
\begin{align*}
\mbox{Trace}(\bm M) &= \|\bs \alpha\|^2 + \|\bs R\|^2_F + \sum_{r=1}^{R_1} \langle \Pi_r,\Pi_r\rangle \\
& = \|\bs \alpha\|^2 + \|\bs R\|^2_F + \sum_{r=1}^{R_1} \langle \sum_{r'=1}^{R_2} S_{r,r'}S_{r,r'}^T,\sum_{r'=1}^{R_2} S_{r,r'}S_{r,r'}^T\rangle\\
& = \|\bs \alpha\|^2 + \|\bs R\|^2_F + \sum_{r=1}^{R_1} \sum_{k=1}^{R_2}\sum_{\ell=1}^{R_2}  |\langle S_{r,k},S_{r,\ell}\rangle|^2
\end{align*}
The gradient for the trace can be computed efficiently as 
$$\partial_\alpha = 2\alpha,\quad \partial_R = 2R\quad \mbox{and}\quad \partial_{S_{r,k}} = 4 \sum_{r'=1}^{R_2} S_{r,r'}S^T_{r,r'}S_{r,k}.$$
There are two sets of structural constraints. The first set enforces the equality between $RR^T$ and $\sum_{r=1}^{R_1} \alpha_r \Pi_r =  \sum_{r=1}^{R_1} \alpha_r \sum_{k} S_{r,k}S_{r,k}^T$. This first set can be expressed compactly as the following Frobenius contribution
\begin{align*}
\|RR^T - \sum_{r=1}^{R_1}\alpha_r(\sum_{k=1}^{R_2} S_{r,k}S_{r,k}^T)\|_F^2 &= \|R^TR\|_F^2+ \sum_{r= 1}^{R_1}\sum_{r'=1}^{R_1}\alpha_r\alpha_{r'} \left(\sum_{k=1}^{R_2} \sum_{k'=1}^{R_2}
|\langle S_{r,k}, S_{r',k'}\rangle|^2\right)\\
& -2 \sum_{r=1}^{R_1} \alpha_r \left(\sum_{k=1}^{R_2}\|S_{r,k}^TR\|_F^2\right).
\end{align*}
The resulting gradient contribution can be computed efficiently as,
\begin{align*}
\partial_R & = 2 \left[(RR^T - \sum_{r=1}^{R_1} \alpha_r(\sum_{k=1}^{R_2} S_{r,k}S_{r,k}^T))R+ R^T(RR^T - \sum_{r=1}^{R_1} \alpha_r(\sum_{k=1}^{R_2} S_{r,k}S_{r,k}^T))\right]\\
 \partial_{S_{r,k}} &= - 2 \alpha_r \left(RR^T - \sum_{r=1}^{R_1}\alpha_r \left(\sum_{k=1}^{R_2} S_{r,k}S_{r,k}^T\right)\right) S_{r,k} - 2 \alpha_r S_{r,k}^T \left(RR^T - \sum_{r=1}^{R_1}\alpha_r \left(\sum_{k=1}^{R_2} S_{r,k}S_{r,k}^T\right)\right) \\
 \partial_{\alpha_r} & = -2 \left\langle RR^T - \sum_{r=1}^{R_1}\alpha_r\left(\sum_{k=1}^{R_2}S_{r,k}S_{r,k}^T\right), \sum_{k=1}^{R_2} S_{r,k}S_{r,k}^T\right\rangle
\end{align*}
The second set of structural constraints enforces equality between corresponding third and fourth order monomials. As an example, we have $\Pi_{i,j}R_k^T = \Pi_{i,k}R_j^T$ or similarly $\Pi_{i,j}R_k^T = \Pi_{i,k}R_j^T$. Those constraints are first expressed through permutations of third and fourth indices. When dealing with third and fourth order monomials equivalences, only some of the permutation have to be explicitly enforced. The others are naturally encoded through the positive semidefinite constraint. The explicit ones are listed below. For third order monomials we have
\begin{align}
(z_iz_j)z_k & =  (z_iz_k)z_j\label{Order3struct1}\\
& = (z_jz_k)z_i\label{Order3struct2}
\end{align}
Equivalently, for fourth order monomials, we can only retain the following relations,
\begin{align}
(z_iz_j)(z_k z_\ell) & =  (z_iz_k)(z_jz_\ell)\label{TotalSym41}\\ 
& = (z_iz_\ell)(z_jz_k)\label{TotalSym42}  
\end{align}
Let each of the permutations for third and fourth order monomials that encode the structural constraints be denoted as $\pi^{(3)}$ and $\pi^{(4)}$ where we let $\pi^{(3)}\;:\;(m+n-1)^3\;\mapsto\; (m+n-1)^3$ with $\pi^{(3)}(i_1,i_2,i_3) = (\pi^{(3)}_1, \pi^{(3)}_2, \pi^{(3)}_3)$ and similarly for $\pi^{(4)}$. The resulting constraints in the framework of formulation~\eqref{LowrankEncoding} can read as 
\begin{align}
\min \quad & \left\|\sum_{\ell=1}^{R_1}  \Pi_\ell (i_1,i_2)R(i_3,\ell) - \sum_{\ell=1}^{R_1}\Pi_\ell (\pi^{(3)}_1,\pi^{(3)}_2)R(\pi^{(3)}_3,\ell)\right\|^2\label{ThirdOrderPerm}
\end{align}
Fourth order relations can be expressed in exactly the same way,
\begin{align}
\min \quad & \left\|\sum_{\ell=1}^{R_1} \Pi_\ell (i_1,i_2)\Pi_\ell(i_3,i_4) - \sum_{\ell=1}^{R_1}\Pi_\ell (\pi^{(4)}_1,\pi^{(4)}_2)\Pi_\ell(\pi^{(4)}_3,\pi^{(4)}_4)\right\|^2\label{FourthOrderPerm}
\end{align}
If we expand the second order low rank factorizations, the two structural contributions~\eqref{ThirdOrderPerm} and~\eqref{FourthOrderPerm} respectively read as 
\begin{align}
&\left\|\sum_{\ell=1}^{R_1} \left(\sum_{k=1}^{R_2} S_{\ell, k}[i_1]S_{\ell, k}[i_2]\right)R[i_3,\ell]- \sum_{\ell=1}^{R_1}\left(\sum_{k=1}^{R_2} S_{\ell, k}[\pi^{(3)}_1]S_{\ell, k}[\pi^{(3)}_2]\right)R[\pi^{(3)}_3,\ell]\right\|^2_F
\end{align}
as well as 
\begin{align}
\min \quad & \left\|\sum_{\ell=1}^{R_1} \left(\sum_{k=1}^{R_2} S_{\ell, k}[i_1] S_{\ell, k}[i_2]\right) \left(\sum_{k'=1}^{R_2} S_{\ell, k'}[i_3] S_{\ell, k'}[i_4]\right)\right.\nonumber \\ 
&\left.- \sum_{\ell=1}^{R_1}  \left(\sum_{k=1}^{R_2} S_{\ell, k}[\pi^{(4)}_1]S_{\ell, k}[\pi^{(4)}_2]\right)\left(\sum_{k'=1}^{R_2} S_{\ell, k'}[\pi^{(4)}_3]S_{\ell, k'}[\pi^{(4)}_4]\right)\right\|^2
\end{align}
This equivalence of monomials is the most expensive step in the minimization. We let $E^3$ and $E^4$ be defined as
\begin{align}
E_3 & = \sum_{\ell=1}^{R_1}  \Pi_\ell (i_1,i_2)R(i_3,\ell) - \sum_{\ell=1}^{R_1}\Pi_\ell (\pi^{(3)}_1,\pi^{(3)}_2)R(\pi^{(3)}_3,\ell)\\
E_4 & = \sum_{\ell=1}^{R_1} \Pi_\ell (i_1,i_2)\Pi_\ell(i_3,i_4) - \sum_{\ell=1}^{R_1}\Pi_\ell (\pi^{(4)}_1,\pi^{(4)}_2)\Pi_\ell(\pi^{(4)}_3,\pi^{(4)}_4)
\end{align}
for which the contributions of~\eqref{ThirdOrderPerm} and~\eqref{FourthOrderPerm} to the gradient accumulate for each $\ell$ and pairs of 3-tuple of indices $\{(i_1,i_2,i_3), (\pi_1^{(3)},\pi_2^{(3)},\pi_3^{(3)} )\}$ appearing in the set of structural constraints as 
\begin{alignat}{2}
&\partial_R [i_3,\ell] \leftarrow \partial_R [i_3,\ell] +\Pi_{\ell}(i_1,i_2)E_3,&&\quad \partial_R [\pi^{(3)}_3,\ell] \leftarrow \partial_R [i_3,\ell] +\Pi_\ell (\pi^{(3)}_1,\pi^{(3)}_2) E_3\\
&\partial_{\Pi_\ell} [i_1,i_2] \leftarrow \partial_{\Pi_\ell} [i_1,i_2]+ R[i_3,\ell]E_3,&&\quad \partial_{\Pi_\ell} (\pi^{(3)}_1,\pi^{(3)}_2)\leftarrow \partial_{\Pi_\ell} (\pi^{(3)}_1,\pi^{(3)}_2)+R(\pi^{(3)}_3,\ell)E_3
\end{alignat}
Accordingly, for the fourth order contribution, we simply accumulate the contributions arising from each of the norms in~\eqref{FourthOrderPerm} for each pair of four-tuples $\{(i_1,i_2,i_3,i_4), (\pi^{(4)}_1,\pi^{(4)}_2,\pi^{(4)}_3,\pi^{(4)}_4)\}$ and each rank index $\ell$ as
\begin{alignat}{2}
&\partial_{\Pi_\ell} [i_1,i_2] \leftarrow \partial_{\Pi_\ell} [i_1,i_2]+\Pi_{\ell}(i_3,i_4)E_4,&&\quad \partial_{\Pi_\ell} [\pi^{(4)}_1,\pi^{(4)}_2] \leftarrow \partial_{\Pi_\ell} [\pi^{(4)}_1,\pi^{(4)}_2]-\Pi_\ell (\pi^{(3)}_1,\pi^{(3)}_2) E_4\\
&\partial_{\Pi_\ell} [i_3,i_4] \leftarrow \partial_{\Pi_\ell} [i_3,i_4] + \Pi_{\ell}(i_1,i_2)E_4,&&\quad \partial_{\Pi_\ell} [\pi^{(4)}_3,\pi^{(4)}_4] \leftarrow \partial_{\Pi_\ell} [\pi^{(4)}_3,\pi^{(4)}_4]-\Pi_\ell (\pi^{(3)}_1,\pi^{(3)}_2) E_4
\end{alignat}
The accumulations on the low rank factors $\Pi_\ell$, $\ell=1,\ldots R_1$ expand as accumulations on each of their low rank factorizations $\partial_{\Pi_\ell [i,j]} = \partial_{S_\ell[i]S\ell[j]}$ from which we get each of the separate partials using the chain rule as $\partial_{S_{\ell}[i]} = \partial_{\Pi_\ell [i,j]} S_\ell[j]$ and equivalently $\partial_{S_{\ell}[j]} = \partial_{\Pi_\ell [i,j]} S_\ell[i]$.

Following the normalization~\eqref{normalization1}, we can now express the original constraints together with their higher order extensions. Those sets of constraints read as follows. We first decompose the map $A$ into the component $A_0\in \mathbb{R}^{N\times m}$ acting on the first column and the remaining part $A_1\in \mathbb{R}^{N\times n-1}$ acting on the matrix $RR^T$. Each of the constraint are encoded by means of appropriate matrices $A_{ij} = (A_0)_{ij} + (A_1)_{ij}$ as 
\begin{align}
\|\mathcal{A}(M) - b\|_2^2 = \sum_{k=1}^{|\Omega|}|\langle (A_0)_{k}, R_x\alpha^T\rangle +  \langle (A_1)_{k}, R_xR_y^T\rangle  b_{k}|^2\label{affineConstraints1}
\end{align}
We will use corresponding linear maps $\mathcal{A}_0$ and $\mathcal{A}_1$ to encode the matrix constraints efficiently. We thus have $\mathcal{A}_0(x) = A_0x $ and $\{\mathcal{A}_1(X)\}_k = \langle (A_1)_k, X\rangle$. Then for any vector $b\in \mathbb{R}^{N}$ $\mathcal{A}_0^*b = A_0^*b = \sum_{k=1}^N (a_0)_k b_k$ where $(a_0)_k$ denotes the transpose of the $k^{th}$ row of $A_0$ and $\mathcal{A}_1^* b = \sum_{k=1}^N (A_1)_k b_k$. The gradient for those constraints reads
\begin{align*}
\partial_\alpha &= 2R_x^T\mathcal{A}_0^*(\mathcal{A}(M)- b) - 4\alpha(b^T(\mathcal{A}(M)- b) )\\
\partial_{R_x}  & = 2\mathcal{A}_0^*(\mathcal{A}(M)- b) \alpha^T + 2\tilde{\mathcal{A}}^*(\mathcal{A}(M)- b) R_y\\
\partial_{R_y}  & = R_x^*\tilde{\mathcal{A}}^*(\mathcal{A}(M)- b) 
\end{align*}
We call higher order affine constraints the constraints derived from multiplying any  of the constraints in $\{(\mathcal{A}(X)-b)_j\}_{j=1}^J = \{h_j(x,y)\}_{j=1}^{J}$ by any of the monomials of degree at most $t-\mbox{deg}(h_j)$ for a relaxation of order $t$ (a.k.a the $t^{th}$ round of the hierarchy). Those higher order constraints can be encoded simply by multiplying the matrix of (pseudo)-moments by each of the constraints vectors of coefficients~\cite{laurent}. Since the moments matrix is low rank, for any vector $a\in \mathbb{R}^{|\mathbb{N}_n^2|}$ this product reads very simply as $a^T M(m) = \langle a, [\alpha^T,\;R^T,\; \Pi^T]\rangle [\alpha^T,\;R^T,\; \Pi^T]^T = 0$. For a general set of affine constraints, the decomposition of $X$ introduced in~\eqref{normalization1} and the decomposition of $\mathcal{A}$ used in~\eqref{affineConstraints1}, stable minimization of the higher order affine constraints reads
\begin{align*}
\min\quad &\left\|\sum_{\ell=1}^{R_1}(\mathcal{A}_0R^x_{\ell} + \mathcal{A}(\Pi^{xy}_\ell) - b\alpha_\ell )(\alpha^T,\;R^T,\;\Pi^T)_\ell \right\|_F^2\label{HigherOrderCons1} \\
& = \left\langle \left(\mathcal{A}_0(R_{k}^x) + \mathcal{A}\left(\sum_{r=1}^{R_2}S^x_{k,\ell}(S^y_{k,\ell})^T\right) - b \alpha_k\right)_{\ell=1}^{R_1} , \left(\mathcal{A}_0(R_{k}^x) + \mathcal{A}\left(\sum_{r=1}^{R_2}S^x_{k,\ell}(S^y_{k,\ell})^T\right) - b \alpha_k\right)_{\ell=1}^{R_1}T^T T\right\rangle \\
& = \langle \Delta^T \Delta, T^T T\rangle\\
& = \sum_{\ell=1}^{R_1}\sum_{\ell'=1}^{R_1} (\Delta^T\Delta)[\ell,\ell']\left(\langle \sum_{k=1}^{R_2}S_{\ell,k}S_{\ell,k}, \sum_{k'=1}^{R_2}S_{\ell',k'}S_{\ell',k'}\rangle\right)\\
\end{align*}
Here we let $S^x_{k,\ell}, S^y_{k,\ell}$ denote the first and second blocks of the low rank factors $S_{k\ell}\in \mathbb{R}^{m+n-1\times R_2}\}$ of each matrix $\Pi_\ell$. $\Delta^T\Delta \in \mathbb{R}^{R_1\times R_1}$. We thus have $S^x_{\ell,k}\in \mathbb{R}^{m\times R_1}$, $S^y_{\ell,k}\in \mathbb{R}^{m\times R_1}$. The definition of $\Pi^{xy}$ follows from those ideas,
$$\Pi^{xy}_\ell = \sum_{k=1}^{R_2} S_{\ell,k}^x(S^y_{\ell,k})^T.$$
Finally, $R_x$ derives from the decomposition of $R^{(m+n-1)\times R_1}$ into $R_x\in \mathbb{R}^{m\times R_1}$ and $R_y^{(n-1)\times R_1}$ with $R = [R_x^T,\; R^T_{y}]^T$. The  contributions to the gradient are given by deriving each side and noting that
\begin{align}
\langle T\Delta^T\Delta, T\rangle &= \langle \alpha\Delta^T\Delta, \alpha\rangle + \langle R\Delta^T\Delta, R\rangle + \sum_{\ell=1}^{R_1} (\Pi\Delta^T\Delta)_\ell \Pi_\ell\\
&  = \sum_{\ell=1}^{R_1} \left\langle \mbox{Mat}(\Pi\Delta^T\Delta)_\ell ,\left(\sum_{k=1}^{R_2} S_{\ell,k}S_{\ell,k}^T\right)\right\rangle
\end{align}
From which we have $\partial_{S_{\ell,k}} =2 \mbox{Mat}\left((\Pi\Delta^T\Delta)_\ell\right)S_{\ell,k} + 2 \left((S_{\ell,k})^T\mbox{Mat}\left((\Pi\Delta^T\Delta)_\ell\right)\right)^T$,
$\partial_R=2R\Delta^T\Delta $ and $\partial_\alpha= 2\alpha\Delta^T\Delta $. For the $\Delta^T\Delta$ term, a similar approach yields $\partial R_x = 2\mathcal{A}_0^*\Delta T^TT$, $\partial_\alpha = -2b^T \Delta T^TT$. For the partials with respect to $S_{\ell,k}$ we use
\begin{align}
\langle\Pi^{xy},\mathcal{A}^*\Delta T^TT\rangle &=\sum_{\ell=1}^{R_1}\left\langle \sum_{k=1}^{R_2} S^x_{\ell,k}\left(S_{\ell,k}^y\right)^T,\left(\mathcal{A}^*\Delta T^TT\right)_\ell\right\rangle
\end{align}
From which we can derive $\partial S_{\ell,k}^x =2\mbox{Mat}\left(\mathcal{A}^*\Delta T^TT\right)_\ell S_{\ell,k}^y $ as well as $\partial S_{\ell,k}^y =2\left(S_{\ell,k}^x\right)^T\mbox{Mat}\left(\mathcal{A}^*\Delta T^TT\right)_\ell $. In the expressions above, we use $\mbox{Mat}(\bs x)$ to denote the usual vector to matrix operator that turns the vector $\bs x= [x_1^T,\ldots x_n^T]^T$ into the matrix $M = [x_1,\ldots x_n]$. Note that the product $\Pi\Delta^T\Delta$ can be computed efficiently as 
\begin{align}
(\Pi\Delta^T\Delta)_{\bullet, \ell} &= \sum_{\ell'= 1}^{R_1} \left(\sum_{k=1}^{R_2}S_{\ell',k}S_{\ell',k}^T\right)(\Delta^T\Delta)_{\ell',\ell}
\end{align}
From which, the expression of the partials follow as,
\begin{align}
2 \mbox{Mat}\left((\Pi\Delta^T\Delta)_\ell\right)S_{\ell,k} & = 2\sum_{\ell'= 1}^{R_1} \left(\sum_{k'=1}^{R_2}S_{\ell',k'}\langle S_{\ell',k'}^T S_{\ell,k}\rangle\right)(\Delta^T\Delta)_{\ell',\ell}
\end{align}
This expression just computes a projection of the low rank factors onto the subspace generated by each low rank decomposition and can be efficiently carried out by stacking all those low rank factors in a matrix of size $\mathcal{O}(m+n-1\times R_2)$ premultiplying the matrix by the $\ell$ column of $\Delta^T\Delta$ and then applying the projector and summing over the $\ell'$ indices.

\bibliography{biblio.bib}

\begin{thebibliography}{10}

\bibitem{ahmadi2017improving}
A.~A. Ahmadi, G.~Hall, A.~Papachristodoulou, J.~Saunderson, and Y.~Zheng.
\newblock Improving efficiency and scalability of sum of squares optimization:
  Recent advances and limitations.
\newblock {\em arXiv preprint arXiv:1710.01358}, 2017.

\bibitem{bandeira2015convex}
A.~S. Bandeira.
\newblock Convex relaxations for certain inverse problems on graphs.
\newblock 2015.

\bibitem{barak2012hypercontractivity}
B.~Barak, F.~G. Brandao, A.~W. Harrow, J.~Kelner, D.~Steurer, and Y.~Zhou.
\newblock Hypercontractivity, sum-of-squares proofs, and their applications.
\newblock In {\em Proceedings of the forty-fourth annual ACM symposium on
  Theory of computing}, pages 307--326. ACM, 2012.

\bibitem{barak2015tensor}
B.~Barak and A.~Moitra.
\newblock Tensor prediction, {R}ademacher complexity and random 3-xor.
\newblock {\em arXiv preprint arXiv:1501.06521}, 2015.

\bibitem{barak2014sum}
B.~Barak and D.~Steurer.
\newblock Sum-of-squares proofs and the quest toward optimal algorithms.
\newblock {\em arXiv preprint arXiv:1404.5236}, 2014.

\bibitem{bardet2005complexity}
M.~Bardet.
\newblock On the complexity of a gr{\"o}bner basis algorithm.
\newblock In {\em Algorithms Seminar, 2002--2004}, page~85, 2005.

\bibitem{berke2009propagation}
R.~Berke and M.~Onsj{\"o}.
\newblock Propagation connectivity of random hypergraphs.
\newblock In {\em Stochastic Algorithms: Foundations and Applications}, pages
  117--126. Springer, 2009.

\bibitem{burer2003nonlinear}
S.~Burer and R.~Monteiro.
\newblock A nonlinear programming algorithm for solving semidefinite programs
  via low-rank factorization.
\newblock {\em Mathematical Programming}, 95(2):329--357, 2003.

\bibitem{burer}
S.~Burer and R.~D. Monteiro.
\newblock A nonlinear programming algorithm for solving semidefinite programs
  via low-rank factorization.
\newblock {\em Mathematical Programming}, 95(2):329--357, 2003.

\bibitem{burer2005local}
S.~Burer and R.~D. Monteiro.
\newblock Local minima and convergence in low-rank semidefinite programming.
\newblock {\em Mathematical Programming}, 103(3):427--444, 2005.

\bibitem{byrod2009fast}
M.~Byr{\"o}d, K.~Josephson, and K.~{\AA}str{\"o}m.
\newblock Fast and stable polynomial equation solving and its application to
  computer vision.
\newblock {\em International Journal of Computer Vision}, 84(3):237--256, 2009.

\bibitem{candes2015phase}
E.~J. Cand\`es, Y.~C. Eldar, T.~Strohmer, and V.~Voroninski.
\newblock Phase retrieval via matrix completion.
\newblock {\em SIAM review}, 57(2):225--251, 2015.

\bibitem{candes2014towards}
E.~J. Cand\`es and C.~Fernandez-Granda.
\newblock Towards a mathematical theory of super-resolution.
\newblock {\em Communications on Pure and Applied Mathematics}, 67(6):906--956,
  2014.

\bibitem{candes2010matrix}
E.~J. Cand\`es and Y.~Plan.
\newblock Matrix completion with noise.
\newblock {\em Proceedings of the IEEE}, 98(6):925--936, 2010.

\bibitem{candes10ma}
E.~J. Cand\`es and Y.~Plan.
\newblock Matrix completion with noise.
\newblock {\em Proc. IEEE}, 98(6):925--936, 2010.

\bibitem{candes2009exact}
E.~J. Cand{\`e}s and B.~Recht.
\newblock Exact matrix completion via convex optimization.
\newblock {\em Foundations of Computational mathematics}, 9(6):717--772, 2009.

\bibitem{voroninski1}
E.~J. Cand\`es, T.~Strohmer, and V.~Voroninski.
\newblock Phaselift: Exact and stable signal recovery from magnitude
  measurements via convex programming.
\newblock {\em Communications on Pure and Applied Mathematics},
  66(8):1241--1274, 2013.

\bibitem{candes2010power}
E.~J. Cand\`es and T.~Tao.
\newblock The power of convex relaxation: Near-optimal matrix completion.
\newblock {\em Information Theory, IEEE Transactions on}, 56(5):2053--2080,
  2010.

\bibitem{cosserank}
A.~Cosse and L.~Demanet.
\newblock Rank-one matrix completion is solved by the sum-of-squares relaxation
  of order two.
\newblock In {\em Proceedings of the 6th IEEE International Workshop on
  Computational Advances in Multi-Sensor Adaptive Processing (CAMSAP'15)}.
  IEEE, 2015.

\bibitem{cox1992ideals}
D.~Cox, J.~Little, and D.~O'shea.
\newblock {\em Ideals, varieties, and algorithms}, volume~3.
\newblock Springer, 1992.

\bibitem{cui2014all}
C.-F. Cui, Y.-H. Dai, and J.~Nie.
\newblock All real eigenvalues of symmetric tensors.
\newblock {\em SIAM Journal on Matrix Analysis and Applications},
  35(4):1582--1601, 2014.

\bibitem{demanetjugnon}
L.~Demanet and V.~Jugnon.
\newblock Convex recovery from interferometric measurements.
\newblock {\em arXiv preprint arXiv:1307.6864}, 2013.

\bibitem{fazel02ma}
M.~Fazel.
\newblock {\em Matrix rank minimization with applications}.
\newblock PhD thesis, Stanford University, March 2002.

\bibitem{goemans1995improved}
M.~X. Goemans and D.~P. Williamson.
\newblock Improved approximation algorithms for maximum cut and satisfiability
  problems using semidefinite programming.
\newblock {\em Journal of the ACM (JACM)}, 42(6):1115--1145, 1995.

\bibitem{gouveia2010theta}
J.~Gouveia, P.~A. Parrilo, and R.~R. Thomas.
\newblock Theta bodies for polynomial ideals.
\newblock {\em SIAM Journal on Optimization}, 20(4):2097--2118, 2010.

\bibitem{heldt2009approximate}
D.~Heldt, M.~Kreuzer, S.~Pokutta, and H.~Poulisse.
\newblock Approximate computation of zero-dimensional polynomial ideals.
\newblock {\em Journal of Symbolic Computation}, 44(11):1566--1591, 2009.

\bibitem{demanetjugnon2}
V.~Jugnon, L.~Demanet, et~al.
\newblock Interferometric inversion: a robust approach to linear inverse
  problems.
\newblock In {\em Proceedings of SEG Annual Meeting, Houston}, pages
  5180--5184, 2013.

\bibitem{keshavan2009matrix}
R.~Keshavan, A.~Montanari, and S.~Oh.
\newblock Matrix completion from noisy entries.
\newblock In {\em Advances in Neural Information Processing Systems}, pages
  952--960, 2009.

\bibitem{keshavan2008learning}
R.~H. Keshavan, A.~Montanari, and S.~Oh.
\newblock Learning low rank matrices from o (n) entries.
\newblock In {\em Communication, Control, and Computing, 2008 46th Annual
  Allerton Conference on}, pages 1365--1372. IEEE, 2008.

\bibitem{keshavan2010matrix}
R.~H. Keshavan, A.~Montanari, and S.~Oh.
\newblock Matrix completion from a few entries.
\newblock {\em Information Theory, IEEE Transactions on}, 56(6):2980--2998,
  2010.

\bibitem{khot2002power}
S.~Khot.
\newblock On the power of unique 2-prover 1-round games.
\newblock In {\em Proceedings of the thiry-fourth annual ACM symposium on
  Theory of computing}, pages 767--775. ACM, 2002.

\bibitem{kiraly2013error}
F.~Kiraly and L.~Theran.
\newblock Error-minimizing estimates and universal entry-wise error bounds for
  low-rank matrix completion.
\newblock In {\em Advances in Neural Information Processing Systems}, pages
  2364--2372, 2013.

\bibitem{kiraly2012combinatorial}
F.~Kir{\'a}ly and R.~Tomioka.
\newblock A combinatorial algebraic approach for the identifiability of
  low-rank matrix completion.
\newblock {\em arXiv preprint arXiv:1206.6470}, 2012.

\bibitem{kiraly2015algebraic}
F.~J. Kir{\'a}ly, L.~Theran, and R.~Tomioka.
\newblock The algebraic combinatorial approach for low-rank matrix completion.
\newblock {\em Journal of Machine Learning Research}, 16:1391--1436, 2015.

\bibitem{lasserre}
J.~B. Lasserre.
\newblock Global optimization with polynomials and the problem of moments.
\newblock {\em SIAM Journal on Optimization}, 11(3):796--817, 2001.

\bibitem{lasserre2006convergent}
J.~B. Lasserre.
\newblock Convergent sdp-relaxations in polynomial optimization with sparsity.
\newblock {\em SIAM Journal on Optimization}, 17(3):822--843, 2006.

\bibitem{laurent}
M.~Laurent.
\newblock Sums of squares, moment matrices and optimization over polynomials.
\newblock In {\em Emerging applications of algebraic geometry}, pages 157--270.
  Springer, 2009.

\bibitem{nesterov2000squared}
Y.~Nesterov.
\newblock Squared functional systems and optimization problems.
\newblock In {\em High performance optimization}, pages 405--440. Springer,
  2000.

\bibitem{nie2013exact}
J.~Nie.
\newblock An exact jacobian sdp relaxation for polynomial optimization.
\newblock {\em Mathematical Programming}, 137(1-2):225--255, 2013.

\bibitem{nie2008sparse}
J.~Nie and J.~Demmel.
\newblock Sparse sos relaxations for minimizing functions that are summations
  of small polynomials.
\newblock {\em SIAM Journal on Optimization}, 19(4):1534--1558, 2008.

\bibitem{parrilo}
P.~A. Parrilo.
\newblock {\em Structured semidefinite programs and semialgebraic geometry
  methods in robustness and optimization}.
\newblock PhD thesis, Citeseer, 2000.

\bibitem{parrilo2003semidefinite}
P.~A. Parrilo.
\newblock Semidefinite programming relaxations for semialgebraic problems.
\newblock {\em Mathematical programming}, 96(2):293--320, 2003.

\bibitem{pimentel2015characterization}
D.~L. Pimentel-Alarc{\'o}n, N.~Boston, and R.~D. Nowak.
\newblock A characterization of deterministic sampling patterns for low-rank
  matrix completion.
\newblock {\em arXiv preprint arXiv:1503.02596}, 2015.

\bibitem{pimenteladaptive}
D.~L. Pimentel-Alarc{\'o}n and R.~D. Nowak.
\newblock Adaptive strategy for restricted-sampling noisy low-rank matrix
  completion.

\bibitem{recht2011simpler}
B.~Recht.
\newblock A simpler approach to matrix completion.
\newblock {\em The Journal of Machine Learning Research}, 12:3413--3430, 2011.

\bibitem{recht10gu}
B.~Recht, M.~Fazel, and P.~Parrilo.
\newblock Guaranteed minimum-rank solutions of linear matrix equations via
  nuclear norm minimization.
\newblock {\em SIAM Review}, 52(3):471--501, 2010.

\bibitem{shor1987class}
N.~Shor.
\newblock Class of global minimum bounds of polynomial functions.
\newblock {\em Cybernetics and Systems Analysis}, 23(6):731--734, 1987.

\bibitem{shor1987quadratic}
N.~Z. Shor.
\newblock Quadratic optimization problems.
\newblock {\em Soviet Journal of Computer and Systems Sciences}, 25(6):1--11,
  1987.

\bibitem{shor1988approach}
N.~Z. Shor.
\newblock An approach to obtaining global extremums in polynomial mathematical
  programming problems.
\newblock {\em Cybernetics}, 23(5):695--700, 1988.

\bibitem{singer2010uniqueness}
A.~Singer and M.~Cucuringu.
\newblock Uniqueness of low-rank matrix completion by rigidity theory.
\newblock {\em SIAM Journal on Matrix Analysis and Applications},
  31(4):1621--1641, 2010.

\bibitem{sturmfels2002solving}
B.~Sturmfels.
\newblock {\em Solving systems of polynomial equations}.
\newblock Number~97. American Mathematical Soc., 2002.

\bibitem{tang2015guaranteed}
G.~Tang and P.~Shah.
\newblock Guaranteed tensor decomposition: A moment approach.
\newblock In {\em Proceedings of The 32nd International Conference on Machine
  Learning}, pages 1491--1500, 2015.

\end{thebibliography}
\bibliographystyle{abbrv}


\end{document}